\documentclass[hidelinks =true, final
]{siamart1116}

\usepackage{amsfonts}
\usepackage{graphicx}
\usepackage{epstopdf}

\usepackage{algorithm} 
\usepackage{algorithmic}
\usepackage[algo2e,lined,boxed,commentsnumbered]{algorithm2e}

\newcommand*\Rnorm[1]{\left[\!\left]#1\right[\!\right]}

\newcommand*\cnorm[1]{\left\vert\!\left[\!\left]#1\right[\!\right]\!\right\vert}

\usepackage{xcolor}

\ifpdf
  \DeclareGraphicsExtensions{.eps,.pdf,.png,.jpg}
\else
  \DeclareGraphicsExtensions{.eps}
\fi

\numberwithin{theorem}{section}

\headers{Robust True Solution Certified Primal-Dual RBM}
{Shun Zhang}

\title{Primal-Dual Reduced Basis Methods for Convex Minimization Variational Problems:
Robust True Solution A Posteriori Error Certification and Adaptive Greedy Algorithms}
\author{
  Shun Zhang\thanks{Department of Mathematics, City University of Hong Kong, Kowloon Tong, Hong Kong SAR, China.
     \email{shun.zhang@cityu.edu.hk}.
   This work was supported in part by Hong Kong Research Grants Council under the GRF Grant Project CityU 11305319.}
   The paper is accepted by SIAM Journal on Scientific Computing.
}

\ifpdf
\hypersetup{
  pdftitle={Primal-Dual Certified Reduced Basis Methods},
  pdfauthor={S. Zhang}
}
\fi

\usepackage{graphics}
\usepackage{epsfig}

\usepackage{enumerate}

\renewcommand{\theequation}{\thesection.\arabic{equation}}

\newtheorem{thm}{Theorem}[section]

\newtheorem{prop}[thm]{Proposition}
\newtheorem{defn}[thm]{Definition}
\newtheorem{rem}[thm]{Remark}

\newtheorem{asm}[thm]{Assumption}

\usepackage{bm}

\usepackage{titlesec}
\titleformat{\paragraph}
{\normalfont\normalsize\bfseries}{\theparagraph}{1em}{}
\titlespacing*{\paragraph}
{0pt}{3.25ex plus 1ex minus .2ex}{1.5ex plus .2ex}

\begin{document}
\newcommand{\BX}{{\bf X}}
\newcommand{\cv}{{\cal V}}
\newcommand{\cW}{{\cal W}}
\newcommand{\co}{{\cal O}}

\renewcommand{\theequation}{\thesection.\arabic{equation}}
\def\@eqnnum{{\reset@font\rm (\theequation)}}

\def\abstract{
\advance \rightskip by 10mm
\advance \leftskip by 10mm
\vspace{-0.8em}
\noindent
\small{\bf Abstract.}
}
\def\endabstract{\par\normalsize\rm}

\def\Xint#1{\mathchoice
{\XXint\displaystyle\textstyle{#1}}%
{\XXint\textstyle\scriptstyle{#1}}%
{\XXint\scriptstyle\scriptscriptstyle{#1}}%
{\XXint\scriptscriptstyle\scriptscriptstyle{#1}}%
\!\int}
\def\XXint#1#2#3{{\setbox0=\hbox{$#1{#2#3}{\int}$}
\vcenter{\hbox{$#2#3$}}\kern-.5\wd0}}
\def\ddashint{\Xint=}
\def\dashint{\Xint-}

\def\a{\alpha}
\def\b{\beta}
\def\d{\delta}\def\D{\Delta}
\def\e{\epsilon}
\def\g{\gamma}\def\G{\Gamma}
\def\k{\kappa}
\def\lam{\lambda}\def\Lam{\Lambda}
\renewcommand\o{\omega}\renewcommand\O{\Omega}
\def\s{\sigma}\def\S{\Sigma}
\renewcommand\t{\theta}\def\vt{\vartheta}
\newcommand{\vphi}{\varphi}
\def\z{\zeta}

\newcommand{\tsigma}{\tilde{\s}}
\newcommand{\tbsigma}{\tilde{\bsigma}}
\def\te{\tilde{\e}}
\def\tu{\tilde{u}}

\newcommand{\bchi}{\mbox{\boldmath$\chi$}}
\newcommand{\bdelta}{\mbox{\boldmath$\delta$}}
\newcommand{\bepsilon}{\mbox{\boldmath$\epsilon$}}
\newcommand{\bfeta}{\mbox{\boldmath$\eta$}}
\newcommand{\bgamma}{\mbox{\boldmath$\gamma$}}
\newcommand{\bomega}{\mbox{\boldmath$\omega$}}
\newcommand{\bvphi}{\mbox{\boldmath$\varphi$}}
\newcommand{\bphi}{\mbox{\boldmath$\phi$}}
\newcommand{\bPhi}{\mbox{\boldmath$\Phi$}}
\newcommand{\bpsi}{\mbox{\boldmath$\psi$}}
\newcommand{\bPsi}{\mbox{\boldmath$\Psi$}}
\newcommand{\bsigma}{\mbox{\boldmath$\sigma$}}
\newcommand{\btau}{\mbox{\boldmath$\tau$}}
\newcommand{\bxi}{\mbox{\boldmath$\xi$}}
\newcommand{\brho}{\mbox{\boldmath$\rho$}}
\newcommand{\bbeta}{\mbox{\boldmath$\beta$}}
\newcommand{\bzeta}{\mbox{\boldmath$\zeta$}}

\def\bk{\boldsymbol{\kappa}}
\def\bmu{\boldsymbol\mu}
\def\bxi{\boldsymbol{\xi}}
\def\bz{\boldsymbol{\zeta}}

\def\ba{{\bf a}}
\def\bb{{\bf b}}
\def\bc{{\bf c}}
\def\be{{\bf e}}
\def\bff{{\bf f}}
\def\bg{{\bf g}}
\def\bn{{\bf n}}
\def\bp{{\bf p}}
\def\bq{{\bf q}}
\def\bs{{\bf s}}
\def\bt{{\bf t}}
\def\bu{{\bf u}}
\def\bv{{\bf v}}
\def\bw{{\bf w}}
\def\bx{{\bf x}}
\def\by{{\bf y}}
\def\bzz{{\bf z}}

\def\bD{{\bf D}}
\def\bE{{\bf E}}
\def\bF{{\bf F}}
\def\bH{{\bf H}}
\def\bJ{{\bf J}}
\def\bV{{\bf V}}
\def\bU{{\bf U}}
\def\bW{{\bf W}}
\def\bX{{\bf X}}
\def\bY{{\bf Y}}

\def\cA{{\cal A}}
\def\cC{{\cal C}}
\def\cD{{\cal D}}
\def\cE{{\cal E}}
\def\cF{{\cal F}}
\def\cG{{\cal G}}
\def\cI{{\cal I}}
\def\cJ{{\cal J}}
\def\cK{{\cal K}}
\def\cL{{\cal L}}
\def\cO{{\cal O}}
\def\cP{{\cal P}}
\def\cQ{{\cal Q}}
\def\cR{{\cal R}}
\def\cS{{\cal \Sigma}}
\def\cT{{\cal T}}
\def\cU{{\cal U}}
\def\cV{{\cal V}}

\def\scT{{_\cT}}
\def\sD{{_D}}
\def\sE{{_E}}
\def\sF{{_F}}
\def\sFz{{_{F_z}}}
\def\sK{{_K}}
\def\sI{{_I}}
\def\sb{{_b}}
\def\sN{{_N}}

\def\curl{{{\bf curl} \ }}
\def\rot{{\mbox{rot}\ }}
\def\BPI{{\bf \Pi}}

\def\cth{\cT_h}
\def\ctH{\cT_H}

\def\tJ{\tilde{\J}}

\def\hK{\widehat{K}}
\def\hx{\widehat{x}}
\def\hy{\widehat{y}}
\def\bhv{\widehat{\bv}}

\def\l{\ell}
\def\bl{\boldsymbol{\ell}}
\def\col{\colon}
\def\f12{\frac12}
\def\dfrac{\displaystyle\frac}
\def\dint{\displaystyle\int}
\def\nab{\nabla}
\def\p{\partial}
\def\sm{\setminus}
\def\dsum{\displaystyle\sum}
\newcommand{\pp}[2]{\frac{\partial {#1}}{\partial {#2}}}
\def\bzero{{\bf 0}}

\def\divv{\nab\cdot}
\def\divx{\nab_x\cdot}
\def\divtx{\nab_{t,x}\cdot}
\def\nabx{\nab_x}

\newcommand{\grad}{\nabla}
\newcommand{\curlt}{{\nabla \times}}
\newcommand{\gperp}{\nabla^{\perp}}
\newcommand{\gradt}{\nabla\cdot}

\def\forallqq{\quad\forall\,}
\def\aph{A^{1/2}}
\def\amh{A^{-1/2}}

\def\osc{{\rm osc \, }}

\def\Im{{\rm Im}}
\newcommand{\tr}{{\rm tr}}
\def\divvr{{\rm div}}
\def\curllr{{\rm curl}}
\def\curll{{\rm curl}}
\def\curl{{\bf curl}}
\newcommand{\bgrad}{{\bf grad}}
\newcommand\diam{\mathrm{diam\,}}
\renewcommand\Im{\mathrm{Im\,}}
\def\Span{\mbox{Span}}
\def\supp{\mbox{supp\,}}
\newcommand{\trace}{{\rm trace}}

\newcommand{\tri}{|\!|\!|}
\newcommand{\ljump}{\lbrack\!\lbrack}
\newcommand{\rjump}{\rbrack\!\rbrack}
\newcommand{\bdm}{\begin{displaymath}}
\newcommand{\edm}{\end{displaymath}}
\newcommand{\beq}{\begin{equation}}
\newcommand{\eeq}{\end{equation}}
\newcommand{\beqa}{\begin{eqnarray}}
\newcommand{\eeqa}{\end{eqnarray}}
\newcommand{\beqas}{\begin{eqnarray*}}
\newcommand{\eeqas}{\end{eqnarray*}}
\newcommand{\ul}{\underline}
\newcommand{\wh}{\widehat}
\newcommand{\la}{\langle}
\newcommand{\ra}{\rangle}

\newcommand{\Lt}{L^2(\Omega)}
\newcommand{\Lts}{L^2(\Omega)^2}
\newcommand{\Ltc}{L^2(\Omega)^3}
\newcommand{\Ho}{H^1(\Omega)}
\newcommand{\Hoh}{H^1(\wh{\Omega})}
\newcommand{\Hoi}{H^1(\Omega_i)}
\newcommand{\Hos}{H^1(\Omega)^2}
\newcommand{\Hoc}{H^1(\Omega)^3}
\newcommand{\Hoch}{H^1(\wh{\Omega})^3}
\newcommand{\Hoci}{H^1(\Omega_i)^3}
\newcommand{\Hoz}{H^1_0(\Omega)}
\newcommand{\Ht}{H^2(\Omega)}
\newcommand{\Hti}{H^2(\Omega_i)}
\newcommand{\Hts}{H^2(\Omega)^2}
\newcommand{\Htc}{H^2(\Omega)^3}
\newcommand{\Htz}{H^0(\Omega)}
\newcommand{\Hh}{H^{1/2}(\Gamma)}
\newcommand{\Hhi}{H^{1/2}(\Gamma_i)}
\newcommand{\Hmh}{H^{-1/2}(\Gamma)}
\newcommand{\Hdiv}{H(\divvr;\,\Omega)}
\newcommand{\Hdivh}{H(\divv;\,\wh \Omega)}
\newcommand{\hcurl}{H(\curl\,A;\,\Omega)}
\newcommand{\Hcurl}{H(\curll\,A;\,\Omega)}
\newcommand{\Hcrl}{H(\curll\,;\,\Omega)}
\newcommand{\hcrl}{H(\curl\,;\,\Omega)}
\newcommand{\Hcrlh}{H(\curll\,;\,\wh\Omega)}
\newcommand{\hcrlh}{H(\curl\,;\,\wh\Omega)}
\newcommand{\Wdiv}{\BW_0(\mbox{\divv}\,;\,\Omega)}
\newcommand{\Wcurl}{\BW_0(\mbox{\curl}\,A;\,\Omega)}
\newcommand{\WcrossV}{\BW \times V}

\def\calS{{\cal S}}
\def\calT{{\cal T}}
\def\cB{{\cal B}}
\def\cH{{\cal H}}
\def\ba{{\mathbf{a}}}
\def\cM{{\cal M}}
\def\cN{{\mathcal{N}}}
\def\cE{{\mathcal{E}}}
\def\cT{{\mathcal{T}}}
\def\cJ{{\mathcal{J}}}
\def\cL{{\mathcal{L}}}

\def\bbA{{\mathbb{A}}}

\def\bbC{{\mathbb{C}}}

\def\bE{{\bf E}}
\def\bS{{\bf S}}
\def\bT{{\bf T}}

\def\br{{\bf r}}
\def\bW{{\bf W}}
\def\bLambda{{\bf \Lambda}}

\def\beps{{\bf\epsilon}}

\newcommand{\lJump}{[\![}
\newcommand{\rJump}{]\!]}
\newcommand{\jump}[1]{[\![ #1]\!]}

\newcommand{\st}{\tilde{\bsigma}}
\newcommand{\sh}{\hat{\bsigma}}
\newcommand{\rd}{\brho^{\Delta}}

\newcommand{\WH}{W\!H}
\newcommand{\NE}{N\!E}

\newcommand{\ND}{N\!D}
\newcommand{\BDM}{B\!D\!M}

\newcommand{\langles}{\langle\!\langle}
\newcommand{\rangles}{\rangle\!\rangle}

\newcommand{\eps}{\varepsilon}

\newcommand{\Jp}{\mathsf{J}^\mathtt{p}}
\newcommand{\Jd}{\mathsf{J}^\mathtt{d}}

\newcommand{\sT}{{_T}}
\newcommand{\sRT}{{_{RT}}}
\newcommand{\sBDM}{{_{BDM}}}
\newcommand{\sWH}{{_{WH}}}
\newcommand{\sND}{{_{ND}}}
\newcommand{\sV}{_\cV}

\newcommand{\dd}{\underline{{\mathbf d}}}
\newcommand{\C}{\rm I\kern-.5emC}
\newcommand{\R}{\rm I\kern-.19emR}
\newcommand{\W}{{\mathbf W}}
\def\3bar{{|\hspace{-.02in}|\hspace{-.02in}|}}
\newcommand{\A}{{\mathcal A}}

\newcommand{\rb}{{\texttt{rb}}}
\newcommand{\fe}{{\texttt{fe}}}

\newcommand{\rferb}{{\mathtt{r_{rb,fe}}}}

\newcommand{\and}{\quad\mbox{and}\quad}

\def\dunderline#1{\underline{\underline{#1}}}

\newcommand{\matr}[1]{\bm{#1}}     

\def\mtx#1{\dunderline{\matr{#1}}}

\newcommand{\vect}[1]{\mathsf{#1}}     
\def\vec#1{\underline{\vect{#1}}}

\newcommand{\cblue}[1]{\color{blue}{#1}}

\maketitle
\begin{abstract}
The a posteriori error estimate and greedy algorithms play central roles in the reduced basis method (RBM). In \cite{Yano:15,Yano:16,Yano:18}, several versions of RBMs based on exact error certifications and greedy algorithms with spatio-parameter adaptivities are developed. In this paper, with the parametric symmetric  coercive elliptic boundary value problem as an example of the primal-dual variational problems satisfying the strong duality, we develop primal-dual reduced basis methods (PD-RBM) with robust true error certifications and discuss three versions of greedy algorithms to balance the finite element error, the exact reduced basis error, and the adaptive mesh refinements. 

For a class of convex minimization variational problems which has corresponding dual problems satisfying the strong duality, the primal-dual gap between the primal and dual functionals can be used as a posteriori error estimator. This primal-dual gap error estimator is robust with respect to the parameters of the problem, and it can be used for both mesh refinements of finite element methods and the true RB error certification. 

With the help of integrations by parts formula, the primal-dual variational theory is developed for the symmetric coercive elliptic boundary value problems with non-homogeneous boundary conditions by both the conjugate function and Lagrangian theories.  A generalized Prager-Synge identity, which is the primal-dual gap error representation for this specific problem, is developed. RBMs for both the primal and dual problems with robust error estimates are developed.   
The dual variational problem often can be viewed as a constraint optimization problem. In the paper, different from the standard saddle-point finite element approximation,  the dual RBM is treated as a Galerkin projection by constructing RB spaces satisfying the homogeneous constraint.

Inspired by the greedy algorithm with spatio-parameter adaptivity of \cite{Yano:18}, adaptive balanced greedy algorithms with primal-dual finite element and reduced basis error estimators are discussed. Numerical tests are presented to test the PD-RBM with adaptive balanced greedy algorithms.
\end{abstract}



\section{Introduction}\label{intro}
\setcounter{equation}{0}
%

The reduced basis method (RBM)  is a very accurate and efficient method for solving parameterized problems many times for different parameter values within a given parametric domain, see review articles and books \cite{Maday:06,RHP:08,QRM:11,QMN:16,HRS:16,OR:16,Ha:17,BCOW:17}. The a posteriori error estimate and greedy algorithms play central roles in RBM. In pioneering papers of Yano \cite{Yano:15,Yano:16,Yano:18}, several versions of RBMs based on exact error certifications and greedy algorithms with spatio-parameter adaptivities are developed.  In this paper, with the parametric symmetric coercive elliptic boundary value problem as an example of the problems in the primal-dual variational framework, we develop primal-dual reduced basis methods (PD-RBM) with robust true error certifications and discuss three versions of greedy algorithms to balance the finite element (FE) error, the exact RB error, and the adaptive mesh refinements. 

In this introduction, we first present some concerns of the  residual type of error estimator and greedy algorithms of the classic RBM as our motivations. Then we discuss PD-RBM and adaptive balanced greedy algorithms to be presented in the paper.


\subsection{Motivations from concerns about the residual error estimator and classic greedy algorithms}

Traditionally, the residual type of error estimator is used in RBM. As it is known to many researchers \cite{Yano:15,Yano:16,Yano:18,OS:15,ASU:17}, several concerns are known about the residual type of error estimator and classic greedy algorithms: (1) {\it The RB error and the FE error tolerances are often unbalanced.} It is often assumed that the FE solution is very accurate and can be viewed as a "truth" solution, then the RB basis is built to make the difference between the RB and FE solutions (measured by the residual error estimator) to be less than some tolerance. In practice, such assumption that the FE error is very small is often questionable, and the FE error and the error between RB and FE solutions should be balanced in the greedy algorithm.  (2) {\em The RB snapshots are often computed on non-optimal meshes.} In RBM, in order to compute the inner products of different RB snapshots, we prefer that all RB snapshots are computed in the same mesh. (Although we can compute them in different meshes and use some grid transfer operators to compute their inner products. But these grid transfer operators are non-trivial for complicated meshes.) For many challenging problems, different parameters need different locally refined meshes. If the computational resource is not unlimited, can we construct a common mesh which is good for all parameters in the domain?

Besides the above two concerns, the parametric robustness of RBM are rarely discussed. For PDEs depending on parameters/coefficients, it is often very important to get the so-called {\em robust} error estimates, which means there should be no unknown generic constant depending on parameters/coefficients appearing in a priori or a posteriori error estimates, see a series of papers on robust a priori and a posteriori finite element estimates \cite{BeVe:00,Ver:98,Ver:05,Ver:05b,TV:15,Pet:02,CaZh:09,CZ:10a,CZ:12,CaHeZh:17,CCZ:20,CCZ:20mixed}. A non-robust result under or over estimates the error. For parameter-dependent  problems, such under or over estimates can lead to inaccurate results and computational resources wasted. Thus, we have two extra concerns. 

(3) {\em The error is often not measured in the parameter-dependent intrinsic energy norm, and is not robust, thus with a possible big gap between the estimated error and actual error.}  It is shown in Propositions 2.26 and 2.28 of \cite{Ha:17} that the ratio of the residual RB error estimator and the true RB error is bounded by the ratio (or the square root of the ratio) of the continuity constant and stability (inf-sup or coercivity) constant. Take the parametric diffusion problem (also called the thermal block problem in RBM literature, see Section 6.1 for a detailed description) with the diffusion coefficient $A=\a(x;\mu)I$, $\a(x;\mu)>0$, and the homogeneous Dirichlet boundary condition $u=0$ on $\p\O$ as an example.  Choose the norm to be the standard $H_0^1(\O)$ norm, $\|\nabla v\|_0$ ($\|\cdot\|_0$ is the $L^2$-norm). For a given parameter $\mu$, the coercivity constant is $\min_{x\in \O}\{\a(x;\mu)\}$ and the continuity constant is $\max_{x\in \O}\{\a(x;\mu)\}$. Thus, the possible ratio of the residual RB error estimator and the true RB error can be quite big if the genetic $H^1_0$-norm is used. Similar results are also true for a priori error estimates of RBM, see Prop. 2.19 of  \cite{Ha:17}.
On the other hand, if we choose the norm to be the intrinsic energy norm, $\|A^{1/2}\nabla v\|_0$, the continuity and coercivity constants of the bilinear form are both one. If the RB error is measured in the intrinsic energy norm, we can have robust a priori and a posteriori error estimates with the ratio being one. Some earlier discussion on using energy norm can also be found in \cite{Yano:15,Yano:18}. 

(4) {\it The role of the lower bound of coercivity constant is often limited for many cases.} 
The discrete stability constant usually appears in the RB residual error estimator. 
If no other means are available, one may use the Successive Constraint Method (SCM) \cite{HuRoSePa:07,Yano:18} to estimate it. But as shown in the diffusion equation example, even if the computation of the constant is exact, the quality of the a priori and a posteriori error estimates depends on the ratio of the continuity constant and the stability constant. From the linear algebra viewpoint, the discrete stability constant and the discrete continuity constant are the smallest and largest generalized singular values of the matrix associated with the bilinear form, respectively. Thus, even if we pay some non-trivial expense to compute the discrete stability constant very accurately, it still can over-estimate the error if we measure the error in parameter-independent genetic norms. On the other hand, the parametric energy norm uses all the information of the bilinear form and its associated matrix. 

For the classic residual type of error estimator of RBM, a parametric dependent norm cannot be used due to the breakdown of the offline-online decomposition. With points (3) and (4), to reduce the effect of non-robustness, we should use an energy norm with a fixed parameter instead of using a genetic norm, as suggested in page A2869 of \cite{OS:15}. This will improve the ratio between the residual estimator and the RB error to the parameter variation (or the square root of it).

\subsection{Methods of the paper}
Having above concerns in mind, it is important to seek alternative error estimators. Classical greedy algorithms should also be modified to balance the RB error, the FE error, and adaptive mesh refinements with the help of the new error estimator. 
For the adaptive FEM, a posteriori error estimators also play the central role. Beside the residual type error estimator, many other types error estimators are studied in adaptive FEMs, see \cite{AO:00, Ver:13}. One of them is {\em the primal-dual gap error estimator}. 

Many mathematical, physical, and economical important problems can be viewed as minimizations of a primal energy functional, like the reaction-diffusion equation, the diffusion equation, the linear elasticity problem, elliptic obstacle problems, the nonlinear Laplacian problem, contact problems in mechanics, the Monge-Kantorovich problem, and others, see \cite{ET:76,Zeidler:85,NR:04,Han:05}. For many of these important problems, {\em a pair of primal-dual variational problems satisfying the strong duality} can be constructed by either the convex conjugate theory or the saddle-point Lagrangian theory. For approximations of the primal and dual problems, the primal-dual gap between the primal and dual functionals at their respective approximations is a guaranteed error bound. In FEMs, the primal-dual variational framework is used to construct robust a posteriori error estimator for various problems with very good properties, see  \cite{BS:08,CZ:12,EV:15,CCZ:20,CCZ:20mixed,Han:05,Repin:00,Repin:00_nonlinear,NR:04,BHS:08,BPS:20}. 

In this paper, we first discuss the primal and dual problems and a posteriori error estimate in an abstract setting, then two approaches are presented to construct primal and dual variational problems. We then discuss a class of symmetric coercive linear elliptic mixed boundary value problem and develop the primal and dual problems and error relations by the two version of theories. 
Our contribution on the primal-dual theory of the elliptic problem are three-fold: instead of using the adjoint operator which is hard to define for non-homogeneous boundary problem, with the help of integration by parts formula and the formal adjoint operator, the dual problem and its constraint set can be stated straightforwardly; second, we present both the convex conjugate function and the Lagrangian approaches to develop the dual problems; third,  we explain how the Lagrangian is constructed by known relations between the primal and dual variable. A generalized Prager-Synge identity, which is the primal-dual gap error representation for the symmetric coercive linear elliptic boundary value problem, is developed.

In the pioneering paper of Yano \cite{Yano:15}, important ideas like a posteriori error estimate for RBM based on the primal-dual gap, the exact error certification based on the energy norm are proposed. Following the ideas of \cite{Yano:15}, with the symmetric coercive elliptic mixed boundary value problem as an example of {\em the primal-dual variational problems satisfying the strong duality}, we propose {\em primal-dual methods with primal-dual gap error estimators used for both FEMs and RBMs}. In the FE context, the primal-dual gap error estimator is reliable, efficient, and robust with respect to the parameters, and can be used to drive the adaptive mesh refinements. When applied to the RBM, unlike the residual error estimator, the primal-dual gap error estimator measures the true error in the robust parameter-dependent  energy norm, not some "truth" FE-RB error in non-robust parameter-independent  generic norm, and there is no need to compute the stability constant. Thus, the primal-dual gap a posteriori error estimator can overcome the concerns of the standard residual type error estimator for this class of primal-dual variational problems. Although, we only discuss PD-RBM for a linear problem for simplicity, with the help of more advanced tools like the empirical interpolation method \cite{GMNP:07}, the methods here can be generalized to more complicated problems in the primal-dual variational framework.

In the primal-dual variational framework, the dual variational problem is often a constraint optimization problem. A constraint, which is also called the dual-feasibility condition, needs to be enforced. In \cite{Yano:15}, the dual RB basis functions are locally constructed, and the dual RB problem is also solved as a constraint optimization problem (a saddle point problem) to enforce the dual-feasibility condition. In \cite{Yano:16,Yano:18}, a minimal-residual mixed method and a modified version with a problem dependent coefficient (both are based on least-squares variational principles) are proposed. As a least-squares method, which is based on an artificial energy functional, the method in \cite{Yano:16} avoids the exact dual-feasibility condition and can be applied to problems beyond the primal-dual variational setting. In \cite{Yano:18}, in a modified version, a problem-dependent weight $1/\delta$ is incorporated into the least-squares functional to control and penalize the violation of the dual-feasibility condition. 

In this paper, we present a systematic version of the PD-RBM with robust a priori and a posteriori error estimates. In the PD-RBM developed in the paper, we construct the RB basis for both the primal and dual problems by solving global FE problems to ensure the good quality of the RB basis. A unified best approximation property and a posteriori error estimates of FEMs and RBMs of the primal, dual, and the primal-dual combined problems are shown. Especially, we show in a comparison result that measured in energy norms, the FE solution is always better than or as good as its corresponding RB solution for the same parameter. This comparison result can guide the choice of RB stopping criteria. Different from \cite{Yano:15}, for the dual RBM, instead of solving a constraint optimization problem, using the technique similar to the treatment of the non-homogenous essential condition, the problem is transformed into a weighted Galerkin projection problem in the space with a homogeneous dual-feasibility condition. Thus, both the primal and dual RBMs are treated in the same classic RBM framework. The treatment of the constraint in the dual RBM also opens doors for simple treatments of RBMs for other problems with constraints. Compared to \cite{Yano:18}, since the dual-feasibility condition is enforced exactly in our dual RBM, the method developed in the paper can be viewed as an optimal case of choosing the weight $\delta=0$ in \cite{Yano:18}.

Even though the idea of the primal-dual gap error estimator is widely used in the adaptive FEM community, we need to point out some important differences when applying it to the RBM. For the primal-dual gap a posteriori error estimator for FEMs, where the goal is to estimate the error for the primal problem, an approximation of the dual solution is often only solved locally to reduce the computational cost, see \cite{BS:08,CZ:12,AV:19}. Sometimes, the dual variable is even constructed purely explicitly to ensure the process is computationally cheap, see \cite{BS:08}. For a parameter-dependent problem, simple explicit construction may lead to the bound to be non-robust, that is, although it is a guaranteed upper bound, but the bound can be too big, see the example presented in \cite{Ver:09}. Thus, even for local constructions of the dual variables for FE a posteriori error estimates, we pointed out that a local optimization problem is necessary to ensure the error estimator is robust with respect to the parameter, see \cite{CZ:12}. For the case of RBM, the scenario is different. Not only used in the FE error estimator, the RB dual solutions are also going to be used to compute the RB primal-dual gap error estimator. To ensure the primal-dual gap and the related RB primal-dual gap error estimator are tight and efficient, we propose to compute the RB snapshots of the dual problems globally in full instead of computing them locally. This will cost a little more in the offline stage, but this can guarantee the primal and the dual FE approximations are of the same order of accuracy. The good quality of the dual RB basis functions is important to get a good dual RB approximation and a tight error bound. It is also crucial to guarantee the optimal a priori error estimates and the comparison results in Theorem \ref{comp_fe_rb}. Computing both the primal and dual unknowns accurately is very common in the least-squares approach, both in the FEM \cite{BG:09} and in the RBM \cite{Yano:16,Yano:18}.

For the greedy procedures, there are several existing improvements of the standard algorithms with adaptive FE computations. In \cite{Yano:16,Yano:18}, spatio-parameter adaptive greedy algorithms are proposed. Other approaches combining RBM and adaptive FEM can be found in \cite{OS:15,URL:16,ASU:17}. Inspired by the greedy algorithm with spatio-parameter adaptivity by \cite{Yano:18}, we discuss three versions of balanced adaptive greedy algorithms to balance the FE error, the exact RB error, and the adaptive mesh refinements. The first algorithm is for a fixed mesh but an adaptive RB tolerance. The purpose of this algorithm is to develop a stopping criteria for the RB greedy algorithm:  the RB set quality is determined by the worst FE approximation in the parameter set. The second algorithm, assuming unlimited computational resources, is developed with a fixed RB tolerance but a common adaptive mesh, which is the spatio-parameter adaptivity algorithm by \cite{Yano:18} in our context. Assuming a limited computational resource, the third algorithm tries to generate a balanced mesh for the whole RB parameter set under a limited number of DOFs with bi-adaptivities: the RB tolerance adaptivity and the mesh adaptivity. When generating the mesh in the third algorithm, an error estimator which estimates the total error for all RB snapshots is used to drive the mesh refinement to ensure the final common mesh is balanced for the whole snapshot parameters. 

\subsection{Layout of the paper}
The layout of the paper is as follows. In Section 2, we first discuss the primal and dual problems and a posteriori error estimate in an abstract setting, then two approaches are presented to construct primal and dual variational problems. In Section 3, we discuss a class of symmetric coercive  linear elliptic mixed boundary value problem and develop its primal and dual problems  by two versions of theories and a generalized Prager-Synge identity. FEMs and PD-RBMs for the linear elliptic boundary value problem are discussed in Sections 4 and 5, respectively. In Section 6, we first present several examples of the symmetric coercive linear elliptic boundary value problem, then we discuss examples beyond the linear elliptic boundary value problems. In Section 7, we present three adaptive balanced greedy algorithms. Using the diffusion equation with discontinuous coefficient on an L-shape domain as an example, numerical experiments performed to test the PD-RBM and  three adaptive balanced greedy algorithms are presented in Section 7. In Section 8, we make some concluding remarks and discuss some future plans.

\section{Primal-Dual Variational Principles}
\setcounter{equation}{0}

For a class of convex minimization problems having a corresponding dual problem satisfying the strong duality, such as the diffusion equation, the reaction-diffusion problem, the linear elasticity problem, the obstacle problem, the $p$-Laplacian problem, and others, the primal-dual variational principle provides a framework for a posteriori error estimation, see the books of Ekeland and Temam \cite{ET:76}, Zeidler \cite{Zeidler:85}, Neittaanmäki and Repin \cite{NR:04}, and Han \cite{Han:05}. In this section, we first discuss the primal and dual problems and a posteriori error estimate in an abstract setting, then we present two approaches to construct primal and dual variational problems.

\subsection{Abstract primal and dual problems and a posteriori error estimate}
First, we describe the basic idea of deriving the a posteriori error estimate using the primal-dual framework (p.57 of \cite{Han:05}). Let $\Jp$ be a given functional ({\em the primal functional or the primal energy functional}) on a non-empty set $K$. For a minimization problem  ({\em the primal problem}):
\beq \label{pp}
\mbox{Find  } u\in K: \quad
\Jp(u) = \inf_{v\in K} \Jp(v), 
\eeq
we wish to construct a functional $\Jd$ ({\em the dual functional or the complementary energy functional}) on a non-empty set $M$, and a related maximization problem ({\em the dual problem}):
\beq \label{pd}
\mbox{Find  } \sigma \in M:\quad 
\Jd(\sigma) = \sup_{\tau\in M} \Jd(\tau).
\eeq
We assume that the following equality for the optimizers holds:
\beq \label{pequal}
\mbox{(strong duality)}\quad\quad
\Jp(u) = \Jd(\sigma). 
\eeq
Also, we expect some {\em extremal relations} hold for the optimizers $u$ and $\sigma$.
%
%

Let $\hat{u}\in K$ and $\hat{\sigma}\in M$, define the primal and dual energy differences, respectively,
\beq
ED_p(u, \hat{u}) = \Jp(\hat{u}) - \Jp(u)\quad\mbox{and}\quad
ED_d(\sigma, \hat{\sigma}) =  \Jd(\sigma) - \Jd(\hat{\sigma}).
\eeq
It is easy to check that both the energy differences are non-negative.
We have the following simple and elegant result.
\begin{prop}\label{error}
Suppose that \eqref{pequal} is true, then for any $\hat{u}\in K$ and $\hat{\sigma}\in M$, we have
\begin{eqnarray}\label{duality_identity}
\mbox{(combined error identity)}\quad
ED_p(u, \hat{u})+ED_d(\sigma, \hat{\sigma}) = \Jp(\hat{u}) - \Jd(\hat{\sigma}), \\[1mm]
\label{duality_bound}
0\leq ED_p(u, \hat{u})\leq \Jp(\hat{u}) - \Jd(\hat{\sigma}),
\quad \mbox{and}\quad
0\leq ED_d(\sigma, \hat{\sigma}) \leq \Jp(\hat{u}) - \Jd(\hat{\sigma}).
\end{eqnarray}
\end{prop}
\begin{proof}
The relation \eqref{duality_identity} follows from the fact $\Jp(u) = \Jd(\sigma)$. The bounds in \eqref{duality_bound} are then obtained by the fact that both energy differences are non-negative. 
\end{proof}

When $\hat{u} \in K$ is a (numerical) approximation of $u$ and $\hat{\sigma} \in M$ is a (numerical) approximation of $\sigma$, the primal-dual gap $\Jp(\hat{u}) - \Jd(\hat{\sigma})$ can be used as a guaranteed upper bound for both energy differences. If we view the primal-dual problem as one combined problem, then \eqref{duality_identity} shows that $\Jp(\hat{u}) - \Jd(\hat{\sigma})$ is precisely the combined energy difference. The following a posteriori error estimator for the primal-dual problem then can be defined:
\beq \label{gap}
\mbox{(primal-dual gap estimator)} \quad\xi=\Jp(\hat{u}) - \Jd(\hat{\sigma}).
\eeq
Suppose there is an underlying computational mesh $\cT$. On each element $T\in\cT$, carefully defining the primal and dual functionals' restriction on $T$, we can construct the following local error indicator to drive the mesh refinement algorithm:
\begin{eqnarray} \label{def_etaK}
\xi_T=\Jp(\hat{u})|_T - \Jd(\hat{\sigma})|_T,\quad \xi_T\geq 0, \quad \forall T \in \cT,\quad
\mbox{ s.t. } \xi = \sum_{T\in\cT} \xi_T.
\end{eqnarray}
For the quadratic functionals, the square roots of $\xi$ and $\xi_T$ are often used as the global error estimator and the local indicator, respectively. 

\subsection{Two approaches to construct primal and dual problems}
In the above abstract framework, it is assumed that the primal problem \eqref{pp} and the dual problem \eqref{pd} can be constructed such that the strong duality \eqref{pequal} is true. In this subsection, 
we present two simplified versions of the duality theory such that the construction is possible. One is the conjugate convex function theory and the other is the Lagrangian/saddle-point theory. More general versions and proofs can be found in \cite{ET:76,Zeidler:85,Han:05,NR:04,ABM:14}.

\subsubsection{Primal-dual variational theory via conjugate convex function}
We start from some definitions from convex analysis, see \cite{Rock:70,ET:76,Zeidler:85,aubin:98,RW:98,NR:04,Han:05}. Let $X$ be a reflexive Banach space (for example, a Hilbert space or an $L^p$ spaces for $1<p<\infty$). Let $\overline{\mathbb{R}}=\mathbb{R} \cup \{\pm \infty\}$ be the set of the extended real number. For a convex functional $f$, define its effective domain $\mbox{dom}(f)=\{x\in X: f(x) < \infty\}$ and its epigraph $\mbox{epi}(f)=\{(x,a)\in X\times \mathbb{R}: f(x) \leq a\}$. The functional $f$ is {\em proper} if $\mbox{dom}(f)\neq \emptyset$ and $f(x) > -\infty$ for all $x\in X$. $f$ is called {\em lower semicontinuous (l.s.c)} if its epigraph is a closed set. We use the notation $\Gamma_0(X)$ to denote the set of all proper, convex, l.s.c functionals on $X$. Define the conjugate functional (also known as the polar function, the Legendre–Fenchel transformation, or the Fenchel transformation, see \cite{RW:98,aubin:98,Rock:70}) of $f$ as
\beq
f^*(x^*) = \sup_{x\in \mbox{dom}(f)} [\langle x^*, x\rangle_{X^*,X} - f(x)],\quad \forall x^*\in X^*.
\eeq 
Let $Y$ and $Q$ be two reflexive Banach spaces, and $Y^*$ and $Q^*$ are their dual spaces. Assume that $\Lambda \in L(Y,Q)$ is a linear continuous operator from $Y$ to $Q$ and $\Lambda^* \in  L(Q^*,Y^*)$ is its adjoint operator,
\beq \label{defLambdastar}
\langle \Lambda^* q^*, v\rangle_{Y^*,Y} =  \langle q^*, \Lambda v\rangle_{Q^*,Q}, \quad\forall v\in Y,\; q^*\in Q^*.
\eeq
In general, we can discuss the primal problem with respect to the primal functional $\phi(v,\Lambda v)$, where $\phi$ maps $Y\times Q\to \overline{\mathbb{R}}$ and its corresponding dual problems, see \cite{ET:76}. Here, only a special case where $\phi$ is of a separated form is considered, 
$$
\phi(v,q) = F(v) + G(q), \quad \forall v\in Y, \; q \in Q.
$$
The conjugate functional of $\phi$ is (see \cite{ET:76,aubin:98,Han:05})
$$
\phi^*(v^*,q^*) = F^*(v^*) + G^*(q^*), \quad \forall v^*\in Y^*, \; q^* \in Q^*,
$$
where $F^*$ and $G^*$ are the respective conjugate functionals of $F$ and $G$. 
Then we have the following pair of primal and dual problems, see \cite{ET:76,aubin:98,Han:05}:

\begin{eqnarray}
\label{pphi}
\mbox{(Primal)}&\;&
\mbox{Find } u\in Y: \;
\Jp(u) = \inf_{v\in Y} \Jp(v), 
\;\Jp(v)=  F(v) + G(\Lambda v),\\
\label{dphi}
\mbox{ (Dual)}&&
\mbox{Find } \sigma \in Q^*: \;
\Jd(\sigma) = \sup_{\tau \in Q^*} \Jd(\tau), 
\; \Jd(\tau)=  -F^*(\Lambda^* \tau) - G^*(-\tau).
\end{eqnarray}
For the problems \eqref{pphi} and \eqref{dphi}, we have the following theorem, see Theorem 2.39 of \cite{Han:05}, Theorem 2.1 of \cite{Repin:00}, and \cite{ET:76}.
\begin{thm} \label{pdframeworkthm}
Assume:
1. $Y$ and $Q$ be two reflexive Banach spaces; $\Lambda \in L(Y,Q)$.
2. $F\in \Gamma_0(Y)$ and $G\in \Gamma_0(Y)$.
3. There exists $u_0\in Y$, s.t. $F$ is finite at $u_0$ and $G$ is continuous and finite at $\Lambda u_0$.
4. (Coercivity) $F(v) + G(\Lambda v) \to +\infty$ as $\|v\|\to \infty$.
Then \eqref{pphi} has a minimizer $u\in Y$ and \eqref{dphi} has a maximizer $\sigma \in Q^*$, such that the strong duality \eqref{pequal} is true,
and the following {\em  extremal relations} hold 
\begin{eqnarray} \label{dr1}
F(u) + F^*(\Lambda^* \sigma)- \langle \Lambda^* \sigma, u\rangle_{Y^*,Y} &=& 0,\\ \label{dr2}
G(\Lambda u) + G^*(-\sigma)+ \langle \sigma, \Lambda u\rangle_{Q^*,Q} &=& 0.
\end{eqnarray}
\end{thm}
%
%
%
%

\subsubsection{Primal-dual variational theory via Lagrangian/saddle-point theory} 
For two nonempty sets $K$ and $M$, introduce a bivariate function $\cL: K\times M \to \mathbb{R}$. A point $(u,\sigma)\in K\times M$ is said to be a saddle-point of the Lagrangian $\cL$, if 
$$
\sup_{\tau\in M} \cL(u,\tau) = \cL(u,\sigma) = \inf_{v\in K}\cL(v,\sigma).
$$
We assume the following conditions hold.
(C1) Let $Y$ and $Z$ be two reflexive Banach spaces, and $K\subset Y$, $M\subset Z$ are convex, closed, and nonempty sets; 
(C2) The function $\cL$ is a continuous function that is concave-convex, i.e., 
$\cL(v,\tau)$  is concave on $v$ for a fixed  $\tau\in M$ and
$\cL(v,\tau)$  is convex on $\tau$ for a fixed $v\in K$;
(C3) $K$ is bounded or there exists a $\tau_0\in M$ such that $\cL(v,\tau_0)\rightarrow \infty$ as $\|v\|\rightarrow \infty$ in $K$;
(C3*) $M$ is bounded or there exists a $v_0\in K$ such that $\cL(v_0,\tau)\rightarrow \infty$ as $\|\tau\|\rightarrow \infty$ in $M$.

The existence of the saddle-point is guaranteed by the von Neumann's minimax theorem (Theorem 49.A of \cite{Zeidler:85}): $\cL$ has a saddle-point $(u,\sigma)\in K\times M$ if assumptions (C1)-(C3) and (C3*) are true.

Now define the primal  and dual problems via the Lagrangian $\cL$:
\begin{eqnarray}
\label{pp1}
\mbox{(Primal)}&\quad& \mbox{Find } u\in K:\quad \Jp(u) = \inf_{v\in K} \Jp(v), \quad \Jp(v)=\sup_{\tau\in M} \cL(v,\tau), \\[1mm]
\label{pd1}
\mbox{(Dual)}&\quad&\mbox{Find } \sigma\in M:\quad
\Jd(\sigma) = \sup_{\tau\in M} \Jd(\tau), \quad \Jd(\tau)= \inf_{v\in K}\cL(v,\tau).
\end{eqnarray}
We have the following Theorem (Theorem 49.B of \cite{Zeidler:85}) to connect the saddle-point and solutions of the primal and dual problems. 
\begin{thm}\label{sadllepointthm}
The following two statements are equivalent:
(i) $(u,\sigma)$ is a saddle-point of $\cL(v,\tau)$ on $K\times M$.
(ii) $u$ is a solution of \eqref{pp1}, 
$\sigma$ is a solution of \eqref{pd1}, and 
$$
\Jp(u) = \cL(u,\sigma) = \Jd(\sigma).
$$
\end{thm}
The problem of finding the critical point $(u,\sigma)$ (the saddle-point) of $\cL$ is the so called mixed method, see \cite{BBF:13,Arnold:90}. 
The two primal-dual variational theories are closely related by introducing a Lagrangian for the primal problem \eqref{pphi}, see \cite{ET:76,NR:04}:
\beq \label{L}
\cL(v,\tau)=\langle \tau, \Lambda v\rangle_{Q^*,Q} - G^*(\tau) -  F(v).
\eeq  
More details about the saddle-point theory can be found in Chapters 49-53 of \cite{Zeidler:85}, Chapter VI of \cite{ET:76}, Section 7.16 of \cite{Ciarlet:13}, and Section 5.4 of \cite{NR:04}.




\section{Symmetric Coercive Linear Elliptic Mixed Boundary Value Problems}
\setcounter{equation}{0}
In this section, we discuss a class of symmetric coercive linear elliptic mixed boundary value problems and develop the primal and dual problems and error relations by the two version of theories. 

The discussion of this problem and other individual problems can also be found in \cite{NR:04,ET:76,Zeidler:85}. 
Our contribution on this problem are three-fold: first, instead of using the adjoint operator $\Lambda^*$ which is hard to define for non-homogeneous boundary problem (in \cite{ET:76}, most problems are of homogeneous boundary conditions; in \cite{NR:04}, all the problems are defined with $\Lambda^*$ which has no explicit definition for a concrete problem with non-homogeneous mixed boundary conditions), with the help of integration by parts formula and a formal adjoint operator $\Lambda^t$, the dual problem and its constraint set can be stated straightforwardly; second, we present both the convex conjugate function and Lagrangian approaches to develop the dual problems; third, we explain how the Lagrangian is constructed by the known relations between the primal and dual variables \eqref{lag}. A generalized Prager-Synge identity, which is the primal-dual gap error representation for the symmetric  coercive linear elliptic boundary value problem, is developed.

Let $\Omega$ be a bounded polyhedral domain in $\mathbb{R}^d$ ($d=2$ or $3$) with Lipschitz boundary. 
We use the notation $(\cdot,\cdot)$ to denote the inner product in $L^2$ space on the domain $\O$ and $\|\cdot\|_0$ to denote the $L^2$-norm. We do not specify the dimension of $L^2$ space, $L^2(\O)$ can be $L^2(\O)^n$, for some positive integer $n$. Let $V$ be a Hilbert space compactly embedded in $L^2(\O)$ and $\Lambda \in L(V,L^2(\O))$. We assume that $\cA$ is a self-adjoint (symmetric) operator on $L^2(\O)$.
Let $\Lambda^t$ be the formal adjoint for $\Lambda$ for smooth enough functions. Note that the definition of $\Lambda^t$ is different from $\Lambda^*$  in \eqref{defLambdastar} since $\Lambda^*$ may depend on boundary conditions.
Define the space $\Sigma$ as 
$$
\Sigma =\{\tau \in L^2(\O): \Lambda^t \tau \in L^2(\O)\}.
$$
Define two trace operators $\tr_0(v)$ on $\p\O$ for $v\in V$ and $\tr_1(\tau)$ on $\p\O$ for $\tau \in \Sigma$, such that the following integrations by parts formula (Green's theorem) holds,
\beq \label{ibp}
(\tau, \Lambda v)_\O - (\Lambda^t \tau, v)_\O = (\tr_1(\tau), \tr_0(v))_{\p\O}, \quad \forall (v, \tau)\in V\times \Sigma.
\eeq
Let the boundary $\p\O = \overline{\Gamma}_D \cup \overline{\Gamma}_N$ with $\Gamma_D \cap \Gamma_N =\emptyset$. We assume $meas({\Gamma_D}) \neq 0$. Define the following subspaces:
\begin{eqnarray}
&V_0: =\{v\in V: \tr_0(v) =0 \mbox{ on }\Gamma_D\}, &
\\ \label{def_S0_S00}
&\Sigma_0: =\{\tau \in \Sigma: \tr_1(\tau) =0 \mbox{ on }\Gamma_N\},
\quad\mbox{and}\quad
\Sigma_{00}=\{\tau \in \Sigma_0: \Lambda^t \tau = 0\}.
&
\end{eqnarray}
We also assume that $c_1\|v\|_0^2\leq (\cA v,v)\leq c_2\|v\|_0^2$ and $\|\Lambda v\|_0 \geq c_3 \|v\|_0$ for $v\in V_0$ with positive constants $c_1$, $c_2$, and $c_3$. Thus we can define the following two norms for $v\in V_0$ and $\tau \in L^2(\O)$ (of which $\Sigma$ is a subset), respectively:
\beq
\tri v \tri = \sqrt{(\cA \Lambda v, \Lambda v)}
\mbox{ and }
\Rnorm{\tau} = \sqrt{(\cA^{-1} \tau, \tau)} \quad \mbox{where } \cA^{-1} \mbox{ is the inverse of }\cA.
\eeq
We also assume the following inf-sup condition holds for some $\beta>0$:
\beq \label{infsup}
\sup_{\tau \in \Sigma} \dfrac{(\Lambda^t \tau,v)}{\|\tau\|_0+\|\Lambda^t \tau\|_0} \geq \beta \|v\|_0, \quad \forall v\in L^2(\O).
\eeq
Consider the {\bf symmetric coercive linear elliptic mixed boundary value problem}:
\beq
\label{ebvp}
\left\{
\begin{array}{lllll}
\Lambda ^t(\cA \Lambda u) = f \mbox{ in }\O,\\[2mm]
\tr_0(u) = g_D \mbox{ on }\Gamma_D\quad \mbox{and}\quad
-\tr_1(\cA \Lambda u) = g_N \mbox{ on }\Gamma_N.
\end{array}
\right.
\eeq
Here, we assume $f$, $g_D$, and $g_N$ are  $L^2$-functions.

\subsection{Primal and dual problems}
Let $u_g\in V$ be a given function with $\tr_0(u_g) =g_D$ on $\G_D$ and
$
V_g=V_0+u_g =
\{v\in V: \tr_0(v) =g_D \mbox{ on }\G_D \}.
$
Test both sides of \eqref{ebvp} by a $v\in V_0$ and use \eqref{ibp}, we have the problem in the weak form.
\begin{defn}[\bf (WP) Weak Primal Problem for \eqref{ebvp}]
\beq \label{ep_weak}
\mbox{Find  } u\in V_g: \quad
a(u,v) = (f,v) - (g_{N}, \tr_0(v))_{\G_N}, \quad \forall v\;\in V_0,
\eeq 
where the bilinear form is defined as
\beq \label{aform}
a(v,w)=(\cA \Lambda v, \Lambda w), \quad  \forall v, w\;\in V. 
\eeq

\end{defn}
In RBM, it is often preferred that the trial and test spaces are the same, so we rewrite the problem (WP) into the following equivalent form.
\begin{defn}{\em\bf ((EWP) Equivalent Weak Primal Problem for \eqref{ebvp})}
\beq \label{ep_weak2}
\mbox{Find  } u=u_0+u_g \mbox{ with } u_0\in V_0 \mbox{ satisfying}: \quad
a(u_0,v) = \varphi(v), \quad \forall v\;\in V_0, 
\eeq 
where the linear form is defined as
\beq \label{def_phi}
\varphi(v)= (f,v) - (g_{N}, \tr_0(v))_{\G_N}-a(u_g, v), \quad \forall v\;\in V. 
\eeq
\end{defn}
We immediately have that the bilinear form $a$ is coercive and continuous with constant being one with respect to $\tri \cdot \tri$ in $V_0$:
\beq \label{one_con_coer}
\tri v \tri^2 =a(v,v) \and a(v,w) \leq \tri v \tri \tri w \tri 
\quad \forall v,w \in V_0.
\eeq
By the Lax-Milgram theorem \cite{Ciarlet:78,Ciarlet:91}, the Problem (WP) and its equivalent Problem (EWP) are well-posed.
By the theory of calculus of variations \cite{ABM:14,Ciarlet:78,Ciarlet:91}, (WP) is the Euler-Lagrange equation of the following minimization problem, and is equivalent to it.
\begin{defn} [\bf (P) Primal Problem of \eqref{ebvp}]
\beq \label{pphi_elliptic}
\mbox{Find } u\in V_g:  
\Jp(u) = \inf_{v\in V_g} \Jp(v), \quad \Jp(v)= \frac{1}{2}a(v,v) - (f,v) +(g_{N}, \tr_0(v))_{\G_N}.
\eeq 
\end{defn}
To introduce the dual problem, define the bilinear form $b$ as:
\beq\label{bform}
b(\rho,\tau)= (\cA^{-1}\rho,\tau) \quad \rho,\tau \in L^2(\O),
\eeq
and define the set $\Sigma_{fg}$ as:
\beq \label{Sigfg}
\Sigma_{fg}=\{\tau \in \Sigma: \Lambda^t \tau = - f \and \tr_1(\tau) = g_N \mbox{  on } \G_N \}.
\eeq
The condition $\Lambda^t \tau = - f$ is often called the dual-feasibility condition.
\begin{defn}[\bf (D) Dual Problem of \eqref{ebvp}]
\beq \label{dphi_elliptic}
\mbox{ Find } \sigma \in \Sigma_{fg}:
\Jd(\sigma) = \sup_{\tau \in \Sigma_{fg}} \Jd(\tau), \quad \Jd(\tau)=  
-\dfrac{1}{2} b(\tau,\tau) -(\tr_1(\tau), g_D)_{\Gamma_D}. 
\eeq 
\end{defn}

\begin{thm} The dual variational problem of \eqref{pphi_elliptic} is \eqref{dphi_elliptic}, and the following two properties hold:
\begin{eqnarray} \label{sd}
\mbox{(strong duality)}&\quad&
\Jp(u) =\Jd(\sigma);\\
\label{relation}
\mbox{(extremal relation)}&\quad&
\sigma = -\cA \Lambda u.
\end{eqnarray}
\end{thm}
Two proofs are presented for the theorem. 
\begin{proof} 
\noindent {\bf (Conjugate convex function theory)} 
Set $Y=V$ and $Q=Q^* = L^2(\O)$, then the functional $\Jp(v)$ can be written as $F(v)+G(\Lambda v)$ with
\beq
F(v) =-(f,v)+ (g_{N}, \tr_0(v))_{\G_N}
\and
G(q) = \dfrac{1}{2}(\cA q,q),
\eeq 
for $ v\in dom(F)=  V_g$ and $q \in L^2(\O)$.
By the definition of $F^*$ and 
\eqref{ibp}, we have
\begin{eqnarray*}
F^*(\Lambda^* \tau) &=& \sup_{v\in \mbox{dom}(F)}\{(\Lambda v,\tau)-F(v)\} \\
&=& 
\sup_{v\in V_g}\{ (\Lambda^t \tau, v) + (\tr_1(\tau), \tr_0(v))_{\Gamma_N}+(f,v)- (g_{N}, \tr_0(v))_{\G_N}\} \\
 &=& 
\sup_{v\in V_g}\{ (\Lambda^t \tau +f, v) + (\tr_1(\tau)-g_N, \tr_0(v))_{\Gamma_N}+ (\tr_1(\tau), g_D)_{\G_D}\}.
\end{eqnarray*}
The constraint set is $
\{\tau \in L^2(\O): (\Lambda^t \tau +f, v) + (\tr_1(\tau)-g_N, \tr_0(v))_{\Gamma_N} =0, \mbox{ for } v \in V_g \}.
$
Since $f$ and $g_N$ are $L^2$ functions, the set is just $\Sigma_{fg}$ in \eqref{Sigfg}. Then, 
\begin{eqnarray*}
F^*(\Lambda^* \tau) &=& 
\left\{ \begin{array}{lc} (\tr_1(\tau),\ g_D)_{\G_D} & \mbox{if } \tau \in \Sigma_{fg}, \\[2mm]
+\infty & \mbox{otherwise.}
\end{array}
\right.
\end{eqnarray*}
A simple calculation (see e.g. \cite{Han:05,RW:98}) shows that 
$
G^*(\tau) = \dfrac{1}{2} (\cA^{-1}\tau,\tau).
$
Following \eqref{dphi}, we have the dual problem \eqref{dphi_elliptic}.
The strong duality \eqref{sd} follows from Theorem \ref{pdframeworkthm}.


Note that from the second duality relation \eqref{dr2}, we have 
\beq \label{usigma}
0=
(\cA\Lambda u, \Lambda u) + (\cA^{-1}\sigma, \sigma) + 2(\sigma, \Lambda u) = \|\cA^{1/2}\Lambda u+\cA^{-1/2}\sigma\|_0^2.
\eeq
Then \eqref{relation} is proved.
\end{proof}
\begin{proof}\noindent {\bf (Lagrangian/saddle-point theory)}
On the set $V_g \times L^2(\O)$, define the following Lagrangian (the choice is made since we know the relation \eqref{relation} for many problems with physical understandings):
\beq \label{lag}
\cL(v,\tau) = \Jp(v) - \frac{1}{2} \|\cA^{1/2} \Lambda v+ \cA^{-1/2} \tau\|_0^2, \quad (v,\btau) \in V_g\times L^2(\O).
\eeq
Clearly, we have 
$\sup_{\tau \in L^2(\O)} \cL(v, \tau) = \cL(v, -\cA\Lambda v)	= \Jp(v)$.
A simple calculation shows that
$$
\cL(v,\tau) = -\frac{1}{2}(\cA^{-1}\tau,\tau) - (\Lambda v,\tau) - (f,v) +(g_N, \tr_0(v))_{\Gamma_N},
$$
which is exactly \eqref{L}. With the integration by parts,  we have
$
\cL(v,\tau) =  -\frac{1}{2}(\cA^{-1}\tau,\tau) - (\Lambda^t\tau+f, v) +(g_N-\tr_1(\tau), \tr_0(v))_{\Gamma_N}-(\tr_1(\tau), g_D)_{\Gamma_D}$,
which implies
\beq
\inf_{v \in  V_g} \cL(v,\tau) = \left\{
\begin{array}{llll}
	-\frac{1}{2}(\cA^{-1}\tau,\tau)-(\tr_1(\tau), g_D)_{\Gamma_D},
&\mbox{if  }  \tau \in \Sigma_{fg},\\[2mm]
	- \infty, &\mbox{otherwise}. & 
\end{array}
\right.
\eeq
So the dual functional and the dual problem \eqref{dphi_elliptic} are derived.

To show \eqref{relation} via the saddle-point theory, we prove the strong duality \eqref{sd} first. Substituting $-\cA \Lambda u\in \Sigma_{fg}$ into $\Jd$ and using \eqref{ibp}, we have  
$\Jd(\sigma)\geq \Jd(-\cA \Lambda u) =\Jp(u)$. 
On the other hand, by the definition of primal and dual functionals \eqref{pp1} and \eqref{pd1} from $\cL$, 
$\Jd(\sigma)  \leq \cL(u,\sigma) \leq \Jp(u)$ is true.
Thus, $\Jp(u) = \Jd(\sigma)$. Substituting into definition and use integration by parts, we also get \eqref{usigma}, thus  \eqref{relation}.
As a byproduct, we have $(u,\sigma)$ is a saddle-point of $\cL$,
\beq \label{saddlepoint}
	\inf_{v\in V_g} \Jp(v)= \Jp(u) =\cL(u, \sigma)= \Jd(\sigma) 
	= \sup_{\tau\in \Sigma_{fg}} \Jd(\tau).
\eeq
\end{proof}

We also rewrite the dual problem (D) \eqref{dphi_elliptic} into the weak form.
\begin{defn}
[\bf (WD) Weak Dual Problem of \eqref{ebvp}] 
\beq \label{ed_weak}
\mbox{Find } \sigma \in \Sigma_{fg}:\quad
b(\sigma, \tau) = -(\tr_1(\tau), g_D)_{\Gamma_D},
\quad \forall \tau \;\in \Sigma_{00},
\eeq 
\end{defn}
where the space $\Sigma_{00}$ is defined in \eqref{def_S0_S00}.
We have $\Sigma_{fg} = \Sigma_{00}+\sigma_{fg}$, where $\sigma_{fg}\in \Sigma_{fg}$ is a given function. 
Similar to the primal problem, in order to develop RBM, we introduce the following equivalent form with the same trial and test spaces.
\begin{defn}
[\bf (EWD) Equivalent Weak Dual Problem of \eqref{ebvp}]
\beq \label{ed_weak_2}
\mbox{Find  } \sigma=\sigma_{00}+\sigma_{fg} \mbox{ with } \sigma_{00}\in \Sigma_{00} \mbox{ satisfying}: \;
b(\sigma_{00},\tau) = \psi(\tau), \; \forall \tau\;\in \Sigma_{00}, 
\eeq 
where the linear form is defined as
\beq \label{def_psi}
\psi(\tau)= -(\tr_1(\tau), g_D)_{\Gamma_D}-b(\sigma_{fg}, \tau), \quad \forall \tau\;\in \Sigma. 
\eeq
\end{defn}
The following coercivity and continuity of $b$ with constant being one are true:
\beq \label{one_con_coer_dual}
\Rnorm{\tau}^2 = b(\tau,\tau)\and b(\rho,\tau) \leq \Rnorm{\rho}\Rnorm{\tau}, \quad \forall \rho, \tau \in \Sigma_{00}.
\eeq
The equivalence of the Problems (D) and (WD) is also standard. The existence and uniqueness of the Problem (EWD) (thus (WD) and (D)) follow from \eqref{one_con_coer_dual} and the Lax-Milgram theorem. 

Since the dual problem (D) is a constraint minimization problem, we can also write its weak form as an equivalent saddle-point problem (or a mixed method, see \cite{BBF:13,Arnold:90}). Let 
\beq
c((\rho, w), (\tau,v))= (\cA^{-1} \rho, \tau) + (\Lambda^t \tau, w) +(\Lambda^t \rho, v)\  \mbox{ for  } (\rho, w), (\tau,v) \in \Sigma\times L^2(\O).
\eeq
Like before, we have an equivalent version with the same trial and test spaces. 
\begin{defn}
[\bf (ESDD) Equivalent Saddle-Point Dual Problem of \eqref{ebvp}]
Find 
$\sigma = \sigma_{00}+\sigma_{fg}$, such that $(\sigma_{00},w) \in \Sigma_0\times L^2(\O)$ satisfying
\beq \label{ed_weak_4}
c((\sigma_{00},w), (\tau,v)) = \psi(\tau) \quad \forall (\tau,v) \in \Sigma_0\times L^2(\O).
\eeq
\end{defn}
Choose $v=0$ in \eqref{ed_weak_4}, we get $\Lambda^t \sigma_{00}=0$ in the strong sense (since $\Lambda^t \Sigma\subset L^2(\O)$ by the definition of $\Sigma$). Combined with the fact $\sigma_{00}\in \Sigma_0$, we have $\sigma_{00}\in \Sigma_{00}$.

With the assumption \eqref{infsup}, the existence, uniqueness, and stability of the formulation (ESSD) are the results of the standard Babuska-Brezzi theory, see \cite{BBF:13}. Also, the solution $\sigma$ of Problem (ESDD) is the solution $\sigma$ of Problem (WD), and vice versa.

\subsection{Error relations and generalized Prager-Synge identity}
Defined the combined norm for the primal-dual problem:
\beq
\cnorm{(v,\tau)} = \sqrt{\tri v \tri^2 + \Rnorm{\tau}^2}, \quad v\in V_0, \tau\in \Sigma_{00}.
\eeq
We then have the following important {\em combined error relation} or {\em generalized Prager-Synge identity} \eqref{PS3} (see also \cite{PS:47,Synge:57,BS:08,CZ:12}), which is Proposition \ref{error} for this specific problem. The energy differences are now squares of respective error in energy norms.
\begin{thm} Let $u$ and $\sigma$ be the solutions of the prime problem (P) \eqref{pphi_elliptic} and the dual problem (D) \eqref{dphi_elliptic}, respectively. Then for any $v\in V_g$ and $\tau \in \Sigma_{fg}$, the following error relations hold:
\begin{eqnarray}
\label{PS1} (primal)\;
\tri u-v \tri^2  &=& 2(\Jp(v)-\Jp(u)) \leq 2 (\Jp(v) -\Jd(\tau)), \\
\label{PS2} (dual)\;
\Rnorm{\sigma-\tau}^2 &=& 2(\Jd(\sigma)-\Jd(\tau)) \leq 2 (\Jp(v) -\Jd(\tau)),\\
\label{PS3} (combined)\;
\cnorm{(u-v,\sigma-\tau)}^2
&=& 2 (\Jp(v) -\Jd(\tau))= \|\cA^{1/2}\Lambda v+\cA^{-1/2}\tau\|_0^2.
\end{eqnarray}
\end{thm}
\begin{proof}
In the sprit of Proposition \ref{error}, we only need to prove $\tri u-v \tri^2  = 2(\Jp(v)-\Jp(u))$, $\Rnorm{\sigma-\tau}^2 = 2(\Jd(\sigma)-\Jd(\tau))$, and $2 (\Jp(v) -\Jd(\tau))= \|\cA^{1/2}\Lambda v+\cA^{-1/2}\tau\|_0^2$.

By the definition of $\Jp$, we have
\begin{eqnarray*}
2(\Jp(v)-\Jp(u)) &=& \|\cA^{1/2} \Lambda v\|_0^2- \|\cA^{1/2} \Lambda u\|_0^2
	- 	2(f,v-u) + 2(g_N, \tr_0(v-u))_{\Gamma_N}\\[1mm]
&=& (\cA \Lambda(u+v), \Lambda (v-u)) - 2(f,v-u)+ 2(g_N, \tr_0(v-u))_{\Gamma_N}.
\end{eqnarray*}
By \eqref{ep_weak}, $(\cA \Lambda u,\Lambda (v-u)) =(f,v-u) -(g_N, \tr_0(v-u))_{\Gamma_N}$, for all $v-u\in V_0$,  then
$$ 
\tri u-v \tri^2  = (\cA \Lambda(u+v),\Lambda (v-u)) - 2(\cA \Lambda u,\Lambda (v-u))
= 2(\Jp(v)-\Jp(u)).
$$
By the definition of $\Jd$, we have
\begin{eqnarray*}
2(\Jd(\sigma)-\Jd(\tau)) = \|\cA^{-1/2} \tau \|_0^2- \|\cA^{-1/2} \sigma\|_0^2 + 	2(\tr_1(\tau-\sigma), g_D)_{\Gamma_D}.
\end{eqnarray*}
By \eqref{ed_weak},
$(\cA^{-1} \sigma, (\tau-\sigma)) = -(\tr_1(\tau-\sigma), g_D)_{\Gamma_D}$, for all $\tau-\sigma \in \Sigma_{00}$, then
$$
\Rnorm{\sigma-\tau}^2 = (\cA^{-1}(\tau+\sigma), (\tau-\sigma)) - 2(\cA^{-1} \sigma, (\tau-\sigma))
= 2(\Jd(\sigma)-\Jd(\tau)).
$$
The second identity in \eqref{PS3} can be proved by the integration by parts,
\begin{eqnarray*}
 2(\Jp(v)-\Jd(\tau)) &=& \tri v \tri^2 +\Rnorm{\tau}^2 - 	2(f,v) + 2(g_N, \tr_0(v))_{\Gamma_N} + 	2(\tr_1(\tau), g_D)_{\Gamma_D} \\
 &=& \tri v \tri^2 +\Rnorm{\tau}^2 + 2 (\Lambda v, \tau) = \|\cA^{1/2}\Lambda v+\cA^{-1/2}\tau\|_0^2.
\end{eqnarray*}
\end{proof}

%


\section{FEMs for Symmetric Coercive Linear Elliptic Mixed Boundary Value Problem}
\setcounter{equation}{0}

Let $V_h$ be a FE subspace of $V$, $D_h$ be a FE subspace of $L^2(\O)$, $\Sigma_h$ be a FE subspace of $\Sigma$ satisfying $\Lambda^t \Sigma_h = D_h$. We assume the pair $\Sigma_h$ and $D_h$ also satisfies a discrete inf-sup condition. Examples of such FE spaces can be found in Section 6.
To simplify our presentation, we make some assumption of the data.

\begin{asm}\label{asss}
We assume that $g_D$ is in the range of $\tr_0$ on $V_h$, $g_N$ is in the range of $\tr_1$ on $\Sigma_h$, and $f$ is in $D_h$.
\end{asm}
\begin{rem} \label{dataH1}
For the case that the boundary data and right-hand side not satisfying Assumption \ref{asss}, we can replace $g_D$, $g_N$, and $f$ by their $L^2$-projections  (or interpolations)  $g_{D,h}$, $g_{N,h}$, and $f_h$ in the corresponding discrete spaces. Since the PDE and the discrete problems are stable, it is easy to show that the difference of the solutions with the original data and the projected/interpolated data are in the same order of the data perturbations, and can be neglected if the mesh is reasonable fine, or can be added to the error estimator as an extra term, see p.181 of \cite{Braess:07} for the discussion on $f$ in mixed formulations and \cite{BCD:04} for analysis on approximations of inhomogeneous Dirichlet boundary conditions. This issue is often not discussed in the RBM since the classic RBM often assumes that the FE solution is a very accurate "truth" solution and these small errors are not worried about.
\end{rem}

Since basis functions of RBM can only be homogeneous functions (in the sense that boundary condition is zero and the constraint condition is homogeneous), we use two equivalent versions of problems (EWP) and (ESDD) with the same trial and dual spaces to find FE solutions.

Let $V_{0,h} = V_h\cap V_0$. We have $u_g\in V_h$ due to Assumption \ref{asss}.
\begin{defn}[\bf FEM for the primal problem (P)] \label{dfn_pfem}
The FEM of the primal problem \eqref{pphi_elliptic} is: 
\beq \label{ep_weak2_fem}
\mbox{Find  } u_h=u_{0,h}+u_g \mbox{ with } u_{0,h}\in V_{0,h} \mbox{ satisfying}: \;
a(u_{0,h},v) = \varphi(v), \; \forall v\;\in V_{0,h}.
\eeq
\end{defn} 
By the standard arguments of FE analysis (for example, Theorem 0.3.3 of \cite{BrSc:08}), \eqref{one_con_coer}, and the fact $u_h=u_{0,h}+u_g$ with $u_{0,h}\in V_{0,h}$, we have the following identity:
\beq \label{ap_primalfem2}
\tri u - u_h\tri = \inf_{v_{0,h}\in V_{0,h}} \tri u-u_g- v_{0,h}\tri.
\eeq
For the FE approximation of the dual problem, the saddle-point formulation is used.
Let $\Sigma_{0,h} = \Sigma_{h} \cap \Sigma_{0}$ and $\Sigma_{00,h} = \Sigma_{h} \cap \Sigma_{00}$. We have $\sigma_{fg}\in \Sigma_h$ due to Assumption \ref{asss}.


\begin{defn}[\bf FEM for the dual problem (D)]
The FEM of the dual problem \eqref{dphi_elliptic} is:
Find $\sigma_h = \sigma_{00,h}+\sigma_{fg}$ with $(\sigma_{00,h},w_h) \in \Sigma_{0,h}\times D_h$ satisfying
\beq \label{fem_d}
c((\sigma_{00,h},w_h), (\tau,v)) = \psi(\tau) \quad \forall (\tau,v) \in \Sigma_{0,h}\times D_h.
\eeq 
\end{defn}
Let $\tau=0$ in \eqref{fem_d} and by the assumption that $\Lambda^t \Sigma_h = D_h$, we have $\Lambda^t\sigma_{00,h} = 0$ in the strong sense, thus 
\beq \label{in00}
\sigma_{00,h} \in \Sigma_{00} \and \sigma_h = \sigma_{00,h}+\sigma_{fg}\in \Sigma_{fg}.
\eeq
By choosing $\tau\in \Sigma_{00}$ and $v=0$ in \eqref{fem_d}, it is easy to check that the solution $\sigma_{00,h}$ of \eqref{fem_d} is also the solution of the following discrete (EWD) problem:
\beq 
b(\sigma_{00,h},\tau) = \psi(\tau), \quad \forall \tau\;\in \Sigma_{00,h}. 
\eeq 
Similarly, we have the best approximation property of the dual FE solution with constant being one due to \eqref{one_con_coer_dual}:
\beq \label{ap_dualfem}
\Rnorm{\sigma - \sigma_h} = \inf_{\tau_{00,h}\in \Sigma_{00,h}} \Rnorm{ \sigma-\sigma_{fg}- \tau_{00,h}}.
\eeq
A standard FEM often tries to find $\sigma_h = \sigma_{0,h}+\sigma_g$, where $\sigma_g$ satisfies the boundary condition on $\Gamma_N$ but not the dual- feasibility condition, and $\sigma_{0,h}$ satisfies zero boundary condition on $\Gamma_N$. We choose to find $\sigma_{00,h}$ here directly due to this is the requirement of a RB basis function. An efficient construction of $\sigma_{fg}$ can be found in Section 5.

\subsection{Combined FE error estimate}
Viewing the primal-dual problem as one combined problem, from \eqref{ap_primalfem2} and \eqref{ap_dualfem}, we have the following best approximation in the combined norm with constant being one:
\beq \label{combined_apriori}
\cnorm{(u-u_h,\sigma-\sigma_h)}= \inf_{(v_{0,h},\tau_{00,h})\in V_{0,h}\times \Sigma_{00,h}}\cnorm{(u-u_g-v_{0,h},\sigma-\sigma_{fg}- \tau_{00,h})}
\eeq
This also suggests that we should choose the approximation order properly for the primal and dual FE spaces so that they have the same order of convergence.
 
Define the PD-FE global a posteriori error estimator and local indicator, respectively: 
\begin{eqnarray*}
\eta_{h}(u_h,\sigma_h) &=& \sqrt{2(\Jp(u_h) - \Jd(\sigma_h)} = \|\cA^{1/2}\Lambda u_h+\cA^{-1/2}\sigma_h\|_0,\\[2mm]
\eta_{h,T}(u_h,\bsigma_h) &=&\|\cA^{1/2}\Lambda u_h+\cA^{-1/2}\sigma_h\|_{0,T}, \quad\forall T\in \cT.
\end{eqnarray*}
By Theorem 3.9, $\eta_h$ is exact (reliable, efficient, and robust with constant being one) with respect to the error in the combined energy norm, also, $\eta_h$ can be used as a reliable and robust error estimator for the individual primal and dual FE problems.

\section{PD-RBM for the Parametric Symmetric Coercive  Linear Elliptic Mixed Boundary Value Problem}
\setcounter{equation}{0}

In this section, we develop PD-RBM for the following {\bf parametric symmetric  coercive linear elliptic mixed boundary value problem}:
\beq
\label{p_ebvp}
\left\{
\begin{array}{lllll}
\Lambda ^t(\cA(\mu) \Lambda u(\mu)) = f(\mu) \mbox{ in }\O,\\[2mm]
\tr_0(u(\mu)) = g_D(\mu) \mbox{ on }\Gamma_D\quad \mbox{and}\quad
-\tr_1(\cA(\mu)\Lambda u(\mu)) = g_N(\mu) \mbox{ on }\Gamma_N,
\end{array}
\right.
\eeq
where $\mu\in \cD$. $\cD$ is a parametric domain.
In this paper, we assume that $\cA$, $f$, $g_D$, and $g_N$ are $\mu$-dependent. Thus, we use $\mu$-dependent forms, functional, norms, and problems. The meanings of these notations are self-explicit.
We also suppose that a set of $N$ parameters $S_N =\{\mu_i\}_{i=1}^N$ is chosen.

\subsection{RBM for the primal problem (P)}
The RB space for the primal problem (P) is the span of the corresponding homogeneous FE solutions (by the primal FEM in Definition \ref{dfn_pfem}): $
V_{\rb}^N = \mbox{span}\{ u_{0,h}(\mu_i)\}_{i=1}^N \subset V_{0,h},
$
\beq 
u_{0,h}(\mu_i)\in V_{0,h} \mbox{ satisfying}: \quad
a(u_{0,h},v;\mu_i) = \varphi(v;\mu_i), \quad \forall v\;\in V_{0,h}.
\eeq
Assume that a $u_g(\mu) \in V_g(\mu)$ can be constructed efficiently (see Section 5.3), then the primal RB problem is: for any given $\mu \in \cD$, find $u_{\rb}^N(\mu) 
= u_{0,\rb}^N(\mu) + u_g(\mu)$ with $u_{0,\rb}^N(\mu) \in V_{\rb}^N $ satisfying
\beq \label{ep_weak2_rb}
	a(u_{0,\rb}^N(\mu), v_{\rb};\mu) = \varphi(v_{\rb};\mu),\quad \forall v_{\rb} \in V_{\rb}^N.
\eeq
Similar to the FE case, we have the following best approximation estimate with constant being one for any given $\mu\in \cD$ with respect to the $\mu$-dependent energy norm for the primal RB approximation:
\beq\label{ap_primalrb}
\tri u(\mu) - u_{\rb}^N(\mu)\tri_{\mu} = \inf_{v_{\rb}\in V_{\rb}^N} \tri u(\mu) -u_g(\mu)- v_{\rb}\tri_{\mu}.
\eeq

\subsection{RBM for the dual problem (D)}
The RB space for the dual problem (D) is the span of the corresponding homogeneous FE solutions: $\Sigma_{\rb}^N = \mbox{span}\{\sigma_{00,h}(\mu_i)\}_{i=1}^N \subset \Sigma_{00,h}$, where 
$(\sigma_{00,h}(\mu_i),w_h(\mu_i)) \in \Sigma_{0,h}\times D_h$ satisfying
\beq
c((\sigma_{00,h},,w_h(\mu_i)), (\tau,v);\mu_i) = \psi(\tau;\mu_i) \quad \forall (\tau,v) \in \Sigma_{0,h}\times D_h.
\eeq
The solution $\sigma_{00,h}(\mu_i)\in \Sigma_{00,h}\subset \Sigma_{0,h}$, due to \eqref{in00}. 

Assume that a function $\sigma_{fg}(\mu) \in \Sigma_{fg}(\mu)$ is constructed (see Section 5.3), the dual RB approximation is: for any given $\mu \in \cD$, find $\sigma_{\rb}^N(\mu) 
= \sigma_{00,\rb}^N(\mu) +\sigma_{fg}(\mu)$ with $\sigma_{00,\rb}^N(\mu) \in \Sigma_{\rb}^N $ satisfying
\beq \label{rbm_dual}
	b(\sigma_{00,\rb}^N(\mu), \tau_{\rb};\mu) = \psi(\tau_{\rb};\mu),\quad \forall \tau_{\rb} \in \Sigma_{\rb}^N.
\eeq 
With respect to the $\mu$-dependent energy norm, we have the following best approximation estimate with constant being one for any given $\mu\in \cD$:
\beq \label{ap_dualrb}
\Rnorm{\sigma(\mu) - \sigma_{\rb}^N(\mu)}_{\mu} = \inf_{\tau_{00,\rb}\in \Sigma_{\rb}^N} \Rnorm{\sigma(\mu) - \sigma_{fg}(\mu)- \tau_{00,\rb}}_{\mu} .
\eeq
\begin{rem}
It should be emphasized that for the dual FE problem, we need to use the saddle point formulation (ESDD), since a FE subspace of $\Sigma_{00}$ is often hard, if not impossible, to be constructed. On the other hand, for the dual RB problem, the basis function are already in $\Sigma_{00}$ (the fact that the solution of \eqref{fem_d} satisfies \eqref{in00} in the strong sense is  discussed in p.182 of \cite{Braess:07} for the hypercircle error estimator and is well-known for the mixed FEM), thus we should use the simpler formulation (EWD). 
\end{rem}

\begin{rem}
In \cite{Yano:15}, the dual solution is not solved as a global FE problem (i.e., a global optimization problem in the corresponding FE space), but is recovered though a local post-processing procedure. In our algorithm, we suggest to solve the dual FE approximation directly. This is due to three reasons: the first reason is that the quality of the a posteriori error estimator is based on how tight the primal-dual gap is. To make the primal-dual gap tight, we should solve the dual problem globally for both FEM and RBM. Also, to ensure the error estimates are robust, simple explicit construction often leads to non-robust results, see \cite{Ver:09,CZ:12}. 
The second reason is that for the standard FE a posteriori error estimators, we do not want the error estimator to be as expensive as the original problem. While in the RB scenario, all these FE steps are offline steps. A good quality of the RB set is more important than some computational cost of finding extra solutions of offline problems. Plus, our greedy algorithms developed in the next section will balance the RB tolerance and FE approximation errors to make sure that the FE problem is solved in an optimal mesh and we do not over-compute too many basis functions.  The third reason is that computing both the primal and dual unknowns accurately is very common in the least-squares approach, both in the FEM \cite{BG:09} and in the RBM \cite{Yano:16,Yano:18}.

Our dual RB problem is also different from the dual RB problem for the diffusion equation in \cite{Yano:15} where a saddle-point (constraint minimization) problem is considered. We believe our approach here is more appropriate since our version is based on a standard weighted projection. In fact, it uses the same trick as handling non-homogeneous boundary condition in primal RBM. Also,	our approach here is also easier to show the best approximation properties of the RB space.  Compared to \cite{Yano:18}, the method developed here can be viewed as an optimal case of choosing the penalty parameter $\delta=0$.
\end{rem}

\begin{rem}
In RBM, a Gram-Scmidt orthonormalization is often performed to increase the algebraic stability. Since the  orthonormalization has to be based on a fixed inner product, the parameter-dependent energy inner product can not be used. The condition number of the discrete system is still bounded by the ratio of the continuity and stability constants, as stated in Prop. 2.17 of \cite{Ha:17}. In our numerical experiments, since the number of basis functions is not big, the Gram-Scmidt orthonormalization is not used.
\end{rem}

\subsection{Efficient offline-online decomposition}
To ensure an efficient offline-online decomposition and efficient constructions of $u_g(\mu)$ and $\sigma_{fg}(\mu)$, we need to assume the following parameter separability. 

\begin{asm} [{\bf Parameter separability/Affine decomposition}] We assume the forms $a$ and $b$ to be parameter separable, i.e., there exist coefficient functions $\theta_q^a(\mu): \cD \to \mathbb{R}$ for $q=1,\cdots, Q_a$ and 
$\theta_q^b(\mu): \cD \to \mathbb{R}$ for $q=1,\cdots, Q_b$
with $Q_a$ and $Q_b \in \mathbb{N}$ and parameter-independent continuous bilinear forms
$a_q(\cdot, \cdot):V\times V\to \mathbb{R}$ and 
$b_q(\cdot, \cdot):\Sigma \times \Sigma \to \mathbb{R}$
such that
\begin{eqnarray}
a(w,v;\mu) &=& \sum_{q=1}^{Q_a}\theta^q_a(\mu)a^q(v,w), \quad \mu \in \cD, v,w\in V, \\
\mbox{and}\quad
b(\rho,\tau;\mu) &=& \sum_{q=1}^{Q_b}\theta^q_b(\mu)b^q(\rho,\tau), \quad \mu \in \cD, \rho, \tau\in \Sigma.
\end{eqnarray}
We also assume the boundary data and the right-hand side are parameter separable, 
\beq 
g_D(\mu) = \sum_{q=1}^{Q_{D}}\theta^q_{D}(\mu)g_{D}^q, \quad 
g_N(\mu) = \sum_{q=1}^{Q_{N}}\theta^q_{N}(\mu)g_{N}^q, \quad 
f(\mu) = \sum_{q=1}^{Q_{f}}\theta^q_{f}(\mu)f^{q}, \quad \mu \in \cD,
\eeq
where $\theta^q_{D}$, $\theta^q_{N}$, $\theta^q_{f}$, $g_{D}^q$, $g_{N}^q$, and $f^{q}$  are defined in a similar fashion as the terms for $a$ and $b$.
We assume that $g_{D}^q$ is in the range of $\tr_0$ on $V_h$, $g_{N}^q$ is in the range of $\tr_1$ on $\Sigma_h$, and $f^q$ is $D_h$, i.e., they satisfy Assumption \ref{asss}.
\end{asm}
\begin{rem}
 The assumption of $a$ (the $a$ form is defined in \eqref{aform}) and $b$ (the $b$ form is defined in \eqref{bform}) can also be stated in the parameter separabilities of $\cA(\mu)$ and $\cA^{-1}(\mu)$:
\beq \label{AandinvA}
\cA(\mu) = \sum_{q=1}^{Q_{A}}\theta^q_{A}(\mu)\cA^q, \quad 
\cA^{-1}(\mu) = \sum_{q=1}^{Q_{A^{inv}}}\theta^q_{A^{inv}}(\mu)\cA_{inv}^q, \quad \mu \in \cD,
\eeq
where $\theta^q_{A}(\mu)$ and $\theta^q_{A^{inv}}(\mu)$ are parameter-dependent functions while $\cA^q$ and $\cA_{inv}^q$ are parameter-independent. Then, if we further assume that the integration $(\cA^q v,w)$ for $v$, $w$ in $V_h$ can be exactly integrated, it is $a^q(v,w)$. Similar results are also true for $\cA^{-1}$ and $b$. For a simple case that $\cA(\mu)$ (and thus $\cA^{-1}(\mu)$) is piecewisely defined on the computational domain, \eqref{AandinvA} is clearly satisfied. 
\end{rem}
\begin{rem}
Same as the discussion in Remark \ref{dataH1}, the assumption that the data $g_{D}^q$, $g_{N}^q$, and $f^q$ satisfying the assumption \ref{asss} is only for simplicity and is not essential.
\end{rem}

For a given $\mu\in \cD$, we want to construct a $u_{g}(\mu)\in V_{g}(\mu)$. First, we find a set of parameter-independent  $u^q_D \in V_{h}$, such that 
\beq \label{uqD}
\tr_0(u^q_D) = g_{D}^q,  \quad  q=1, \cdots, Q_D.
\eeq
These $u^q_D$ usually can be explicitly constructed by letting the DOFs inside the domain to be zero while the DOFs on $\G_D$ to be appropriate values to match $g_D^q$, which is a common FE technique. Another common and probably better way to construct a function $u^q_D$ is to solve a fixed parameter version of the corresponding PDE \eqref{p_ebvp} with a zero right-hand side and a zero boundary condition on $\Gamma_N$ which satisfies the nonzero boundary condition \eqref{uqD}. The function $u^q_D$ constructed by the second way is smooth and this ensures that the associated parametric manifold of $V^N_{\rb}$ is smooth and improves the RBM convergence. 

Once this is done, let
\beq
 V_g(\mu) \ni u_g(\mu) = \sum_{q=1}^{Q_{D}}\theta^q_{D}(\mu)u^q_D.
\eeq
With the parameter separability of $f$, $g_N$, $a$, and $u_g$, the linear form $\varphi(v;\mu)$ (see \eqref{def_phi}) is also parameter separable, and the primal RBM has an offline-online decomposition.

For a given $\mu\in \cD$, in order to construct a $\sigma_{fg}(\mu)\in \Sigma_{fg}(\mu)$, a divide and conquer strategy is taken. We plan to find a  $\sigma_{fg}(\mu) \in \Sigma_{fg}(\mu)$ with the decomposition:
$
\sigma_{fg}(\mu) = \sigma_{f0}(\mu) +\sigma_{0g}(\mu),
$
where $\sigma_{f0}(\mu)$ and $\sigma_{0g}(\mu)$ $\in \Sigma_h$, satisfying:
\begin{eqnarray} \label{f0}
\Lambda^t \sigma_{f0}(\mu) = -f(\mu)\mbox{ in }\O, &\quad&  \tr_1(\sigma_{f0}(\mu)) =0  \mbox{ on } \Gamma _N,\\[2mm]
\label{0g}
\and\Lambda^t\sigma_{0g}(\mu) =0\mbox{ in }\O,  && \tr_1(\sigma_{0g}(\mu)) =g_N(\mu)  \mbox{ on } \Gamma _N.
\end{eqnarray}
First, find a set of parameter-independent  $\sigma_{f0}^q \in \Sigma_{h}$, such that 
\beq
\Lambda^t \sigma_{f0}^q = -f^q \mbox{ in } \O \and \tr_1(\sigma_{f0}^q) =0 \mbox{ on }  \Gamma _N,\quad  q=1, \cdots, Q_f.
\eeq
This can be done by finding a FE solution $\sigma_{f0}^q \in \Sigma_{0,h}$, such that
\beq 
c((\sigma_{f0}^q ,w_h), (\tau,v)) = -(f_q,\,v)\quad \forall \,\, (\tau,v)\in \Sigma_{0,h}\times D_h, \quad  q=1, \cdots, Q_f.
\eeq
Then, find a set of parameter-independent  $\sigma_{0g}^q \in \Sigma_{h}$, such that 
\beq
\Lambda^t \sigma_{0g}^q = 0  \in \O \and   \tr_1(\sigma_{0g}^q) =g_N^q  \mbox{ on } \Gamma _N,  \quad q=1, \cdots, Q_N.
\eeq
This can be done by finding a FE solution $\sigma_{0g}^q \in \Sigma_{h}$,  $\tr_1(\sigma_{0g}^q) =g_N^q$  on $\Gamma _N$, such that
\beq 
c((\sigma_{0g}^q ,w_h), (\tau,v)) = 0\quad \forall \,\, (\tau,v)\in \Sigma_{0,h}\times D_h, \quad  q=1, \cdots, Q_N.
\eeq
Let
$
\sigma_{f0}(\mu) = \sum_{q=1}^{Q_{f}}\theta^q_{f}(\mu)\sigma^q_{f0}$
and
$
\sigma_{0g}(\mu) = \sum_{q=1}^{Q_{N}}\theta^q_{N}(\mu)\sigma^q_{0g},
$
then $\sigma_{f0}(\mu)$ satisfies \eqref{f0} and $\sigma_{0g}(\mu)$ satisfies \eqref{0g}.  The sum $\sigma_{fg}(\mu) = \sigma_{f0}(\mu) +\sigma_{0g}(\mu)$ is in $\Sigma_{fg}(\mu)$ and it is parameter separable. Together with the assumption that $g_D$ and $b$ are parameter separable,  the dual RB problem has an offline-online decomposition.

Note that all online computations only depend on the dimensions of RB basis $N$, and various $Q$'s, but not depend on the FEM's number of DOFs. The computations of online operation counts in a similar setting can be found  in \cite{Yano:15,Yano:16}, we skip them here for simplicity.

\subsection{Combined RB error estimate}
If we view the primal-dual problem as one combined problem, for any fixed $\mu\in \cD$, we have the same best approximation result in the combined $\mu$-dependent norm as \eqref{combined_apriori}.

Define the PD-RB a posteriori error estimator: 
$$
\eta_{\rb}(u_{\rb}^N,\sigma_{\rb}^N;\mu) = \sqrt{2(\Jp(u_{\rb}^N(\mu);\mu) - \Jd(\sigma_{\rb}^N(\mu);\mu))} = \|\cA^{1/2}\Lambda u_{\rb}^N(\mu)+\cA^{-1/2}\sigma_{\rb}^N(\mu)\|_0.
$$
By \eqref{PS3}, it is an exact error estimator for the combined reduced basis energy error
since 
\beq \label{exact_rb}
\eta_{\rb}(u_{\rb}^N,\sigma_{\rb}^N;\mu) = \cnorm{(u(\mu)-u_{\rb}^N(\mu) , \bsigma(\mu)-\bsigma_{\rb}^N(\mu))}_{\mu}.
\eeq
With Assumption 5.4 on the parameter separability, the primal and dual functionals are also parameter separable, thus the online evaluation of $\eta_{\rb}(u_{\rb}^N,\sigma_{\rb}^N;\mu)^2$ also only depends on $N$ and various $Q$'s, see also \cite{Yano:15}.

\subsection{Relation between FE and RB errors}

By the best approximation properties \eqref{ap_primalfem2},\eqref{ap_dualfem}, \eqref{ap_primalrb}, and \eqref{ap_dualrb} (note that they are all identities) and the facts that the RB spaces are subspaces of the corresponding FE spaces, $V_{\rb}^N \subset V_{0,h}$ and  $\Sigma_{\rb}^N \subset \Sigma_{00,h}$, we have the following comparison theorem between FE and RB solutions.

\begin{thm}[Comparison of RB and FE errors] \label{comp_fe_rb}
For a given $\mu\in \cD$, with the same data, let $u_{h}(\mu)$ and $u_{\rb}^N(\mu)$ be the primal FE  and RB solutions of \eqref{ep_weak2_fem} and \eqref{ep_weak2_rb},  and $\sigma_{h}(\mu)$ and $\sigma_{\rb}^N(\mu)$ be the dual FE  and RB solutions of \eqref{fem_d} and \eqref{rbm_dual}, respectively, then we have the following bounds:
\begin{eqnarray} \label{com1}
\tri u(\mu) - u_{h}(\mu)\tri_{\mu} &\leq& \tri u(\mu) - u_{\rb}^N(\mu)\tri_{\mu}, \\
\Rnorm{\sigma(\mu) - \sigma_h(\mu)}_{\mu}  &\leq& \Rnorm{\sigma(\mu) - \sigma_{\rb}^N(\mu)}_{\mu}, \\
 \cnorm{(u(\mu)-u_h(\mu) , \bsigma(\mu)-\bsigma_h(\mu))}_{\mu} &\leq& 
  \cnorm{(u(\mu)-u_{\rb}^N(\mu) , \bsigma(\mu)-\bsigma_{\rb}^N(\mu))}_{\mu}.
\end{eqnarray}
For the special case that $\mu\in S_N$, the above inequalities become equalities.
\end{thm}
\begin{rem}
For the case that $\mu\in S_N$, we have the reproduction of solutions (Prop. 2.21 of \cite{Ha:17}), the RB solutions are the FE solutions.

The above results show that the true RB errors measured in energy norms are always worse than (or equal to) their corresponding FE errors. Note that, since the primal RBM is essentially the standard RBM, so \eqref{com1} is also true for the standard primal RBM. The difference is that in the standard RBM, the error estimator is not designed to measure this specific error. On the other hand, in the PD-RBM, the combined RB energy error can be explicitly and exactly computed by the primal-gap gap a posteriori error estimator \eqref{exact_rb}. This sets up the foundation of our balanced greedy algorithms to determine the RB stopping criteria.

In principle, for the true RB and FE errors in the non primal-dual setting, we should expect the similar results, that the RB errors are always worse than (or equal to) their corresponding FE errors for the same $\mu$. But the constant may not being one, and it may depend on parameters and be non-robust with respect to parameters.
\end{rem}

\section{Several Examples of Primal-Dual Variational Problems}
\setcounter{equation}{0}
In this section, we first present several examples of the symmetric coercive linear elliptic mixed boundary value problem, their corresponding primal and dual problems, Lagrangians, and error relations. Then we discuss some examples beyond the linear elliptic boundary value problems.

First, we define some function spaces and sets. Let
$
	H^1_{D,g}(\O)= \{v\in H^1(\O):\, v= g_{D} \; \text{ on }\, \G_D \}$ 
	and
	$H^1_{D}(\O) = \{v\in H^1(\O):\, v= 0 \; \text{ on }\, \G_D \}$.
Define $H(\divvr;\Omega)= \{ \btau\in L^2(\O)^d : \gradt\btau \in L^2(\O)\}$, then let 
$
H_{N,g}(\divvr;\Omega)= \{ \btau \in H(\divvr;\O): \btau \cdot \bn = g_N  \mbox{ on }\G_N\}
$ and $H_{N}(\divvr;\Omega)= \{ \btau \in H(\divvr;\O): \btau \cdot \bn = 0  \mbox{ on }\G_N\}$,
where $\bn$ is the unit outer normal on $\p\O$.
We assume that $A(x)$ is a symmetric positive definite matrix and $0<\kappa(x)\in L^{\infty}(\O)$. 

In the following subsections, we use the subscripts $d$, $rd$, and $le$ to denote the functionals and norms are for the diffusion problem, the reaction-diffusion problem, and the linear elasticity problem, respectively.

\subsection{Diffusion Problem}

Consider the diffusion problem:
\begin{equation}
	-\gradt(A \nabla u) = f  \text{ in }\, \O,	\quad
			u = g_D  \text{ on }\, \G_D,		\quad
	-A \nabla u \cdot \bn  = g_{N} 	\text{ on }\, \G_N.
\end{equation}
For this problem, let $V= H^1(\O)$,
$V_0 = H^1_D(\O)$, $\Lambda  = \nabla$,  $\Lambda^t  = -\gradt $, and $\cA =A$, thus $\cA\Lambda v =  A\nabla v$, $\Sigma = H(\divvr;\O)$. The trace $\tr_1(\btau)$ for $\btau\in H(\divvr;\O)$ is $\btau\cdot\bn$ on $\p\O$, and the trace $\tr_0$ for $v\in H^1(\O)$ is the standard trace. Then we have the following integration by parts formula (corresponding to \eqref{ibp}):
\beq \label{ibp2}
(\btau, \nabla v)_\O + (\gradt \btau, v)_\O = (\btau\cdot\bn, v)_{\p\O}, \quad \forall (v, \tau)\in H^1(\O)\times H(\divvr;\O).
\eeq
The weak problem is:
\beq
\mbox{Find } u\in H_{D,g}^1(\O): \quad (A\nabla u,\nabla v) = (f,v)- (g_{N}, v)_{\G_N}, \quad \forall v\in H_D^1(\O).
\eeq
The primal variational problem is:
\begin{eqnarray} \label{pphi_d}
\mbox{Find } u\in H^1_{D,g}(\O): \;
\Jp_{d}(u) = \inf_{v\in H^1_{D,g}(\O)} \Jp_{d}(v), \\ \nonumber
\Jp_{d}(v)= \frac{1}{2}\|A^{1/2}\nabla v\|_0^2 - (f,v) + (g_{N}, v)_{\G_N}.
\end{eqnarray} 
To derive the dual problem (D), we set the constraint set $\Sigma_{fg}$ as
\beq
 H_{N,g}(\divvr;\O;f) = \{\btau \in H(\divvr;\O): \gradt \btau = f\and \btau\cdot \bn =g_N \mbox{ on } \Gamma_N\}.
\eeq
The space $\Sigma_{00}$, which is $H_{N}(\divvr;\O;0)$ in the diffusion equation case, can de defined accordingly.
The dual problem is:
\begin{eqnarray} \label{dphi_d}
\mbox{Find } \bsigma\in H_{N,g}(\divvr;\O;f):  \quad
\Jd_d(\bsigma) = \inf_{\btau\in H_{N,g}(\divvr;\O;f)} \Jd_d(\btau),\\ \nonumber
	\Jd_d(\btau) = - \frac{1}{2} (A^{-1}\btau,\btau) -(\btau\cdot\bn, g_D)_{\Gamma_D}.
\end{eqnarray}
The energy norms are
\beq
\tri v \tri_d = \|A^{1/2}\nabla v\|_0 \mbox{ for } v\in H^1_D(\O) 
\quad\mbox{and}\quad
\Rnorm{\btau}_d = \|A^{-1/2}\btau\|_0  \mbox{ for }\btau \in H_{N}(\divvr;\O;0). 
\eeq
The relation \eqref{relation} for the diffusion equation is:
\beq
\mbox{(extremal relation)}\quad 
\bsigma = -A\nabla u.
\eeq
The combined error relation (the Prager-Synge identity) \eqref{PS3} is: For all $v\in H_{D,g}(\O)$ and $\btau\in H_{N,g}(\divvr;\O;f)$,
\beq \label{PS}
\|A^{1/2}\nabla (u-v)\|_0^2 + \|A^{-1/2}(\bsigma-\btau)\|_0^2 
= \|A^{1/2}\nabla v+ A^{-1/2}\btau\|_0^2.
\eeq
We can also derive the prime-dual problems by setting the Lagrangian as in \eqref{lag}:
\beq
\cL_d(v,\tau) = \Jp_d(v) - \frac{1}{2} \|A^{1/2} \nabla v+ A^{-1/2} \btau\|_0^2,
\quad (v,\btau)\in H_{D,g}^1(\O) \times H(\divvr;\O).
\eeq
The saddle-point relation \eqref{saddlepoint} (also the strong duality) is:
\beq
 \inf_{v\in H^1_{D,g}(\O)} \Jp_{d}(v) =
\Jp_d(u) = \cL_{d}(u,\bsigma) = \Jd_d(\bsigma) = \inf_{\btau\in H_{N,g}(\divvr;\O;f)} \Jd_d(\btau).
\eeq
Let $k\geq 1$ be an integer, 
$S_k\subset H^1(\O)$, the global continuous degree-$k$ piecewise polynomial space, and $RT_{k-1}\subset H(\divvr;\O)$, the Raviart-Thomas space with index $k-1$, constitute a good pair of FE approximation spaces for $H^1(\O)$ and $H(\divvr;\O)$. Standard FE convergence results can be found in \cite{Ciarlet:78,BBF:13}. Robust a priori FE approximation results for the primal and dual FE problems can be found in \cite{CCZ:20,zhang:20mixed}. Robust a posteriori FE error estimators based on the Prager-Synge identity (the so-called equilibrated error estimator, which is a primal-dual gap error estimator) can be found in \cite{BS:08,CZ:12,CCZ:20} for the primal problem and in \cite{CCZ:20mixed} for the dual mixed problem.

%

\subsection{Reaction-Diffusion Problem}
Consider the reaction-diffusion equation:
\begin{equation} \label{rd_equation}
	-\gradt(A \nabla u) + \kappa^2 u = f  \text{ in }\, \O,	\quad
			u = g_D  \text{ on }\, \G_D,		\quad
	-A \nabla u \cdot \bn  = g_{N} 	\text{ on }\, \G_N.
\end{equation}
For this problem, let $V= H^1(\O)$, $V_0 = H^1_D(\O)$, $\Lambda v = [\nabla v, v]$ and $\cA =[A,\kappa^2]$, thus $\cA\Lambda v = [A\nabla v, \kappa^2 v]$. Let $\tau = [\btau, w]$, then $\Lambda^t [\btau, w] = -\gradt\btau+w$, and $\Sigma= H(\divvr;\O)\times L^2(\O)$.
The trace $\tr_1([\btau, w]) = \btau\cdot\bn$ on $\p\O$, and the trace $\tr_0$ for $v\in H^1(\O)$ is the standard.
Let $\tau = [\btau, w] \in \Sigma =H(\divvr;\O)\times L^2(\O) $ in the integration by parts formula \eqref{ibp}, we have
$$ 
(\btau, \nabla v)+(w,v) +(\gradt \btau, v) - (w,v) = (\btau\cdot\bn, v)_{\p\O}, \quad \forall (v, \btau)\in H^1(\O)\times H(\divvr;\O), w\in L^2(\O),
$$
which is exactly \eqref{ibp2}.
The weak primal problem is:
\beq \label{pphi_rd_weak}
\mbox{Find } u\in H_{D,g}^1(\O): \; (A\nabla u,\nabla v)+(\kappa^2 u, v) = (f,v) - (g_{N}, v)_{\G_N},  \quad \forall v\in H_D^1(\O).
\eeq
The primal problem is:
\begin{eqnarray} \label{pphi_rd}
&&\mbox{Find } u\in H_{D,g}^1(\O):  \quad
\Jp_{rd}(u) = \inf_{v\in H_{D,g}^1(\O)} \Jp_{rd}(v), \\ \nonumber
 && \Jp_{rd}(v)= ((A\nabla v,\nabla v)+(\kappa^2v,v))/2 - (f,v)+(g_{N}, v)_{\G_N}.
\end{eqnarray}
To derive the dual problem, we first check the conditions in $\Sigma_{fg}$. 
The condition $\Lambda^t [\btau, w] = -\gradt\btau+w = -f$, we get the constraint set $\Sigma_{fg}$ is 
$$
\{\btau \in H(\divvr;\O), w\in L^2(\O): w -\gradt\btau = -f, \btau\cdot\bn =g_N \mbox{ on }\G_N\}.
$$
The dual functional is 
$
	- \frac{1}{2} (A^{-1}\btau,\btau) 
	-\frac{1}{2}(\kappa^{-2} w,w)
	-(\btau\cdot\bn, g_D)_{\Gamma_D}.
$ 
We can eliminate $w$ by enforcing the condition $w = \gradt\btau  -f$. Then the $\btau$-only dual functional is 
\beq \label{rd_d}
\Jd_{rd}(\btau)=  
- \frac{1}{2} (A^{-1}\btau,\btau) 
	-\frac{1}{2}(\kappa^{-2} (\gradt\btau  -f),\gradt\btau  -f)
	-(\btau\cdot\bn, g_D)_{\Gamma_D}.
\eeq
So the dual problem is:
\beq \label{dphi_rd}
\mbox{Find } \bsigma\in H_{N,g}(\divvr;\O):\quad
\Jd_d(\bsigma) = \inf_{\btau\in H_{N,g}(\divvr;\O)} \Jd_{rd}(\btau).
\eeq
Note that in this formulation, the dual-feasibility condition is removed, and there is no need to use the saddle-point formulation (SDD) or (ESDD). 
The weak problem of \eqref{rd_d} is: Find $\bsigma\in H_{N,g}(\divvr;\O)$ s.t.
\beq  \label{dphi_rd_weak}
(A^{-1}\bsigma,\btau)+(\kappa^{-2} \gradt \bsigma, \gradt \btau) = (\kappa^{-2} f, \gradt \btau)-(\btau\cdot\bn, g_D)_{\Gamma_D}, \quad \forall \btau\in H_{N}(\divvr;\O).
\eeq
The energy norms are
\begin{eqnarray*}
\tri v \tri_{rd} &=& (\|A^{1/2}\nabla v\|_0^2+\|\kappa v\|_0^2)^{1/2} \quad v\in H^1_D(\O), \\ [2mm]
\and
\Rnorm{\btau}_{rd} &=& (\|A^{-1/2}\btau\|_0+\|\kappa^{-1} \gradt\btau\|_0^2 )^{1/2}\quad \btau \in H_{N}(\divvr;\O). 
\end{eqnarray*}
The relation \eqref{relation} for the reaction-diffusion equation is:
$
[\bsigma, \gradt \bsigma-f] = -\cA\Lambda u = [-A \nabla u, -\kappa^2 u].
$
That is,
\beq \label{relation_rd}
\mbox{(extremal relation)}\quad 
\bsigma = -A\nabla u
\and
\gradt \bsigma + \kappa^2 u= f.
\eeq
Note that \eqref{relation_rd} is just the first-order system reformulation of \eqref{rd_equation}.

Substitute the definitions of $\tau$, $\Lambda$, and $\cA$ into \eqref{PS3},  the combined error relation is: For all $v\in H_{D,g}(\O)$ and $\btau\in H_{N,g}(\divvr;\O)$,
\begin{eqnarray*}
\tri u-v\tri_{rd}^2 + \Rnorm{\bsigma-\btau}_{rd}^2
 = \|A^{1/2}\nabla v+ A^{-1/2}\btau\|_0^2 +
 \|\kappa^{-1}\gradt(\bsigma-\btau)+\kappa (u-v)\|_0^2.
\end{eqnarray*}
Since the dual functional is obtained by enforcing the condition $w =\gradt\btau -f$ directly, the Lagrangian can be defined by adding only one term from \eqref{relation_rd}:
\beq
\cL_{rd}(v,\tau) = \Jp_{rd}(v) - \frac{1}{2} \|A^{1/2} \nabla v+ A^{-1/2} \btau\|_0^2,  \quad (v,\btau)\in H_{D,g}^1(\O) \times H(\divvr;\O).
\eeq
We can check that the primal and dual functionals obtained by $\cL_{rd}$ are the same as the ones in \eqref{pphi_rd} and \eqref{rd_d}. It is also easy to check that the saddle-point relation (also the strong duality) holds:
\beq
 \inf_{v\in H^1_{D,g}(\O)} \Jp_{rd}(v) =
\Jp_{rd}(u) = \cL_{rd}(u,\bsigma) = \Jd_{rd}(\bsigma) = \inf_{\btau\in H_{N,g}(\divvr;\O)} \Jd_{rd}(\btau).
\eeq
\begin{rem}
In the literature, for example, \cite{Yano:15} and  p.185 of \cite{Braess:07}, when constructing the dual problem, a condition $\gradt \bsigma + \kappa^2 u_h =f$ is enforced, where $u_h$ is the approximation to $u$. This adds extra complexity to the dual problem. In \cite{AV:19}, the right and simpler dual problem \eqref{dphi_rd_weak} is chosen.
\end{rem}
For the primal and dual problems \eqref{pphi_rd_weak} and \eqref{dphi_rd_weak}, the same spaces $S_k\subset H^1(\O)$ and  $RT_{k-1}\subset H(\divvr;\O)$, $k\geq 1$,  as the discrete spaces for the diffusion problem can be used to approximate $u$ and $\bsigma$, respectively. The robust a posteriori FE error estimator based on the primal-dual gap can be found in \cite{AV:19}.

\subsection{Linear Elasticity Problem}
Let $\beps(\bu) = (\nabla \bu + (\nabla \bu)^T)/2$ be the strain tensor. For an isotropic material, we have the following constitutive relation between the strain tensor $\beps(\bu)$ and the Cauchy stress tensor $\bsigma$, 
$\bsigma = \bbC\beps(\bu)$ or $\beps(\bu) = \bbA \bsigma,
$
where the elasticity tensor $\bbC$ and the compliance tensor $\bbA=\bbC^{-1}$ are ($\lambda_1>0$ and $\lambda_2>0$):
\beq
\bbC\beps(\bu)=\lambda_1 \tr (\beps(\bu))I + 2\lambda_2 \beps(\bu)
\and
\bbA \bsigma=\dfrac{1}{2\lambda_2} (\bsigma- \dfrac{\lambda_1}{d\lambda_1+2\lambda_2} (\tr \bsigma)I).
\eeq
Consider the linear elasticity problem:
\beq
-\gradt (\bbC \beps(\bu))= \bff \mbox{ in } \O,
\eeq
with boundary conditions (we use homogeneous boundary conditions for simplicity, non-homogenous conditions can be easily added):
$$
\bu =0 \quad \mbox{on }\Gamma_D \and \bbC\beps(\bu)\cdot\bn =0 \quad \mbox{on }\Gamma_N. 
$$
For this problem, let $V= H^1(\O)^d$,
$V_0 = H^1_D(\O)^d$, $\Lambda \bv  = -\beps(\bv)$,  $\Lambda^t  = \gradt $, $\cA =\bbC$, and $\cA^{-1} =\bbA$ thus $\cA\Lambda v =  -\bbC \beps(\bv)$.
Let $H(\divvr;\O;\mathbb{S})$  be the space of square-integrable symmetric tensors with square-integrable divergence (the divergence is taken row-wise):
$$
H(\divvr;\O;\mathbb{S}) := \{\btau \in L^2(\O)^{d\times d}:  \gradt\btau \in L^2(\O)^d, \btau = \btau^T\}.
$$
Let $\Sigma = H(\divvr;\O;\mathbb{S})$. 
We have the following integration by parts formula (see for example, page xxix and page 288 of \cite{Ciarlet:88} and Theorem 3.1 of \cite{CGK:12}):
\beq 
(\btau, \beps(\bv))_\O + (\gradt \btau, \bv)_\O = (\btau\bn, \bv)_{\p\O}, \quad \forall (\bv, \btau)\in H^1(\O)^d\times H(\divvr;\O;\mathbb{S}).
\eeq
The primal problem is:
\beq \label{pphi_le}
\mbox{Find } \bu\in H^1_D(\O):  \;
\Jp_{le}(\bu) = \inf_{\bv\in H^1_D(\O)^d} \Jp_{le}(\bv), \;
\Jp_{le}(\bv)= \frac{1}{2}(\bbC\beps(\bv), \beps(\bv)) - (\bff,\bv).
\eeq
By the derivation of the dual problem (D), we set the constraint set as
\beq
 H_{N}(\divvr;\O;S;f) = \{\btau \in H(\divvr;\O;\mathbb{S}): \gradt \btau = -f\and \btau \bn = 0 \mbox{ on } \Gamma_N\}.
\eeq
The dual problem is:
\beq \label{dphi_le}
\mbox{Find } \btau\in  H_{N}(\divvr;\O;\mathbb{S};f):  \;
\Jd_{s}(\bsigma) = \inf_{\btau\in  H_{N}(\divvr;\O;\mathbb{S};f)} \Jd_{le}(\bv), \; \Jd_{le}(\btau)= -\frac{1}{2}(\bbA \btau, \btau).
\eeq

\begin{rem}
As it is well-known in the mechanics \cite{Ciarlet:88,OR:83,Reddy:17}, the primal problem represents the principle of virtual work. The functional $\Jp_{le}$ is the total potential energy, which is the sum of the strain energy $ \frac{1}{2}(\bbC\beps(\bv), \beps(\bv))$ and of the potential energy of the exterior forces $- (\bff,\bv)$, see p.27 of \cite{Ciarlet:78}. The functional $\Jd_{le}$ is called the complementary energy functional. Similar explanations also hold for the diffusion and reaction diffusion problems. This is also the reason that the norms $\tri\cdot \tri $ and $\Rnorm{\cdot}$ are called energy norms.
\end{rem}
The relation \eqref{relation} for the linear elasticity equation is:
\beq \mbox{(extremal relation)}\quad 
\bsigma = \bbC \beps (u).
\eeq
The energy norms are
$$
\tri \bv \tri_{le} = \|\bbC^{1/2}\beps (\bv)\|_0 \mbox{ for } \bv \in H_D^1(\O)^d
\and
\Rnorm{\btau}_{le} =\|\bbA^{1/2}\btau\|_0 \mbox{ for }  \btau \in H_N(\divvr;\O;\mathbb{S};0).
$$
The combined error relation \eqref{PS3} is 
\beq 
\tri \bu-\bv \tri_{le}^2 + \Rnorm{\bsigma-\btau}_{le}^2 
= \|\bbC^{1/2}\beps( \bv) - \bbA^{1/2}\btau\|_0^2,
\; v\in H^1_{D}(\O)^d, \btau\in H_{N}(\divvr;\O;\mathbb{S};f).
\eeq
The Lagrangian for the linear elasticity problem is 
\beq 
\cL_{le}(\bv,\btau) = \Jp_{le}(\bv) - \frac{1}{2} \|\bbC^{1/2} \eps (\bv)- \bbA^{1/2} \btau\|_0^2
 = -\frac{1}{2}(\bbA\btau,\btau) + (\gradt\btau-f, v),
\eeq
for $(\bv, \btau) \in H_{D}^1(\O)^d\times  H_{N}(\divvr;\O;\mathbb{S})$.
This is the renown Hellinger-Reissner principle.

We can check that the primal and dual functionals obtained by $\cL_{le}$ are the same as the ones in \eqref{pphi_le} and \eqref{dphi_le}. It is also easy to check that the saddle-point relation (also the strong duality) holds:
\beq
 \inf_{\bv\in H^1_{D}(\O)^d} \Jp_{le}(v) =
\Jp_{le}(u) = \cL_{le}(\bu,\bsigma) = \Jd_{le}(\bsigma) = \inf_{\btau\in H_{N}(\divvr;\O;\mathbb{S};f)} \Jd_{le}(\btau).
\eeq

\begin{rem}
They are several other ways to develop the dual problems for the linear elasticity problems, see for example \cite{Ciarlet:88,OR:83,CGK:11,CGK:12}.  
\end{rem}
The FE approximations for the primal problem is standard, see \cite{Ciarlet:78,Braess:07}. For the dual problem with symmetric tensors, the mixed FEM is developed in \cite{AW:02}.

\subsection{Other linear and nonlinear examples}
The 4th-order Kirchhoff-Love plate problem, is also in the framework of the linear elliptic boundary value problem, see \cite{Arnold:90}. Its FE a posteriori error estimator based on the primal-dual variational principle can be found in \cite{BPS:20}.

Several variational inequality problems including the obstacle problem (Example 1.26 of \cite{Han:05}), contact problems in elasticity (Example 1.27 of \cite{Han:05}), are in the framework of primal-dual variational principles discussed in Section 2, see \cite{ET:76,HHNL:88,Glow:84,Bar:15}. FE a posteriori error estimator based on the primal-dual gap for the obstacle problem can be found in \cite{BHS:08}.

Nonlinear equations, such as the nonlinear Dirichlet problem (p-Laplacian problem) (see p.81 of \cite{ET:76}), the  Rudin-Osher-Fatemi model of total variation denoising \cite{Bar:15_mathcomp,BM:19}, and the Monge–Kantorovich problem \cite{BS:17} are also examples of primal-dual variational problems in the framework of Section 2. In \cite{ET:76,Zeidler:85,NR:04,Han:05}, there are many other physically and economically interesting problems in the framework.

In principle, the idea of using primal-dual gap error estimator as both the FE and RB a posteriori error estimators can be applied to these more challenging problems too. For RBM, advanced tools like the EIM (Empirical Interpolation Method) are needed. The rigorous upper bound is still true as long as the primal and dual approximations are in the right approximation spaces and the nonlinear functionals are exactly evaluated. If EIM is used to approximate nonlinear primal and dual functionals, we should also design algorithms carefully to balance the approximation error of EIM and other error contributions of RB problems, as we will do in Section 7 for FE and RB errors.

\subsection{Least-squares variational principle as a special case}
For many linear or nonlinear problems, for examples, non-symmetric problems, there is no natural minimization principle associated, then the primal-dual framework does not apply. In some cases, the dual problem can be very complicated, for example, the Stokes equation, see \cite{ET:76}. One remedy is constructing an artificial minimization principle: the least-squares principle \cite{BG:09}. For example, for the equation $-\gradt(A\nabla u) + Xu =f$, with $Xu = \bb\cdot\nabla u + cu$, $\bb(x)\in C^1(\O)^d$, $c(x)\in L^{\infty}(\O)$, with $u=0$ on $\p\O$, the least-squares functional is 
\beq
J^{ls}(\btau,v) = \|A^{-1/2}\btau+ A^{-1/2}\nabla v\|_0^2 + \|\gradt \btau +Xv -f\|_{0}^2, 
\; (\btau,v) \in H(\divvr;\O)\times H^1_0(\O).
\eeq
If we view $J^{ls}(\btau,v) =\Jp(\btau,v)$ as the primal functional, then its dual functional is simply $\Jd(\btau,v) = -J^{ls}(\btau,v)$. Both the primal and dual problems in the sense of \eqref{pp} and \eqref{pd} are the same least-squares optimization problem. The optimizer is the solution $(\bsigma, u)$ with $\bsigma = -A \nabla u$, and $J^{ls}(\bsigma,u)=0$. Thus, the strong duality \eqref{pequal} is simply
$$
\Jp(\bsigma,u) = J^{ls}(\bsigma, u) = 0 =-J^{ls}(\bsigma, u)=\Jd(\bsigma,u).
$$
By choosing the dual approximation the same as the primal approximation, the gap of the primal and dual approximations \eqref{gap}, is simply 
$2J^{ls}(\bsigma_h, u_h)$, while it is well known that
\beq
J^{ls}(\bsigma_h, u_h) = \tri (\bsigma-\bsigma_h, u-u_h)\tri_{ls}^2
\eeq
with $\tri (\btau, v) \tri_{ls}^2=  \|A^{-1/2}\btau+ A^{-1/2}\nabla v\|_0^2 + \|\gradt \btau +Xv\|_{0}^2$, the square of norm induced by the least-squares functional. The global error estimator and local error indicator can be naturally defined.

In this understanding, the ideas, methods, and algorithms developed in the paper for the primal-dual variational principle can be applied to least-squares variational principle based-RBM too. The small and important difference is that the natural primal-dual variational principle is based on energy minimization principles of the real world, while the least-squares energy functional is an artificial energy functional (thus the related least-squares energy norm is an artificial energy norm) and needs a careful construction (for example, choosing weights of different terms) and analysis. Results of RBM based on least-squares variational principles can be found in \cite{Yano:16,Olson:20}.


\section{RB Greedy Algorithms and Numerical Results}
In this section, we discuss three greedy algorithms for our PD-RBM.
We use the following parametric diffusion problem on an L-shape domain as our test problem.

\noindent ({\bf Test Parametric Diffusion Problem})
{\em
Consider a two-parameter diffusion equation on an L-shape domain $\O = (-1,1)^2\backslash (-1,0]^2$ with a two dimensional parameter} $\mu = (\mu_{[1]},\mu_{[2]})$:
$$
-\gradt (A(\mu) \nabla u)  =1 \quad \mbox{in }\O, \quad u=0\mbox{ on }\p\O,
$$
$$
\mbox{where }\quad A =\a I, \quad 
\a(\mu) = \left\{ \begin{array}{lll}
10^{\mu_{[1]}} & \mbox{if} & xy>0,\\[1mm]
10^{\mu_{[2]}}&\mbox{if} & xy\leq 0,
\end{array}
\right.
\quad \mbox{for}\quad
\mu\in \cD = [-2,2]^2.
$$
We use standard linear continuous FE space to approximate the primal FE problem and $RT_0$ (Raviart-Thomas element of index $0$) and $P_0$ (piecewise constant) pair of mixed FE spaces to approximate the dual problem.

For the combined primal-dual problem, we use $\eta_{h}(u_h,\sigma_h;\mu)$ as the notation for the global PD-FE a posteriori estimator and $\eta_{h,T}(u_h,\sigma_h;\mu)$ as the notation for the local PD-FE a posteriori error indicator on an element $T\in\cT$. The notation $\eta_{\rb}(u_{\rb}^N,\sigma_{\rb}^N;\mu)$ is used as PD-RB a posteriori error estimator. They are all from the same primal-dual gap error estimator.


\subsection{Comparison with the method in \cite{Yano:18}}
In \cite{Yano:16,Yano:18}, minimum-residual mixed (or least-squares) RBMs are introduced. These formulations are introduced to avoid the dual-feasibility condition $\Lambda^t \sigma = -f$. In \cite{Yano:18}, for the diffusion equation, a term 
$\delta^{-1}\|\gradt \btau+f\|_0^2$ is appeared in the least-squares formulation (the boundary condition is also penalized in \cite{Yano:18}). The coefficient $\delta$ is chosen to be smaller than $\inf_{v\in H_0^1}\dfrac{\|\a^{1/2}\nabla v\|_0^2}{\|v\|_0^2}$ to ensure the coercivity constant of the least-squares problem is close to one with respect to the energy norm. The $\delta$ is not $0$ to ensure that the dual-feasibility condition is not exactly enforced in \cite{Yano:18}, although, in theory $\delta=0$ is the perfect constant since the exact solution satisfies the dual-feasibility condition. In this understanding, the PD-RBM in this paper can be viewed as the theoretically perfect case of $\delta =0$ in \cite{Yano:18}. Thus, many numerical observations in \cite{Yano:18} are also true for our method.

\subsection{Greedy algorithm on a fixed mesh with an adaptive RB tolerance (Algorithm 1)}

In practice, many model reduction problems are computed on a fixed mesh. It is important to re-examine the greedy algorithm to see with an exact primal-dual energy error estimator, what are the best stopping criteria and greedy strategy. 

First, we propose a greedy algorithm on a fixed mesh with an adaptive RB tolerance. 
In the classic RB greedy algorithm, we choose a fixed RB error tolerance $\eps_{\rb}$, and stop the greedy algorithm when all error estimators for $\mu$ in the training set $\Xi_{\mathtt{train}}$ are smaller than $\eps_{\rb}$. In the classic algorithm, the estimator estimates the difference between the RB solution and the "truth" FE solution. In the extreme case, when all points in $\Xi_{\mathtt{train}}$ are used to compute RB snapshots, the estimator is zero. So the greedy algorithm will stop no matter how small the RB tolerance is. On the other hand, if the PD-RB error estimator is used, it is a measure of the true energy error, which will never be zero or converge to zero on a fixed mesh (except for some trivial cases). So we need to adaptively choose the RB error tolerance $\eps_{\rb}$ to make the algorithm stoppable. 

The notation $\eps_h$ is used to represent the FE error tolerance. Suppose the RB parameter set $S_N =\{\mu_i\}_{i=1}^N$ is chosen, define:
$$
\eps_h^0 = 0
\quad\mbox{and}\quad
\eps_h^n = \max\{\eta_{h}(u_h,\sigma_h;\mu_1),\cdots, \eta_{h}(u_h,\sigma_h;\mu_n)\}, \quad 1\leq n \leq N.
$$
Here, we emphasize that $\eta_{h}(u_h,\sigma_h;\mu_i)$ is understood as $\eta_h(u_h(\mu_i),\sigma_h(\mu_i))$.
It is easy to see that 
$$
\eps_h^n = \max\{\eps_h^{n-1}, \eta_{h}(u_h,\sigma_h;\mu_n) \}.
$$ 
So $\eps_h^n \geq \eps_h^{n-1}$, $\eps_h^n$ is monotonically increasing with respect to the RB dimension  $n$.

We use the notation $\eps_\rb$ for the RB error tolerance and want to determine a reasonable $\eps_\rb^N$, such that when 
$
\max_{\mu \in \Xi_{\mathtt{train}}} \eta_{\rb}(u_{\rb}^N,\sigma_{\rb}^N;\mu) \leq \eps_\rb^N 
$,
the greedy algorithm can be terminated.
By Theorem \ref{comp_fe_rb},  for those $\mu_i \in S_N$, their FE combined energy error and RB combined energy error are the same; and for all $\mu\in\cD$, the FE combined energy error is at least as good as the RB combined energy error, thus we define $\eps_\rb^n$ such that
$$
\eps_\rb^n \geq \eps_h^n.
$$
On the other hand, choosing $\eps_\rb^n$ too close to $\eps_h^n$ will result more RB basis functions needed to be computed. From the triangle inequality, for a fixed $\mu$, and a combined norm $\|(\cdot,\cdot)\|$ for both the primal and dual variables, we have
\beq \label{normtri}
\|(u,\sigma) - (u_{\rb}^N,\sigma_{\rb}^N)\| \leq 
\|(u,\sigma) - (u_{h},\sigma_{h})\| 
+ \|(u_h,\sigma_h) - (u_{\rb}^N,\sigma_{\rb}^N)\|.
\eeq
Thus, to balance of the FE tolerance $\eps_h$ and RB tolerance $\eps_{\rb}$, one simple and reasonable choice is 
\beq \label{epsrb}
\eps_{\rb}^n: = \max\{\rferb \eps_h^n, \eps_{\rb}^{n-1}\}, \quad 1\leq n\leq N,
\eeq
where $\rferb >1$ is a fixed number, which is the ratio between the $\eps_\rb^n$ and $\eps_h^n$. 
In this paper, we test two values of $\rferb$, $2$ and $1.1$. 
From \eqref{normtri}, 
the definition \eqref{epsrb} means that for a parameter $\mu\in \Xi_{\mathtt{train}}\backslash S_N$, we keep the norm difference between the RB solution and the FE solution roughly the same ($\rferb=2$) or a fraction ($\rferb=1.1$) of the FE error.
If $\eps_{\rb}^0$ is reasonably small, in the end, $\eps_{\rb}^N$ is $\rferb\eps_h^N$, and they are of the same magnitude.
Note that, if $\eps_{\rb}^0$ is achievable in this given mesh, then the algorithm proceed as the standard greedy method. 

Algorithm \ref{greedy_fixed} is an exact energy error certified PD-RB greedy algorithm on a fixed mesh.
\begin{algorithm}[!ht]
\KwIn{
Training set $\Xi_{\mathtt{train}} \subset \cD$, 
initial FE accuracy tolerance $\eps_h^0 = 0$, 
initial RB tolerance $\eps_{\rb}^{0}$, 
some $\rferb >1$,
maximal RB dimension $N_{\max}$, initial parameter $\mu_1$, $N=1$.
}
\KwOut{
$N$, $\eps_h^N$, $\eps_{\rb}^N$, $V_{\rb}^N$, $\Sigma_{\rb}^N$.
}
\BlankLine
\begin{algorithmic} [1]
\STATE 
Compute FE solutions $u_h(\mu_N)$ and $\sigma_h(\mu_N)$.

\STATE
Let $\eps_h^N=\max\{\eta_{h}(u_h,\sigma_h;\mu), \eps_h^{N-1}\}$ 
and  $\eps_{\rb}^N = \max\{ \rferb \eps_h^N, \eps_{\rb}^{N-1}\}$.

\STATE Update $V^N_{\rb} $ and $\Sigma^N_{\rb}$. 

\STATE 
Choose 
$\mu_{N+1} = \mbox{argmax}_{\mu\in \Xi_{\mathtt{train}}} \eta_{\rb}(u_{\rb}^N,\sigma_{\rb}^N;\mu).
$

\STATE If $\eta_{\rb}(u_{\rb}^N,\sigma_{\rb}^N;\mu_{N+1}) > \eps_{\rb}^N$ and $N<N_{\max}$, then set $N=N+1$, {\bf goto} 1. Otherwise, {\bf terminate}. 
\end{algorithmic}\caption{An exact energy error certified PD-RB greedy algorithm with an adaptive RB tolerance on a fixed mesh (an adaptive RB tolerance greedy algorithm)}
\label{greedy_fixed}
\end{algorithm}

Since the computations of RB primal and dual solutions are based on optimizations, for a fixed $\mu$, the following is true: 
\beq \label{SA}
\eta_{\rb}(u_{\rb}^N,\sigma_{\rb}^N;\mu) \leq \eta_{\rb}(u_{\rb}^{N-1},\sigma_{\rb}^{N-1};\mu).
\eeq
Which means, the more basis functions, the smaller the energy error. The saturation assumption (Definition 3.3 of \cite{HSZ:14}) is satisfied with the  saturation constant being exact one. Then
the improved algorithm based on the saturation trick (Algorithm 2 of \cite{HSZ:14}) can be used to avoid going through all $\mu\in \Xi_{\mathtt{train}}$ in Step 4 of Algorithm \ref{greedy_fixed}. 
In the improved algorithm of saturation trick, when looping for all $\mu\in \Xi_{\mathtt{train}}$, we can skip those $\mu$'s whose error estimators in the previous stages (with smaller RB dimensions) are already smaller than the current temporary maximum error, since these parameters will not be chosen in this round.

%

\subsection{Numerical experiments for Algorithm 1}
In our numerical experiments, we use the notation {\tt maxerror} in the tables and figures:
$$
{\tt maxerror} = \mbox{max}_{\mu\in \Xi_{\mathtt{train}}} \eta_{\rb}(u_{\rb}^N,\sigma_{\rb}^N;\mu).
$$

\noindent ({\bf Setting for the numerical experiment for Algorithm 1})
{\em For the test problem given at the beginning of the section,
we use a uniform mesh with  $393216$ elements.  Thus the number of degrees of freedom (DOFs) for the primal problem is $197633$, and the number of DOFs for the dual problem is $984064$. We choose a fixed $\Xi_{\mathtt{train}}$ to be a random uniform discretization of $[-2,2]^2$ of $10^5$ points.

%


The original $\eps_{\rb}^0$ is chosen to be $10^{-3}$, which is a too strong requirement for the current mesh. We choose $\mu_1 = [0,0]$. The max RB dimension $N_{\max}$ is chosen to be 20, which is more than enough for all our tests in the paper.}


\begin{table}[htp]
\caption{Convergence results of PD-RBM for test diffusion problem on a fixed mesh with Algorithm 1 ($\mathtt {r_{rb,fe}}=2$)}
\begin{center}
\begin{tabular}{|c|r|r|r|r|r|}
 \hline
       RB Dim  & 1 & 2 & 3&4\\
 \hline
$\mu_{[1]}$ &    0.0000&
    -1.9996&
    1.9936&
   -1.9970\\
 \hline
$\mu_{[2]}$ &    0.0000&
       1.9808&
   -1.9999&
   -1.0199\\
 \hline
 $\eps_{h}$ &          
 	0.0077&
    0.0354&
    0.0354&
    0.3261\\
       \hline
 $\eps_{\rb}$ &      
 	0.0154&
    0.0708&
    0.0708&
    0.6522\\
    \hline
       maxerror &      3.6569&
    2.3279&
    0.6066&
    0.3506\\
 \hline
  skipped &         0&
        90693&
      55251&
      83013\\
    \hline
 test error&
  && &
    0.3461
    \\
    \hline
\end{tabular}
\end{center}
\label{tab-d2}
\end{table}%

\begin{figure}[!ht]
    \centering
    \begin{minipage}[!hbp]{0.33\linewidth}
        \includegraphics[width=0.99\textwidth,angle=0]{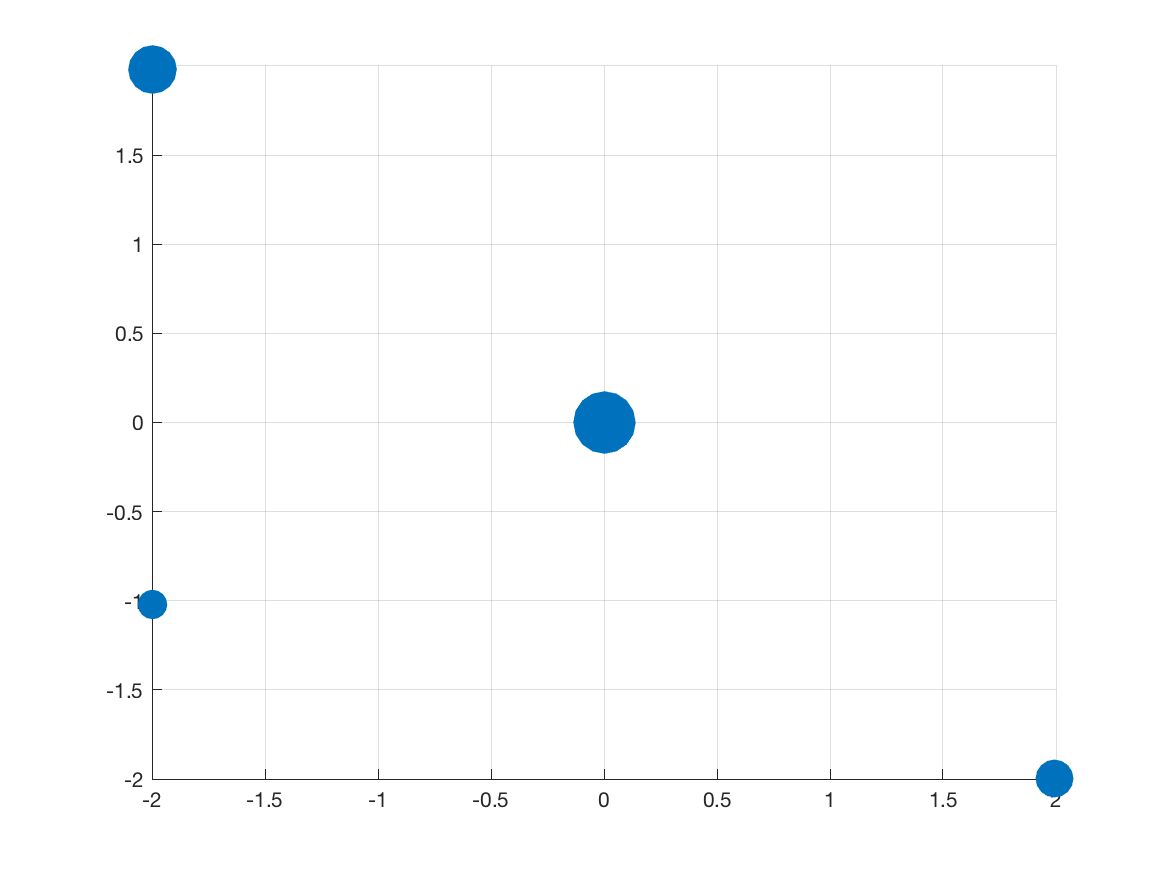}
        \end{minipage}%
        \quad
    \begin{minipage}[!htbp]{0.33\linewidth}
        \includegraphics[width=0.99\textwidth,angle=0]{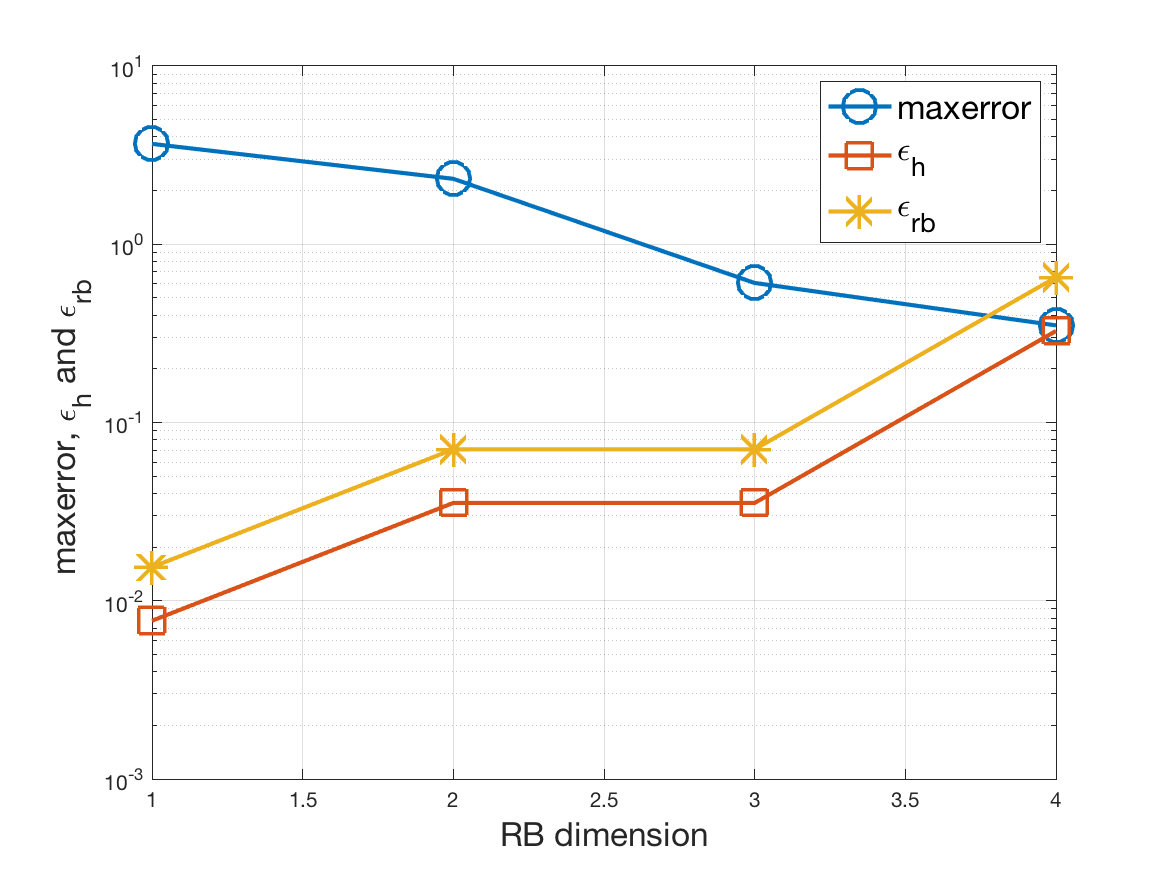}
        \end{minipage}
        \caption{Selection of RB points $S_N$ (left) and convergence history of RB maxerror (right) for test diffusion problem on a fixed mesh with PD-RBM and  Algorithm 1 ($\mathtt {r_{rb,fe}}=2$)}%
            \label{rberror_d_fixed}
\end{figure}
From computational results on Table \ref{tab-d2}, we see that with one RB function, the RB tolerance is already updated to $0.0154$. After the second RB function is picked, due to boundary and interior layers and corner singularity, this parameter one has big FE approximation error, thus the RB tolerance is further updated to $0.0708$. The algorithm terminated with $4$ basis function.

%
%
On the left of Fig. \ref{rberror_d_fixed}, we depict the RB parameter set $S_N$. The bigger the size, the earlier the point is selected.  We also show the numerical values of the $4$ $\mu_i$ on the second and third rows of Table \ref{tab-d2}. On the 4th and 5th rows of Table \ref{tab-d2}, $\eps_{h}^i$ and $\eps_{\rb}^i$ are shown. 
It is clear that both of them are monotonically non-decreasing. The final $\eps_\rb$ is quite big compared to the common practice  of the standard residual type RB a posteriori error estimator. 

On the right of Fig. \ref{rberror_d_fixed}, we show the decay of {\tt maxerror} and the increasing of $\eps_h$ and $\eps_{\rb}$ with respect to RB space dimensions. We can see a clear spectral convergence of {\tt maxerror} and the adaptivity of  $\eps_h$ and $\eps_{\rb}$ from the figure. The algorithm stops when {\tt maxerror} $\leq \eps_\rb$. 

The 7th row of Table \ref{tab-d2} , the row of "skipped", denotes the number of parameter in the train set being skipped in step 4 of Algorithm 1 by the algorithm of saturation trick. We find that the algorithm of saturation trick saves a large amount computations of the loop in $\Xi_{\mathtt{train}}$.

We also test the quality of the resulting RB set by running 10000 random online test cases. The max error of these 10000 test cases is $0.3461$, which is very similar to the final {\tt maxerror} $0.3506$, and is less than our $\eps_{\rb} = 0.6522$.  For the later numerical tests, we test the online RB problems with the same set of random parameters, and we list the maximum test combined RB energy error on the last row of the numerical results table.


On Fig. \ref{rberror_d_fixed11}, we show the numerical results for the same test setting with $\mathtt {r_{rb,fe}}=1.1$. Note that for this case, the same number of RB basis function is computed as the $\mathtt {r_{rb,fe}}=2$ case.

In both experiments  ($\mathtt {r_{rb,fe}}=2$ and $1.1$) for the test parametric diffusion problem on a uniform mesh with PD-RBM and Algorithm 1, we clearly find that for a non-optimal uniform mesh, even with quite a large number of DOFs, the final $\eps_{h}$ and $\eps_{\rb}$ can be quite big, and the RB dimension is often small.

\begin{figure}[!ht]
    \centering
    \begin{minipage}[!hbp]{0.33\linewidth}
        \includegraphics[width=0.99\textwidth,angle=0]{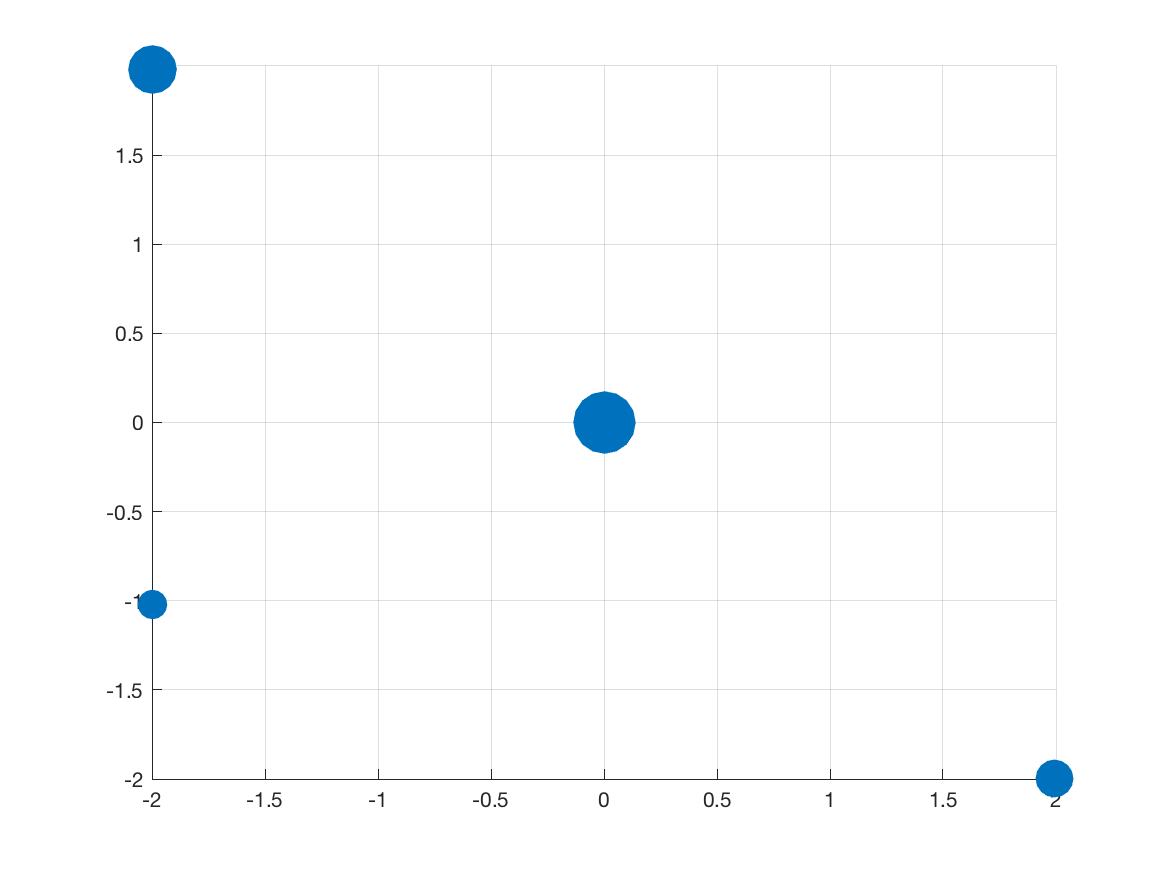}
        \end{minipage}%
        \quad
    \begin{minipage}[!htbp]{0.33\linewidth}
        \includegraphics[width=0.99\textwidth,angle=0]{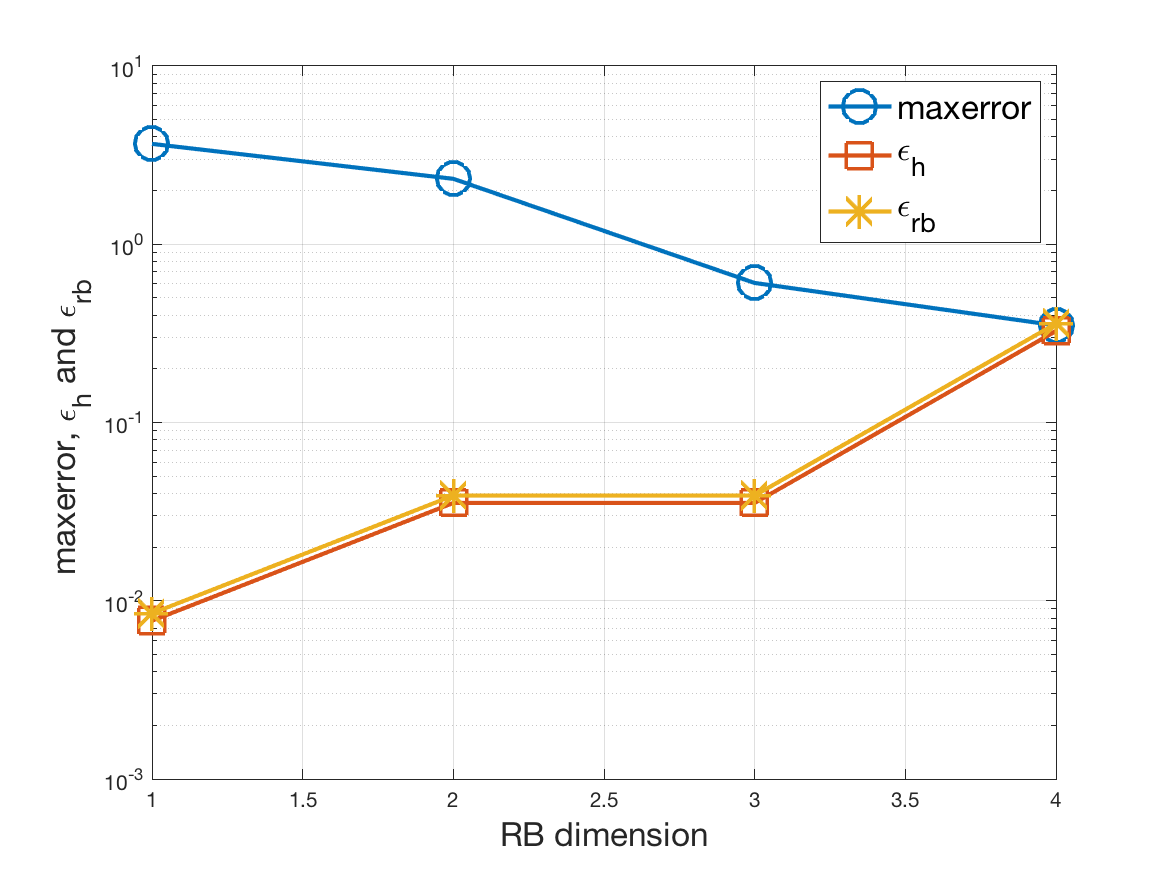}
        \end{minipage}
        \caption{Selection of RB points $S_N$ (left) and convergence history of RB maxerror (right) for test diffusion problem on a fixed mesh with PD-RBM and  Algorithm 1 ($\mathtt {r_{rb,fe}}=1.1$)}%
            \label{rberror_d_fixed11}
\end{figure}

\begin{rem}
From the experiments, we can see that eventually, the quality of the RBM is controlled by the last {\tt maxerror} (the number depends on $\Xi_{\mathtt{train}}$). This number can be used as the quality indicator of the final RB set, instead of $\eps_{\rb}^N$, although these two numbers are often of the same scale and very close, since
$$
\eps_{h}^N \leq \mathtt{maxerror} \leq \eps_{\rb}^N,
$$
where $N$ is the final RB space dimension.
\end{rem}

\begin{rem}
As already did in the numerical experiments of Section 7.1 of \cite{Yano:18}, even with more RB basis functions, the exact error of a RB solution will stop decreasing at some point. This is due to the result of Theorem \ref{comp_fe_rb}, that a RB solution cannot out-perform  its corresponding FE solution. From our 10000 random test, we also see the max error of the 10000 test is 0.3461, which is between the final $\eps_{h}^N$ and $\eps_{\rb}^N$. This suggests that adaptive mesh refinement is necessary to improve the accuracy.
\end{rem}

\subsection{Greedy algorithm with an adaptive mesh and a fixed RB tolerance (Algorithm 2)}
In this subsection, we discuss a greedy algorithm with a spatial adaptivity for the field variable and a fixed RB tolerance. 

As seen in numerical examples in Section 7.3, for parameter-dependent PDEs, a uniform mesh is often not the best choice for approximations due to possible singularities, sharp layers, and other solution features. So it is reasonable to use an adaptive refined mesh to compute RB snapshots. On the other hand, when assembling the RB offline matrices, we need to compute the inner products of different RB snapshots. It is preferred to make sure that all RB snapshots to be computed on a same mesh. The following algorithm is basically the same spatio-parameter adaptivity algorithm of \cite{Yano:16,Yano:18}. In our context, all error estimators (FE or RB) are based on the same primal-dual gap estimator.  Other approaches combined RBM and adaptive FEM can be found in \cite{OS:15,URL:16,ASU:17}. 
We list the detailed algorithm in Algorithm  \ref{greedy_adaptive}.

\begin{algorithm}[!ht]
\KwIn{
Training set $\Xi_{\mathtt{train}} \subset \cD$, 
FE  tolerance $\eps_h$, 
some $\rferb >1$,
RB tolerance $\eps_{\rb} = \rferb\eps_h$, 
maximal RB dimension $N_{\max}$, 
initial coarse FE mesh $\cT_0$,
$N=1$,  $\mu_1$,  flag \texttt{refined}=0.
}
\KwOut{
$N$, 
$V_{\rb}^N$, $\Sigma_{\rb}^N$.
}
\BlankLine
\begin{algorithmic} [1]
\STATE Starting from mesh $\cT_{N-1}$, adaptively compute FE solutions $u_h(\mu_N)$ and $\sigma_h(\mu_N)$
with $\eta_{h,T}(u_h,\sigma_h;\mu)$ as the FE error indicator, compute until 
$\eta_{h,T}(u_h,\sigma_h;\mu)\leq \eps_h$.
The final mesh is $\cT_N$. Modify the flag \texttt{refined} $=1$, if $\cT_N \neq \cT_{N-1}$.


\STATE If \texttt{refined} $=1$, recompute $u_h(\mu_i)$ and $\sigma_h(\mu_i)$, $i=1, \cdots, N-1$ on the current mesh
$\cT_N$ and reset the flag \texttt{refined} $=0$.

\STATE 
Update $V^N_{\rb} $ and $\Sigma^N_{\rb}$.

\STATE 
Choose 
$\mu_{N+1} = \mbox{argmax}_{\mu\in \Xi_{\mathtt{train}}} \eta_{\rb}(u_{\rb}^N,\sigma_{\rb}^N;\mu).$

\STATE If $\eta_{\rb}(u_{\rb}^N,\sigma_{\rb}^N;\mu) > \eps_{\rb}$ and $N<N_{\max}$, then set $N=N+1$, {\bf goto} 1. Otherwise, {\bf terminate}. 

\end{algorithmic}\caption{A primal-dual exact energy error certified PD-RB greedy algorithm with an adaptive mesh  and a fixed RB tolerance (an adaptive mesh refinement greedy algorithm)}
\label{greedy_adaptive}
\end{algorithm}

Since 
$V_{\rb}^N = \mbox{span} \{ u_h(\mu_i)\}_{i=1}^N$ is computed on $\cT_N$ and $\cT_N$ is a refinement of $\cT_{N-1}$, we have confidence that $V_{\rb}^{N}$ has a better quality than $V_{\rb}^{N-1}$. Similarly, we have confidence that $\Sigma_{\rb}^{N}$ is better than $\Sigma_{\rb}^{N-1}$. Thus, for a fixed $\mu$, the following result is still true:
$$
\eta_{\rb}(u_{\rb}^N,\sigma_{\rb}^N;\mu) \leq \eta_{\rb}(u_{\rb}^{N-1},\sigma_{\rb}^{N-1};\mu).
$$
Thus, in the step 4 of Algorithm \ref{greedy_adaptive}, the saturation trick can still be used to avoid computations of all $\mu\in \Xi_{\mathtt{train}}$.

\begin{rem}[Computational Cost]
In Step 2 of Algorithm  \ref{greedy_adaptive}, we need to recompute $N$ (number of RB basis functions) both primal and dual FE solutions. 
Since an optimal adaptive mesh is used, for a complicated problem requiring adaptive mesh, to achieve the same accuracy, a uniform or not optimized mesh requires much more number of points, see for example p.521 of \cite{NSV:09}, where a similar diffusion problem with discontinuous coefficients is discussed. A uniform FEM requires  $10^{20}$ elements to achieve the similar resolution with $2\times 10^3$ element in adaptive FEM. Thus, with a reasonable number of RB basis $N$ (less than 100 typically), the computational cost of Algorithm 2 with adaptive mesh refinements is still much cheaper than a method on a non-optimal mesh. 

In fact, it is a common practice that we re-compute solutions after refinements in the adaptive FEM, see \cite{AO:00,NSV:09,Ver:13}, since an adaptive optimal mesh is much efficient than a bad mesh.

Instead of re-compute the old $u_h(\mu_i)$ and $\sigma_h(\mu_i)$ in Step 2, using the idea from \cite{Yano:16}, we can also interpolate the solution into the final common mesh, since each old mesh is a sub-mesh of the final common mesh. 
This is definitely cheaper than the recomputing approach, although the best approximation properties of \eqref{ap_primalfem2} and \eqref{ap_dualfem} are lost and the interpolation operator between two unstructured grids may not be simple to construct. In this paper, we choose the recomputing approach to keep the theoretical optimal results.
\end{rem}

\subsection{Numerical experiments for Algorithm 2}

In this subsection, we test the Algorithm \ref{greedy_adaptive}. 

\noindent ({\bf Setting for the numerical experiment for Algorithm 2})
{\em 
Choose $\eps_h = 0.08$. The initial mesh $\cT_0$ is shown as the first mesh on Fig. \ref{d_2d_afem_mesh}. The  rest is the same as the setting for the numerical experiment for Algorithm 1.
} 

\begin{figure}[!ht]
    \centering
       \begin{minipage}[!hbp]{0.32\linewidth}
        \includegraphics[width=0.99\textwidth,angle=0]{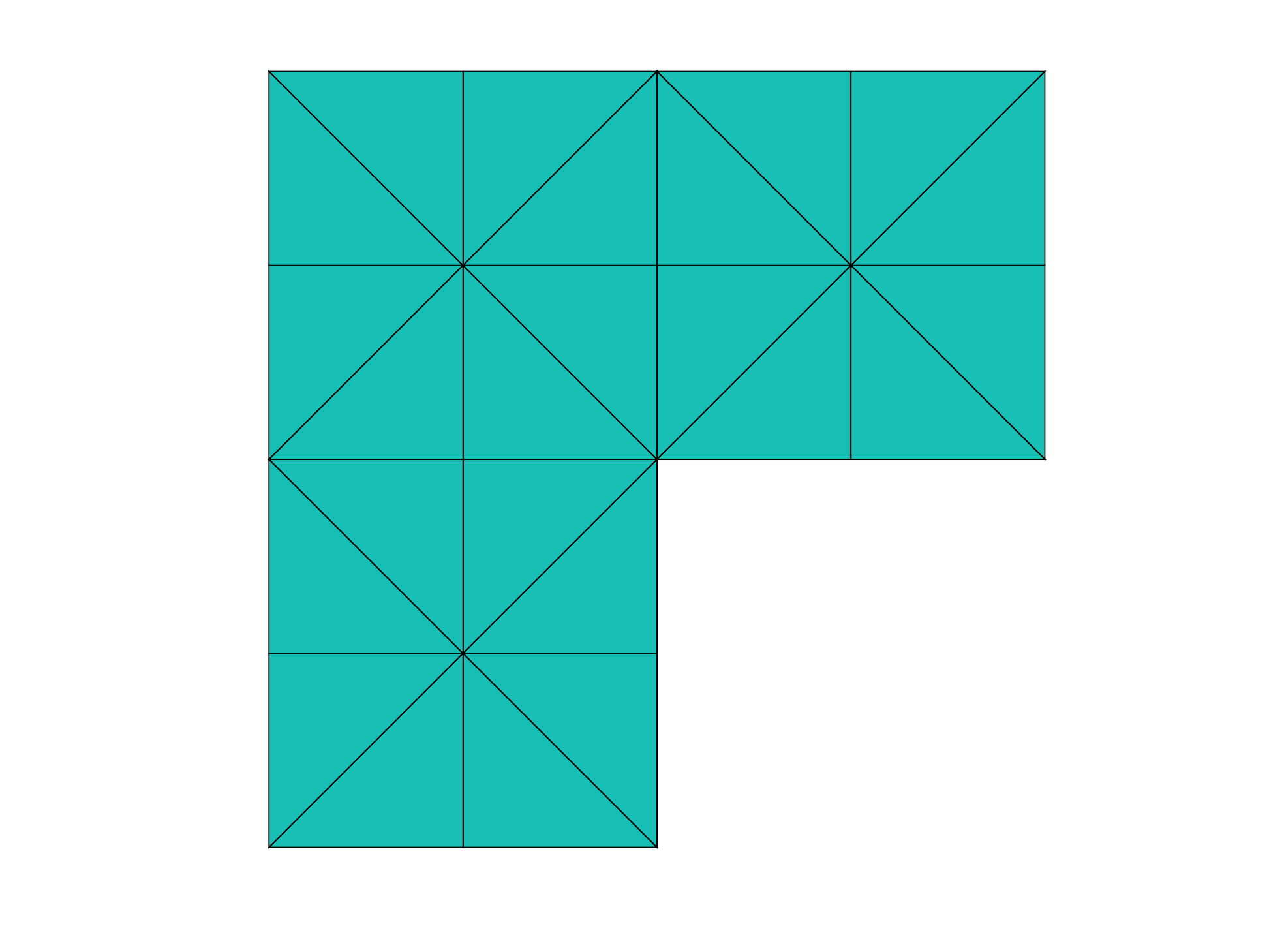}
        \end{minipage}%
  \begin{minipage}[!hbp]{0.32\linewidth}
        \includegraphics[width=0.99\textwidth,angle=0]{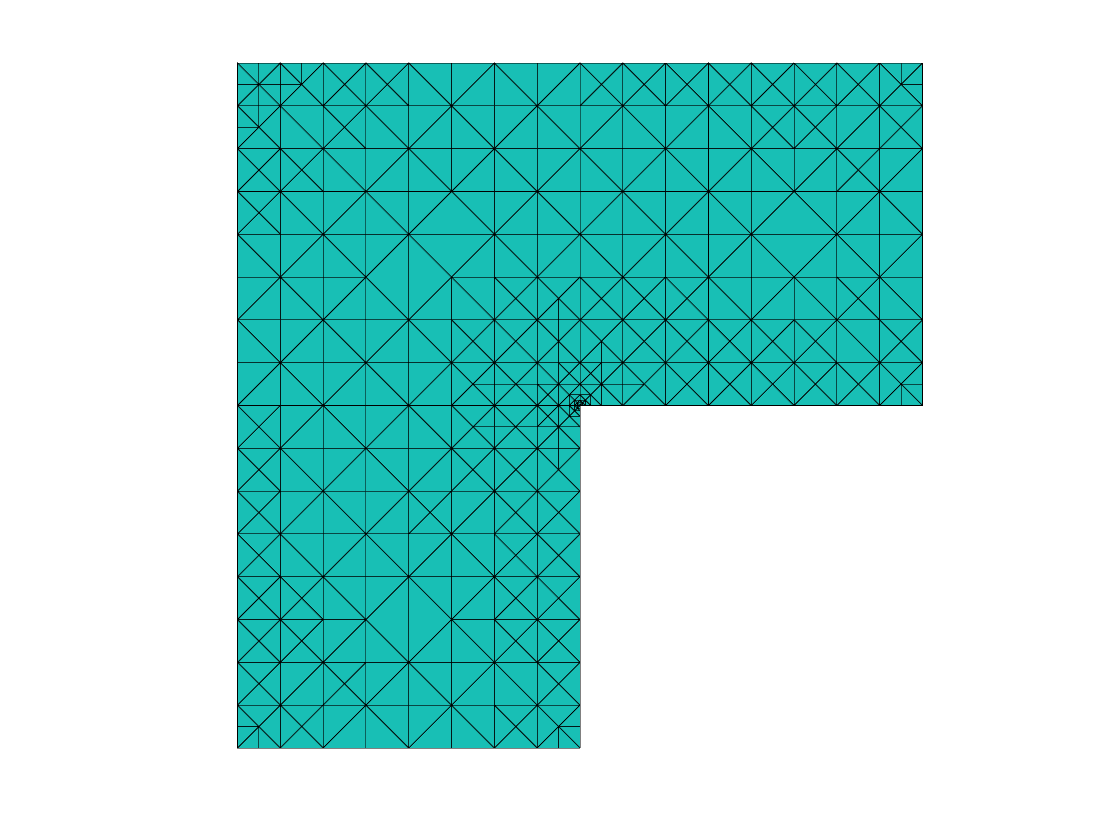}
        \end{minipage}%
    \begin{minipage}[!htbp]{0.32\linewidth}
        \includegraphics[width=0.99\textwidth,angle=0]{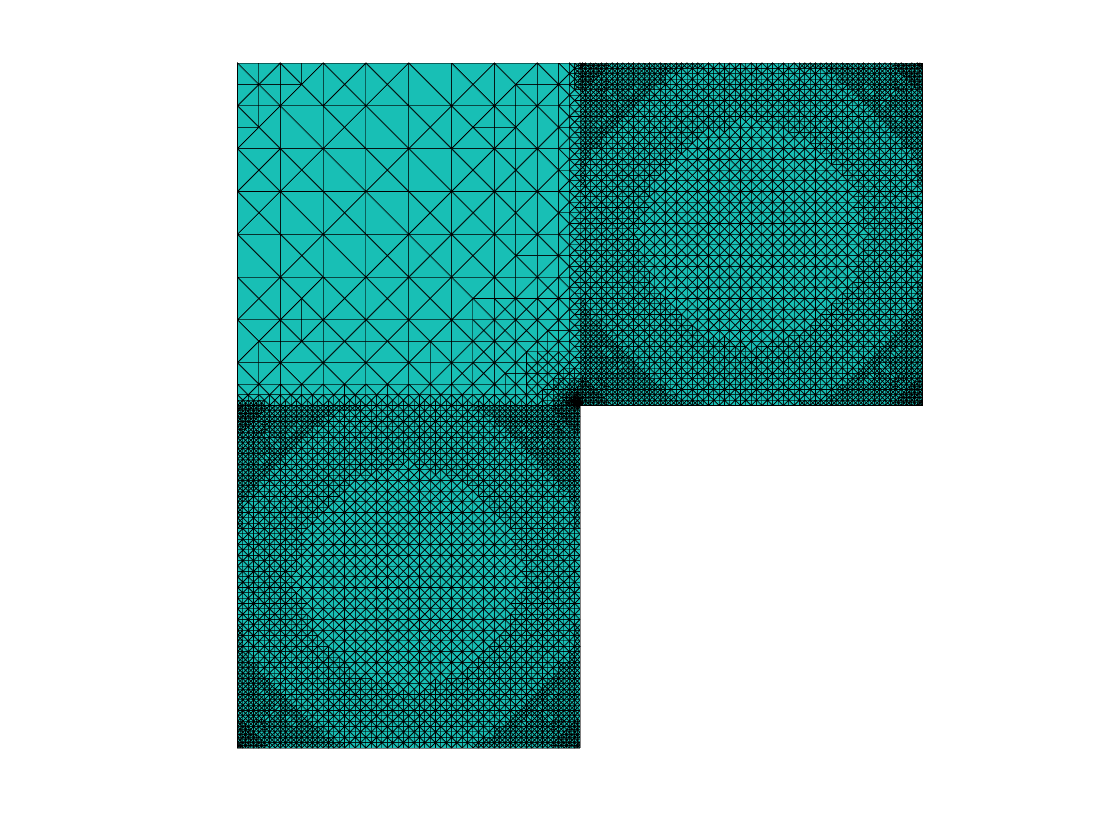}
        \end{minipage}
    \begin{minipage}[!htbp]{0.32\linewidth}
        \includegraphics[width=0.99\textwidth,angle=0]{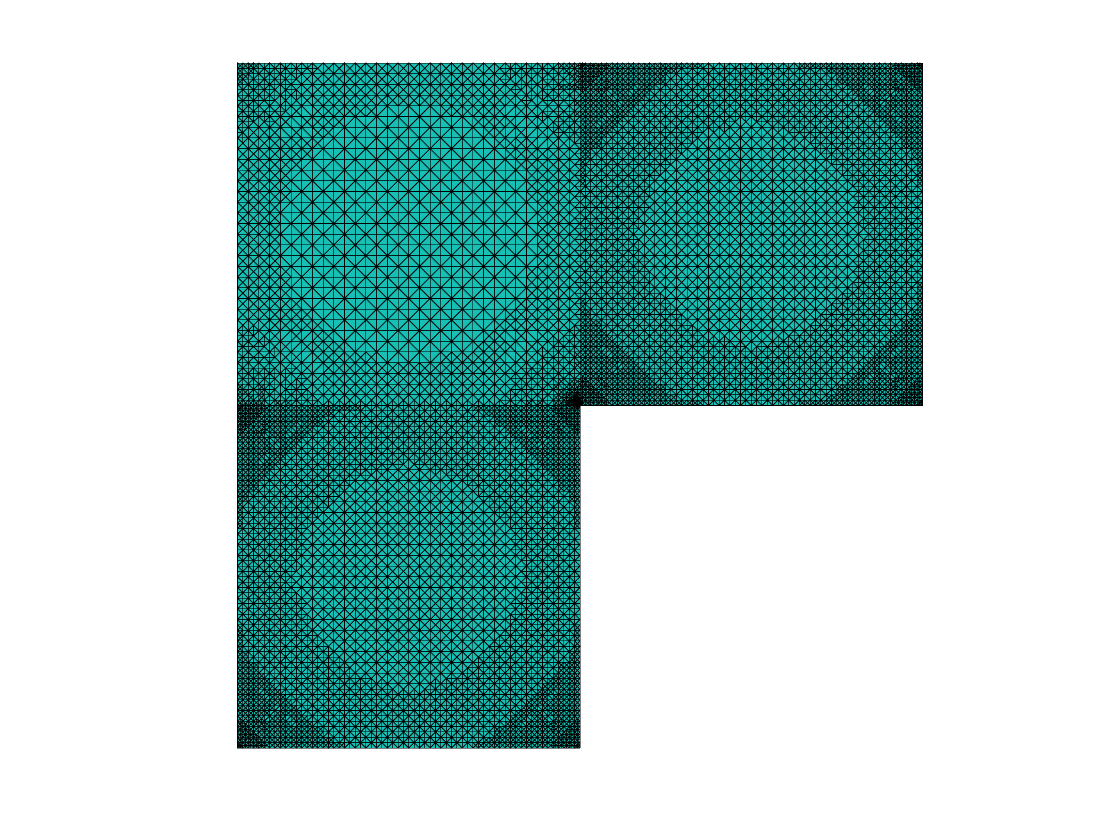}
        \end{minipage}
            \begin{minipage}[!htbp]{0.32\linewidth}
        \includegraphics[width=0.99\textwidth,angle=0]{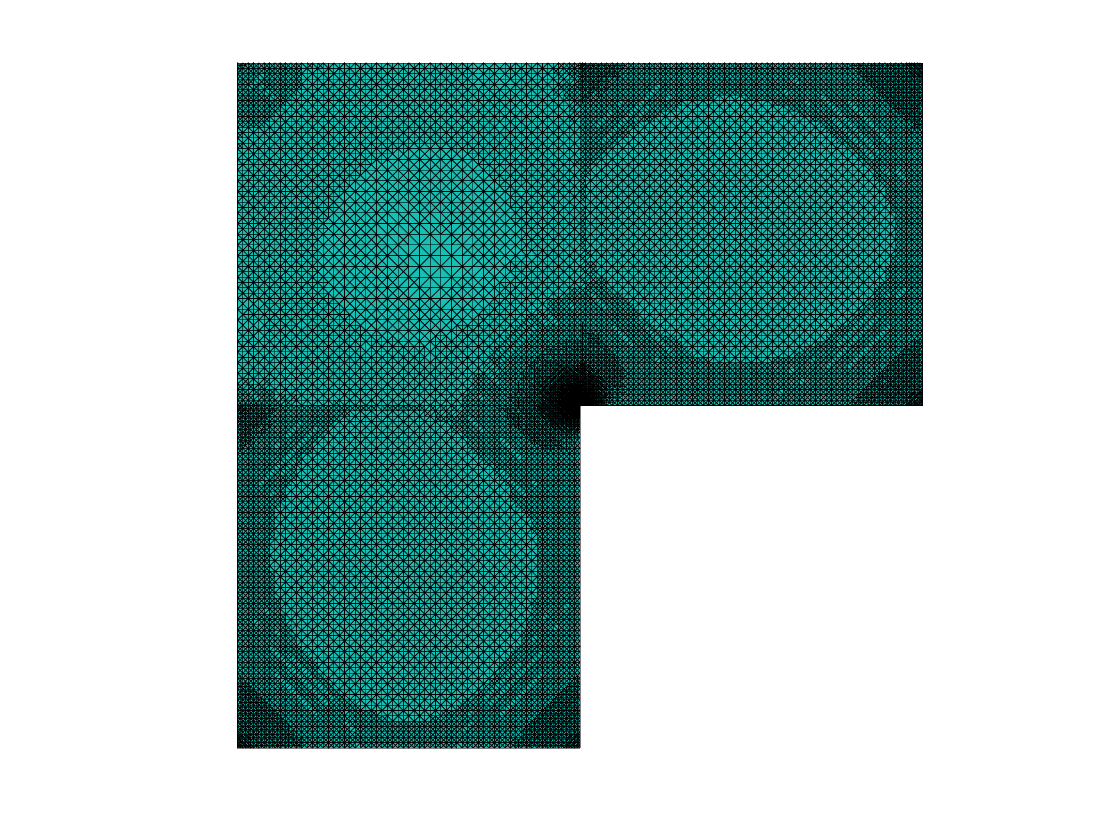}
        \end{minipage}
            \begin{minipage}[!htbp]{0.32\linewidth}
        \includegraphics[width=0.99\textwidth,angle=0]{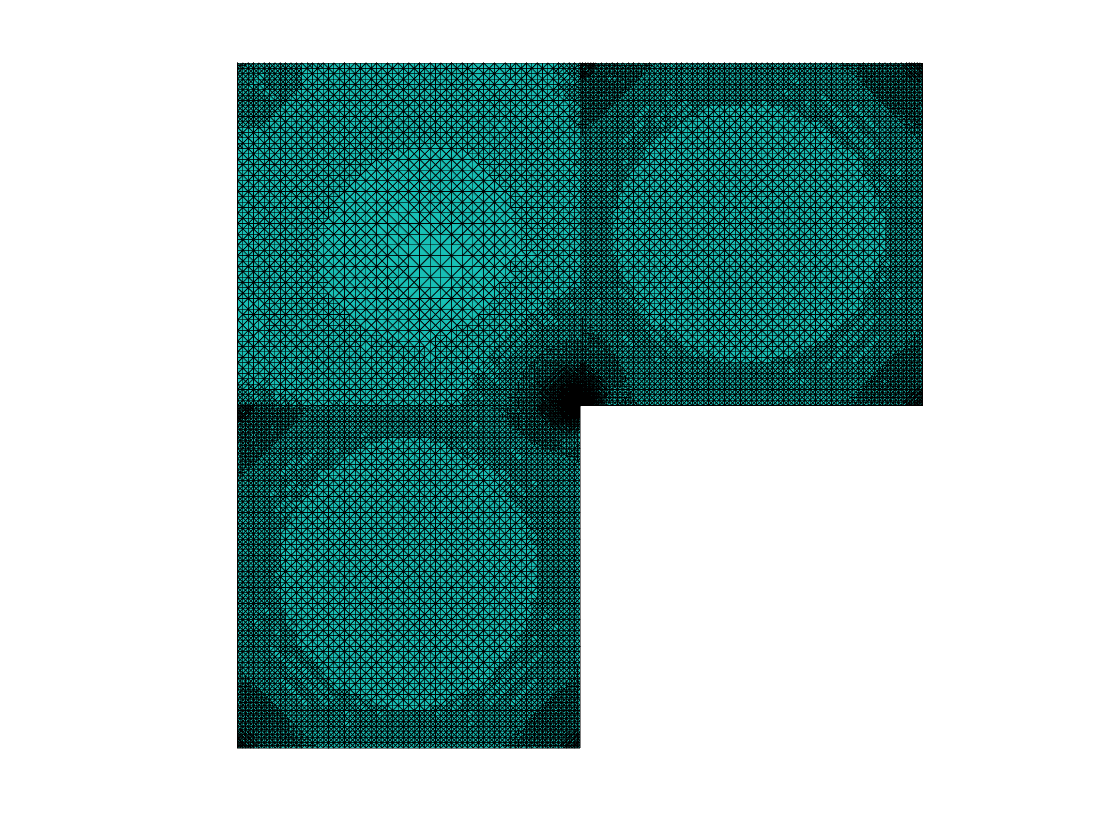}
        \end{minipage}
        \caption{Adaptively refined meshes generated by Algorithm 2 for the test diffusion problem with PD-RBM}%
        \label{d_2d_afem_mesh}
\end{figure}

\begin{table}[!ht]
\caption{Convergence results of PD-RBM for test diffusion problem on an adaptive mesh with Algorithm 2 ($\mathtt {r_{rb,fe}}=2$)}
\begin{center}
\begin{tabular}{|c|r|r|r|r|r|}
 \hline
       RB Dim  & 1 & 2 & 3 &4&5\\
 \hline
$\mu_{[1]}$ &
	0&      
    -1.9996 &
    1.9936&
   -1.9970&
   -1.9976\\
    \hline
    $\mu_{[2]}$ &
	0&      
    1.9808 &
   -1.9999&
   -1.0199&
   -0.2083\\
    \hline
maxerror &      3.7450 &
    2.3489&
    0.6195&
    0.1696&
    0.0995\\
 \hline
 $\cN_{\fe,p}$ &  
       351 &
        9651&
       11148&
       21552&
       26249\\
\hline
 $\cN_{\fe,d}$ &  
       1582 &
       47290&
       54643&
      106319&
      129600\\
 \hline
  skipped &  
        0&
          90695 &
       54872&
       57756&
       76380
      \\
      \hline
refined &
       1 &
       1&
       1&
       1&
        1
      \\
       \hline
  test error &
    &&&   &
        0.0993
      \\
      \hline
\end{tabular}
\end{center}
\label{tab-d2_adaptive}
\end{table}%

\begin{figure}[!ht]
    \centering
   \begin{minipage}[!hbp]{0.33\linewidth}
        \includegraphics[width=0.99\textwidth,angle=0]{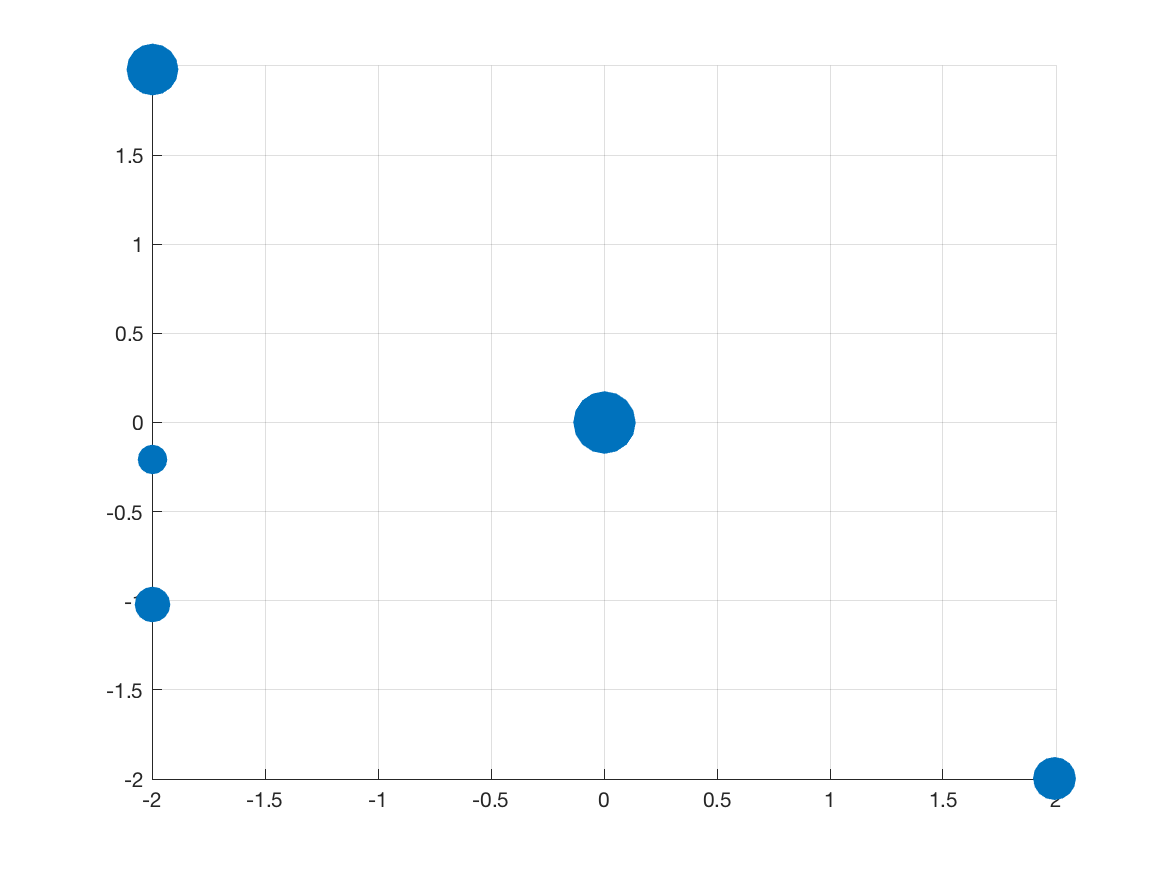}
        \end{minipage}%
        \quad
    \begin{minipage}[!htbp]{0.33\linewidth}
        \includegraphics[width=0.99\textwidth,angle=0]{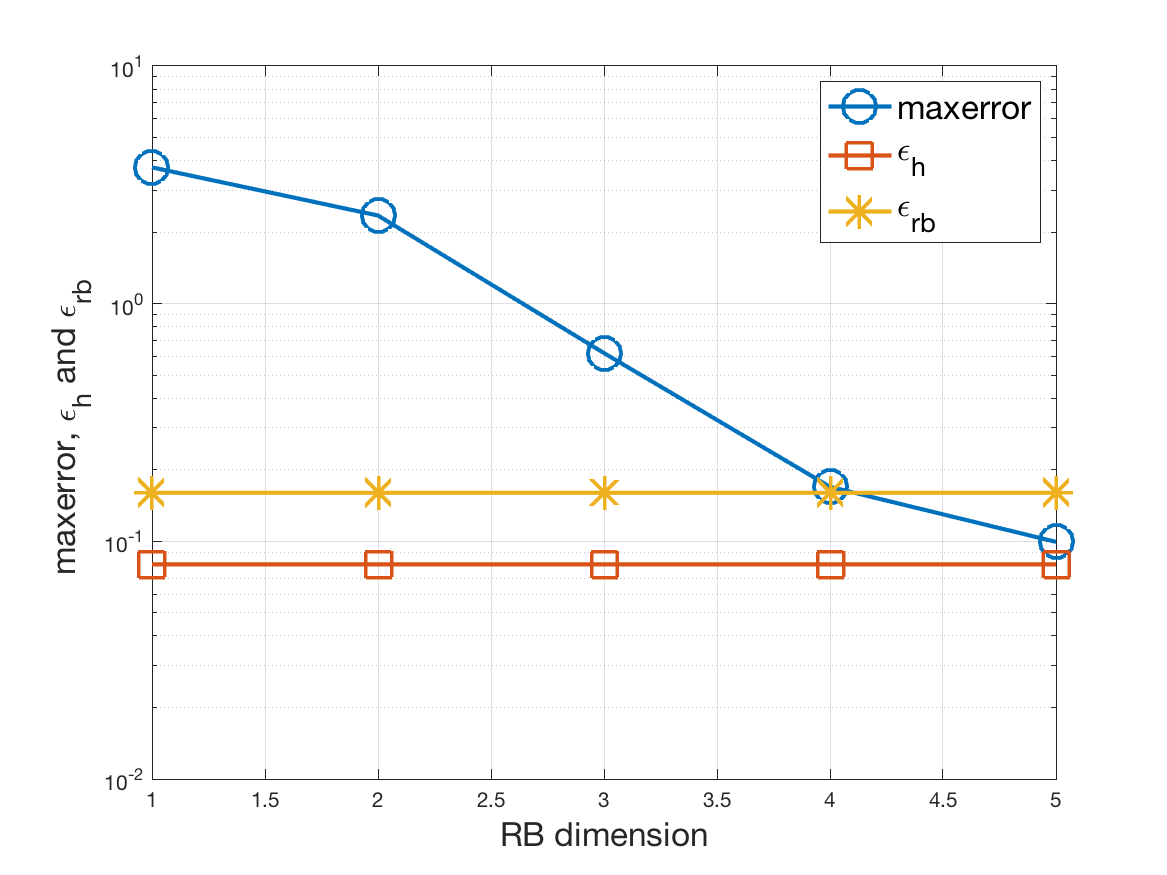}
        \end{minipage}
        \caption{Selection of RB points $S_N$ (left) and convergence history of RB maxerror (right) for test diffusion problem on a adaptive mesh with PD-RBM and  Algorithm 2 ($\mathtt {r_{rb,fe}}=2$)}%
        \label{d_2d_afem}
\end{figure}
%
%

The numerical test results can be found in Table \ref{tab-d2_adaptive} and Fig. \ref{d_2d_afem} for $\mathtt {r_{rb,fe}}=2$ ($\eps_{\rb} = 0.16$). We need $5$ RB snapshots to reduce {\tt maxerror} to be less than $\eps_{\rb}=0.16$. From the 8th row of table, we see that the mesh is refined for $5$ times (including the mesh generated for $\mu_1$). We show those meshes on  Fig. \ref{d_2d_afem_mesh}. The latter meshes are the refinements of previous ones. We can clearly see that refinements on different areas of the domain due to the features of solutions corresponding to later $\mu_i\in S_N$.

On the rows 5 and 6 of  Table \ref{tab-d2_adaptive}, we list the number of FE DOFs for the primal and dual problems. All the numbers are increasing due to mesh refinements. But compared to Table \ref{tab-d2}, we can clearly see that number of DOFs used with adaptive mesh refinements is much smaller than a fixed mesh for a similar (even smaller) $\eps_{\rb}$.
Also, we find that for the adaptive mesh greedy algorithm, the algorithm of saturation trick also saves a good amount of computation in the loop for all $\mu \in \Xi_{\mathtt{train}}$, see the 7th row of  Table \ref{tab-d2_adaptive}.
%

\begin{figure}[!ht]
    \centering
   \begin{minipage}[!hbp]{0.33\linewidth}
        \includegraphics[width=0.99\textwidth,angle=0]{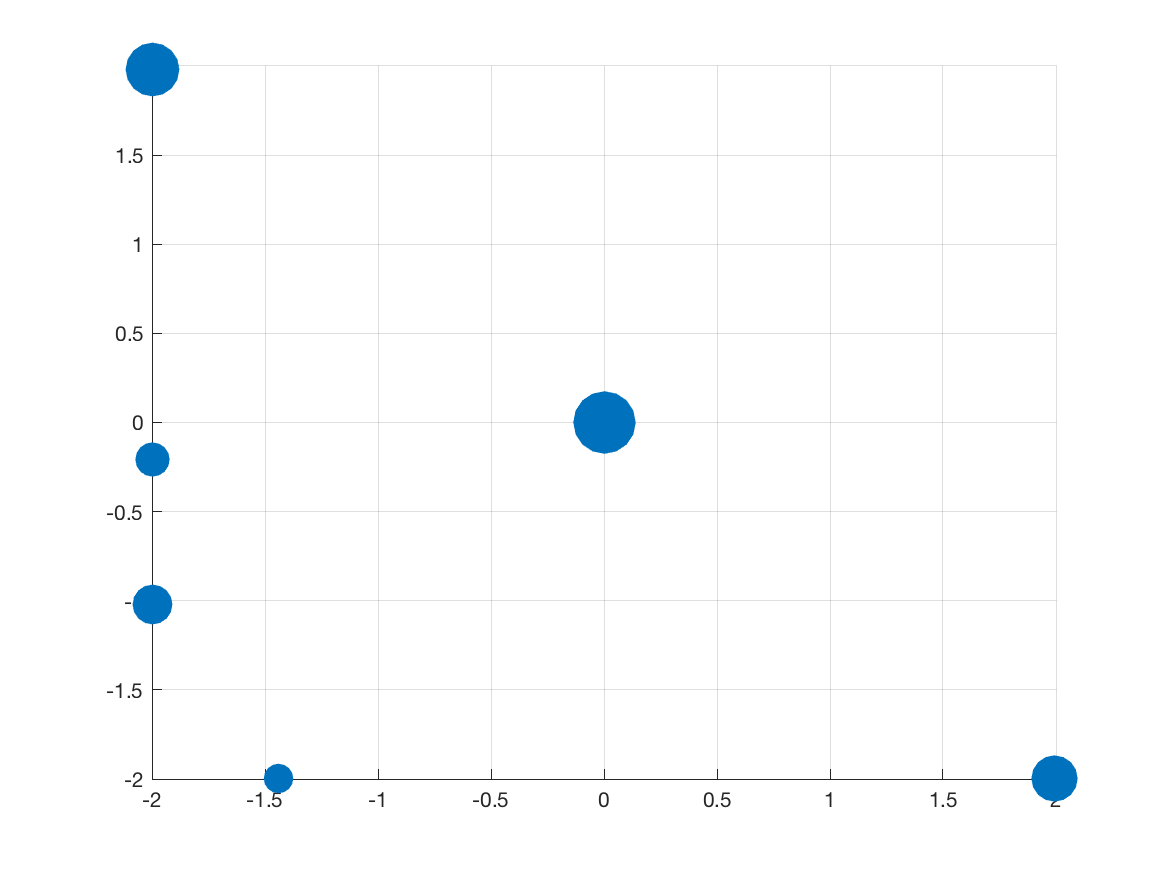}
        \end{minipage}%
        \quad
    \begin{minipage}[!htbp]{0.33\linewidth}
        \includegraphics[width=0.99\textwidth,angle=0]{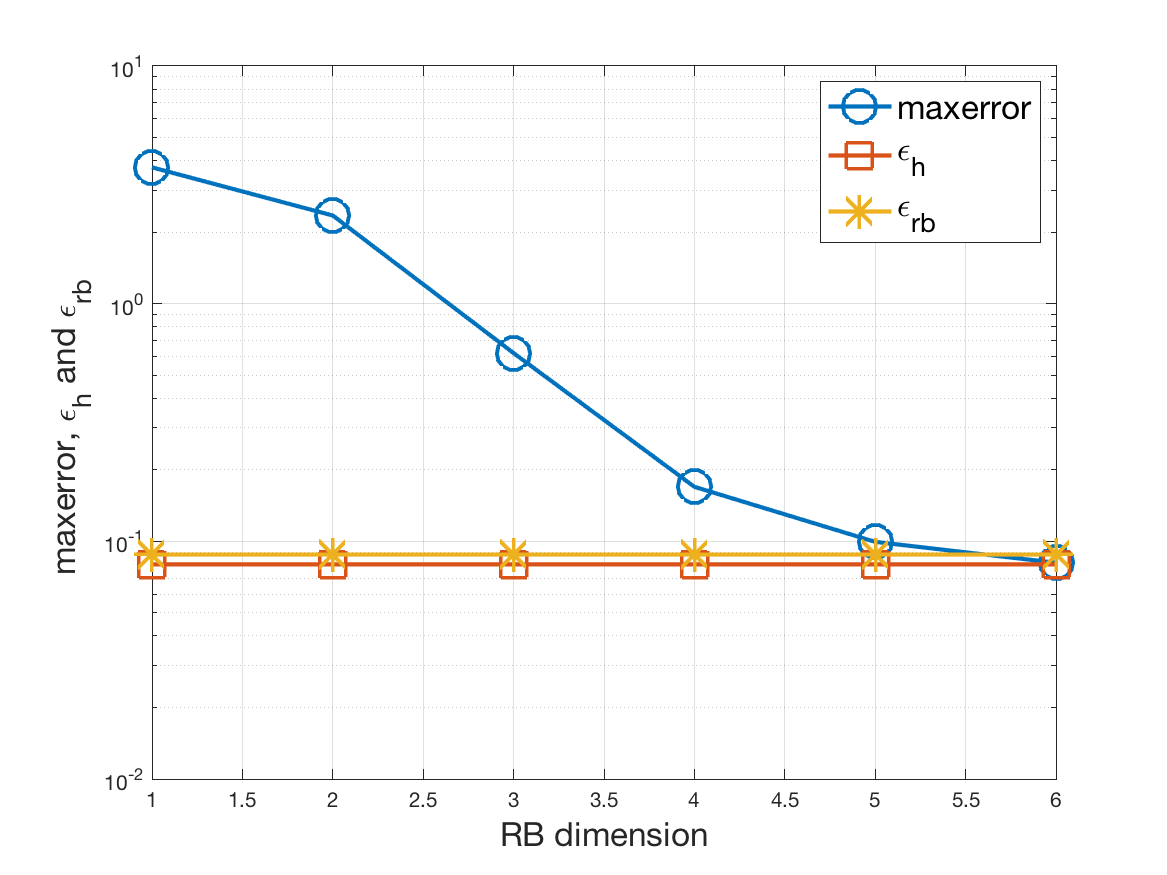}
        \end{minipage}
        \caption{Selection of RB points $S_N$ (left) and convergence history of RB maxerror (right) for test diffusion problem on a adaptive mesh with PD-RBM and Algorithm 2 ($\mathtt {r_{rb,fe}}=1.1$)}%
        \label{d_2d_afem11}
\end{figure}

The numerical test results for $\eps_h = 0.08$ and $\eps_{\rb} = 0.88$ ($\mathtt {r_{rb,fe}}=1.1$)  can be found in Fig. \ref{d_2d_afem11}. We need $6$ RB snapshots to reduce {\tt maxerror} to be less than $\eps_{\rb}=0.88$. 

\subsection{A bi-adaptive greedy algorithm with a balanced mesh for the whole RB parameter set}
In this greedy algorithm, under the situation that the computational resource may be limited, 
we plan to do both the spatial and RB tolerance adaptivities and to generate a mesh which is balanced for the whole RB parameter set $S_N$.
It is clear that different parameters $\mu$ may require very different FE meshes. 
In Fig. \ref{mesh_d_diff_para}, we show three meshes and  corresponding solutions with parameter points $(0,0)$, $(-2,2)$, and $(2,-2)$ generated by the PD-FE error estimator $\eta_{h}$ until the $\eta_{h} \leq 0.1$ for the diffusion equation. 
%

\begin{figure}[!ht]
    \centering
  \begin{minipage}[!hbp]{0.33\linewidth}
        \includegraphics[width=0.99\textwidth,angle=0]{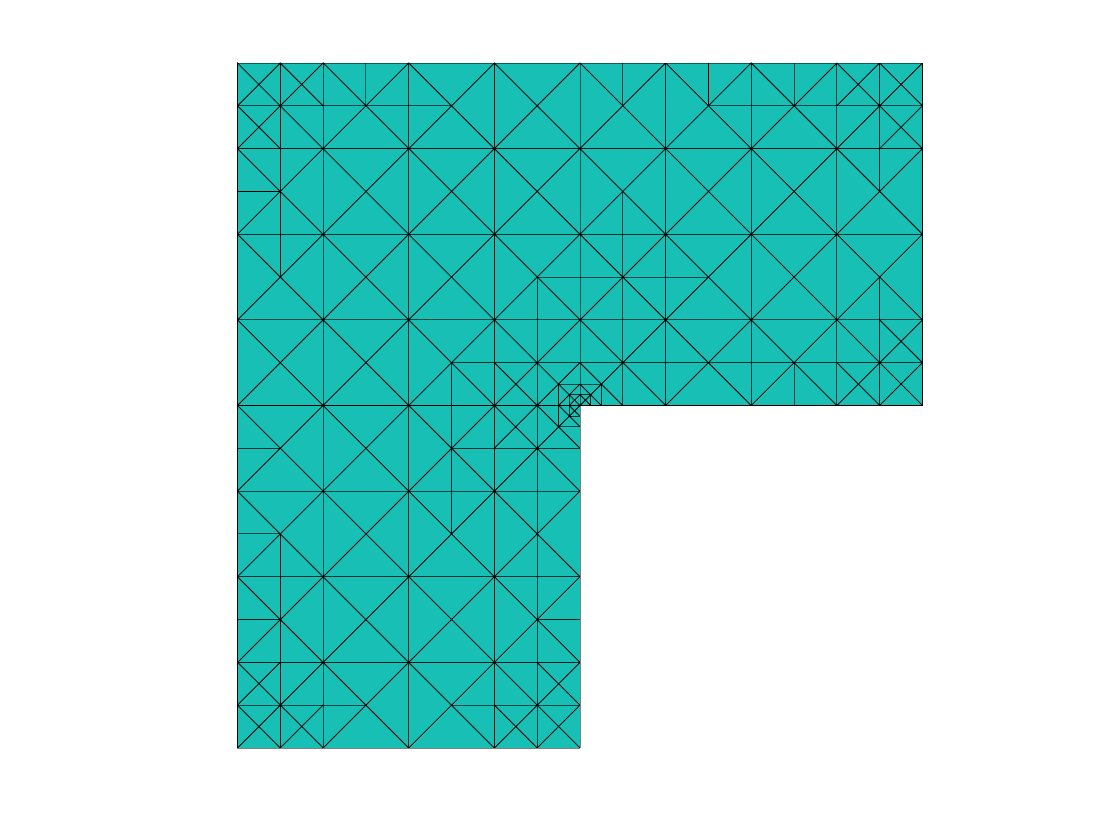}
        \end{minipage}%
    \begin{minipage}[!htbp]{0.33\linewidth}
        \includegraphics[width=0.99\textwidth,angle=0]{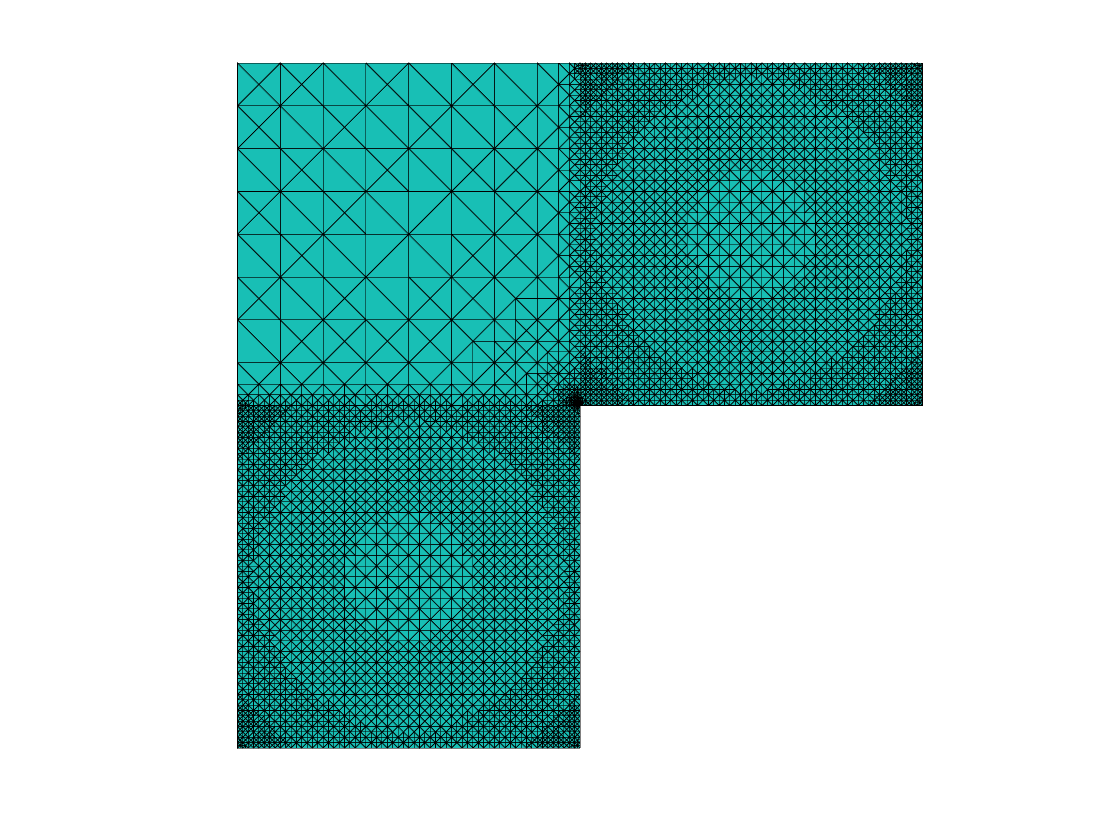}
        \end{minipage}
    \begin{minipage}[!htbp]{0.33\linewidth}
        \includegraphics[width=0.99\textwidth,angle=0]{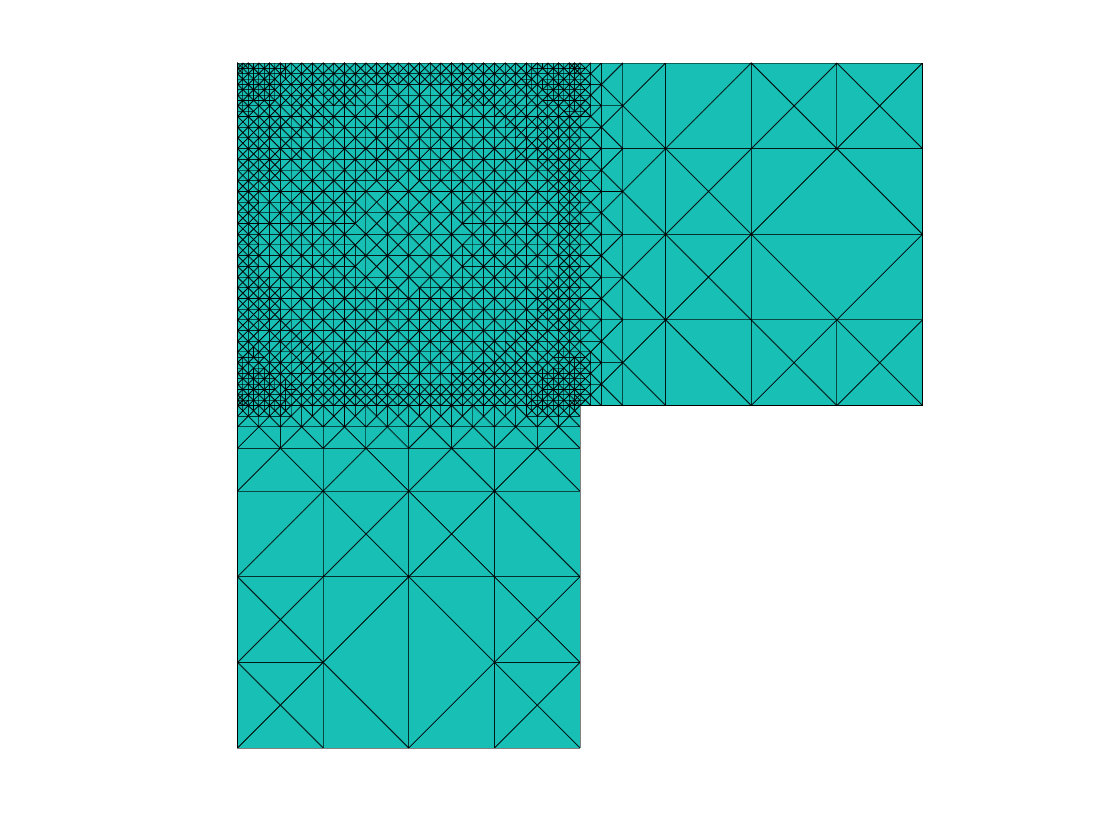}
        \end{minipage}\\
        \centering
            \begin{minipage}[!htbp]{0.32\linewidth}
        \includegraphics[width=0.99\textwidth,angle=0]{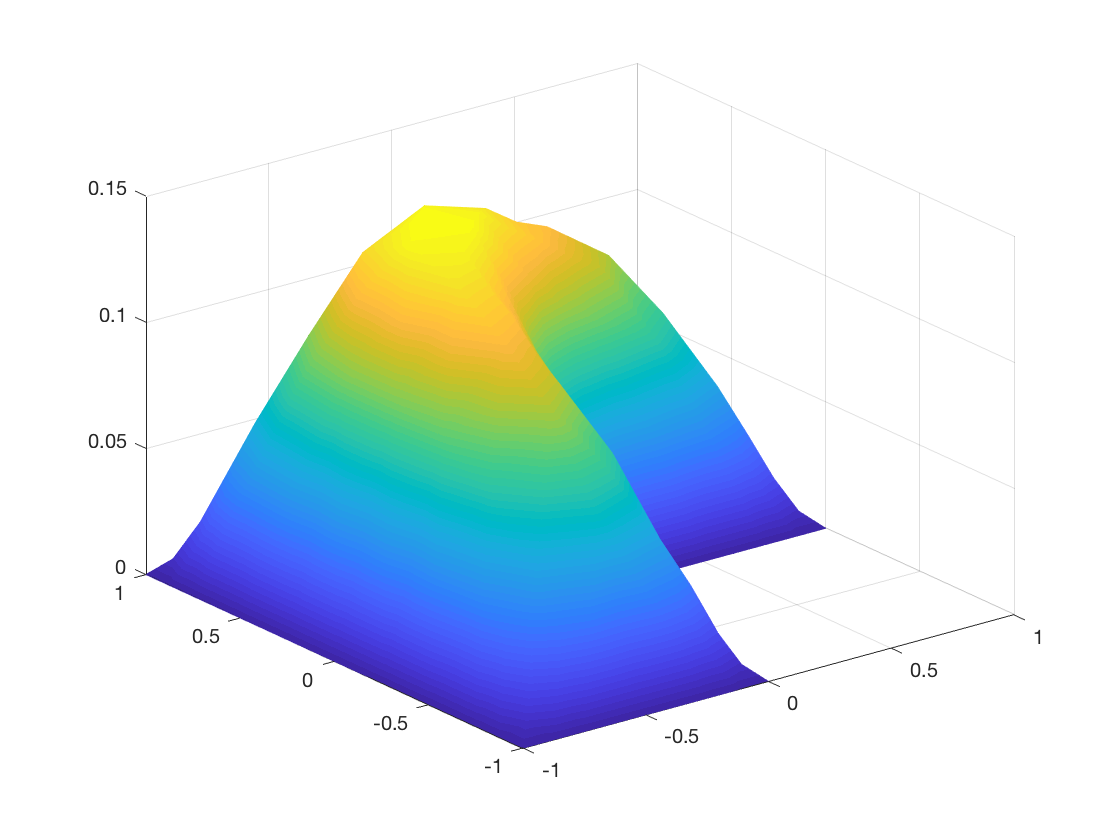}
        \end{minipage}
        \begin{minipage}[!htbp]{0.32\linewidth}
        \includegraphics[width=0.99\textwidth,angle=0]{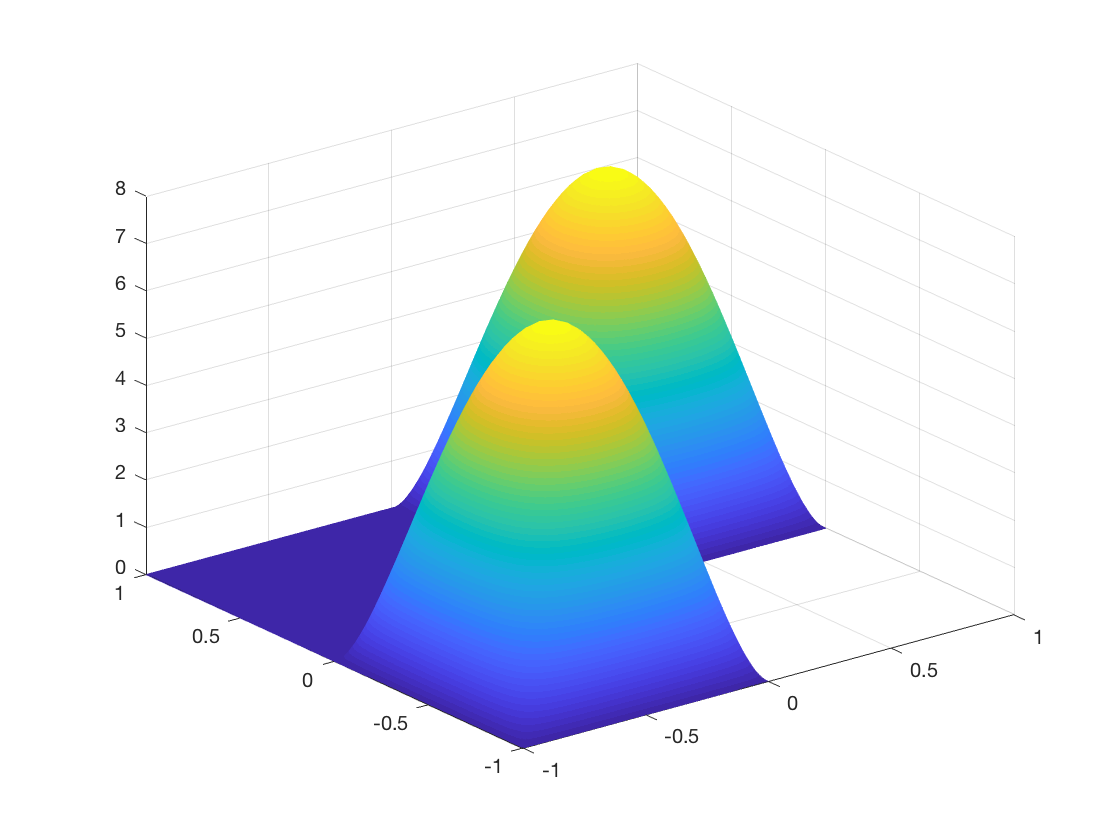}
        \end{minipage}
            \begin{minipage}[!htbp]{0.32\linewidth}
        \includegraphics[width=0.99\textwidth,angle=0]{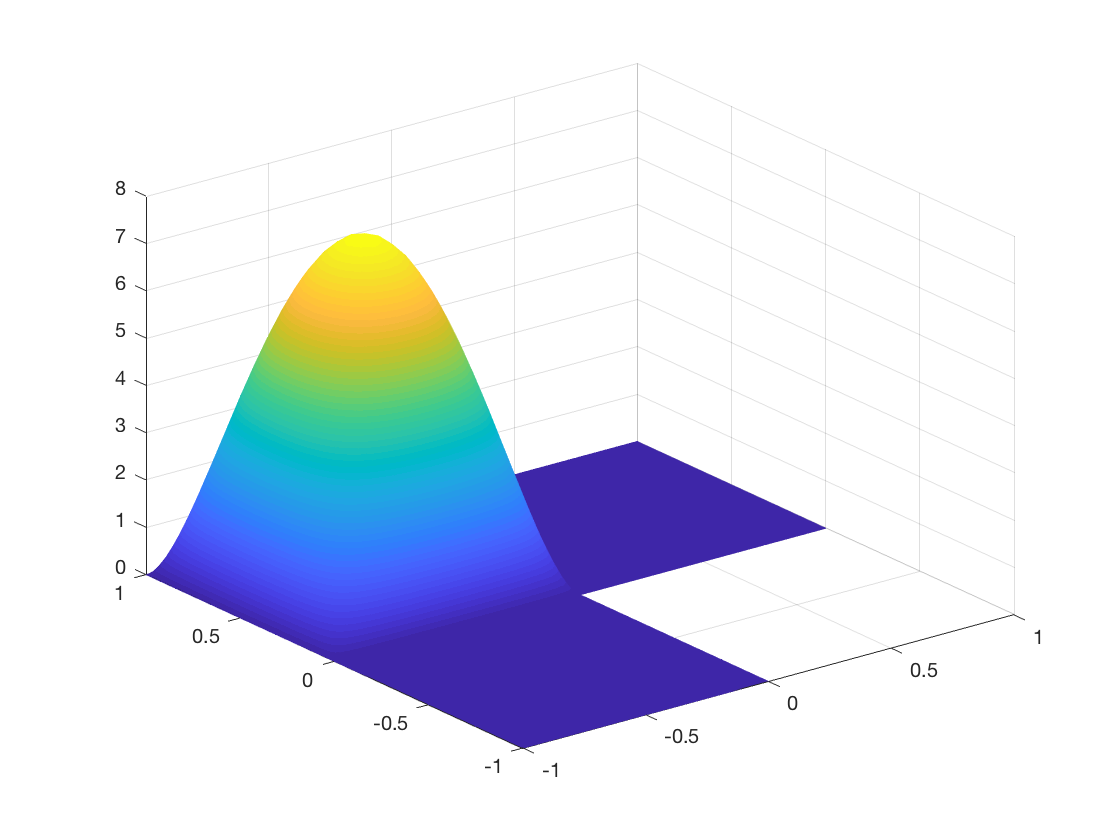}
        \end{minipage}
        \caption{ 
Different adaptively refined meshes for different parameters for the test diffusion equation: the 1st column, $\mu = (0,0)$, the 2nd column, $\mu = (-2,2)$, the 3rd column, $\mu = (2,-2)$}%
        \label{mesh_d_diff_para}
\end{figure}

In practical computations, the computational resource is not unlimited. Suppose the maximum number of DOFs that the FE solver can handle is $\cN_{\fe,\max}$. When an adaptive mesh refinement method is used, we cannot use more DOFs than $\cN_{\fe,\max}$. One naive method is that we just stop refining the mesh when $\cN_{\fe,\max}$ is reached and treat it as a fixed mesh algorithm as Algorithm 1. This approach has a problem due to that different parameters have their own optimal meshes. If the mesh with DOFs $\approx \cN_{\fe,\max}$ is  optimal for $\mu_1=(2,-2)$ and no further mesh refinements can be done. Then for  parameters  $(-2,2)$ or $(-2,-2)$, the mesh is not good. Thus, we need an adaptive mesh refinement algorithm which is balanced for the whole parameter set $S_N = \{\mu_i\}_{i=1}^N$ with a maximum number of DOFs $\cN_{\fe,\max}$.

%

%

It is important to recall that the underlying principle of adaptive mesh refinement is the so-called "equi-distribution" of the error, i.e., each element has a similar size of the error. Note that we have the following PD-sum energy error relation:
\begin{eqnarray} \label{sum_err}
\sum_{i=1}^N\cnorm{(u(\mu_i)-u_h(\mu_i),\sigma(\mu_i)-\sigma_h(\mu_i))}_{\mu_i}^2 \\ \nonumber
= \sum_{i=1}^N \Jp(u_h(\mu_i);\mu_i) - \Jd(\sigma_h(\mu_i);\mu_i) 
= \sum_{i=1}^N \eta_{h}(u_h,\sigma_h;\mu_i) ^2,
\end{eqnarray}
where $\mu_i\in S_N$. Thus, we define the PD-sum error indicator:
$$
\eta_{h,T}(u_h,\sigma_h; S_N) = \left(
\sum_{i=1}^N \eta_{h,T}(u_h,\sigma_h;\mu_i)^2
\right)^{1/2}.
$$
The PD-sum error indicator can be used to drive the mesh refinement to minimize the sum energy error  
$
(\sum_{i=1}^N\cnorm{(u(\mu_i)-u_h(\mu_i),\sigma(\mu_i)-\sigma_h(\mu_i))}_{\mu_i}^2)^{1/2}$.
With this error indicator, the mesh it generates equally distributes the PD-sum energy error, thus  it is balanced for the reduction of sum energy error for the whole $S_N$.
%

We explain the spatial and RB tolerance bi-adaptive greedy algorithm with a balanced mesh for the whole $S_N$ under the condition that number of DOFs of the FE mesh can not exceeds $\cN_{\fe,\max}$ as follows.

On each loop of the greedy algorithm, we start from a FE mesh from the previous iteration $\cT_{N-1}$, which is already good for $S_{N-1}$ and a FE tolerance $\eps_h^{N-1}$.  Now adaptively compute FE solutions $u_h(\mu_N)$ and $\sigma_h(\mu_N)$ with $\eta_{h,T}(u_h,\sigma_h;\mu_N)$ as the FE error indicator, compute the {\em best possible} solutions $\eta_{h,T}(u_h,\sigma_h;\mu)\leq \eps_h^{N-1}$ and the number of FE DOFs $\leq \mathcal{N}_{\fe,\max}$. The {\em best possible} means that if $\mathcal{N}_{\fe,\max}$ is big enough to ensure that $\eta_{h,T}(u_h,\sigma_h;\mu_N)\leq \eps_h^{N-1}$, we will use an algorithm similar to Algorithm 3 to generate a mesh $\cT_N$ (it is also possible that no refinement from $\cT_{N-1}$ is needed); if $\mathcal{N}_{\fe,\max}$ is not big enough to ensure $\eta_{h,T}(u_h,\sigma_h;\mu_N)\leq \eps_h^{N-1}$, we use as many DOFs as we can to generate a new mesh with the number of FE DOFs $\leq \cN_{\fe,\max}$.

Now we have two situations. If the $\cN_{\fe,\max}$ is enough to ensure all RB snapshots for $\mu_i\in S_N$ can be computed with the previous level FE tolerance $\eps_{h}^{N-1}$, we just need to let $\eps_{h}^{N} = \eps_{h}^{N-1}$ and $\eps_{\rb}^{N} = \eps_{\rb}^{N-1}$ and add  $u_h(\mu_N)$ and $\sigma_h(\mu_N)$ to the corresponding RB sets.

On the other hand, if the $\cN_{\fe,\max}$ is not enough to ensure all RB snapshots for $\mu_i\in S_N$ can be computed with the previous level FE tolerance $\eps_{h}^{N-1}$, we need to regenerate a new mesh which is good for all $S_N$. Starting from an initial mesh $\cT_{0}$, we adaptively compute FE solutions $u_h(\mu_i)$ and $\sigma_h(\mu_i)$, $i=1, \cdots, N$, with the sum error estimator $\eta_{h,T}(u_h,\sigma_h;S_N)$ as the FE error indicator,  and compute best possible solution with $\eps_h^{N-1}$ and $\cN_{\fe,\max}$. The final mesh is $\cT_N$. Note that for each $\mu_i$, the FE problems are solved independently, thus, the number of DOFs of each individually problem is still smaller than $\leq \cN_{\fe,\max}$. 

Define the new FE and RB tolerances to be:
$$
\eps_h^N = \max\{\eps_h^{N-1},\eta(u_h,\sigma_h;\mu_i), i=1:N\} 
\quad\mbox{and}\quad \eps_{\rb}^N: = \rferb\eps_h^N.
$$
That is, $\eps_h^N$ is the worst FE error of all $\mu_i\in S_N$ on $\cT_N$.

For the algorithm of saturation trick, the relation
$
\eta_{\rb}(u_{\rb}^N,\sigma_{\rb}^N;\mu) \leq \eta_{\rb}(u_{\rb}^{N-1},\sigma_{\rb}^{N-1};\mu)
$
is not always true since now the meshes $\cT_{N}$ and $\cT_{N-1}$ are different. But as we discussed in \cite{HSZ:14}, as long as the assumption that {\em the error estimator is smaller if the RB space is of a better quality} is true, we can use the algorithm of saturation trick. After all, we only need to find the {\em almost worst} candidate.  

We list the detailed algorithm in Algorithm \ref{greedy_combined}.

\begin{algorithm}[!ht]
\KwIn{
Training set $\Xi_{\mathtt{train}} \subset \cD$, 
initial FE accuracy tolerance $\eps_h^0 >0$, 
some $\rferb >1$,
initial RB tolerance $\eps_{\rb}^{0} = \rferb \eps_h^0$, 
maximal RB dimension $N_{\max}$, 
maximum number of FE DOFs $\mathcal{N}_{\fe,\max}$,
initial coarse mesh $\cT_0$ (whose number of DOFs $\leq \mathcal{N}_{\fe,\max}$), 
$\mu_1$, $N=1$, 
flag ${\mathtt{enough}}=1$,
flag ${\mathtt{refined}}=0$.
}
\KwOut{
$N$, $\eps_h^N$, $\eps_{\rb}^N$, $V^{\rb}_N$, $\Sigma^{\rb}_N$, $\cT_N$.
}
\BlankLine
\begin{algorithmic} [1]
\STATE Starting from mesh $\cT_{N-1}$, adaptively compute FE solutions $u_h(\mu_N)$ and $\sigma_h(\mu_N)$
with $\eta_{h,T}(u_h,\sigma_h;\mu_N)$ as the FE error indicator, compute best possible solutions 
$\eta_{h,T}(u_h,\sigma_h;\mu)\leq \eps_h^{N-1}$ and the number of FE DOFs $\leq \mathcal{N}_{\fe,\max}$ .
The final mesh is $\cT_N$. Mark ${\mathtt{refined}}=1$ if the mesh is refined, otherwise  ${\mathtt{refined}}=0$;

If  $\mathcal{N}_{\fe,\max}$ is not enough to make $\eta_{h,T}(u_h,\sigma_h;\mu)\leq \eps_h^{N-1}$,
${\mathtt{enough}}=0$.

 \IF  {${\mathtt{enough}}=1$}	
 \STATE $\eps_h^{N} = \eps_h^{N-1}$, $\eps_{\rb}^N =  \rferb \eps_h^N$.
 \STATE If {${\mathtt{refined}}=1$}, 
 recompute $u_h(\mu_i)$ and $\sigma_h(\mu_i)$, $i=1, \cdots, N-1$
 on the current mesh
$\cT_N$. 
\STATE Update $V^N_{\rb} $ and $\Sigma^N_{\rb}$.

%

 	\ELSE 

\STATE Starting from mesh $\cT_{0}$, adaptively compute FE solutions $u_h(\mu_i)$ and $\sigma_h(\mu_i)$, $i=1, \cdots, N$,
with $\eta_{h,T}(u_h,\sigma_h;S_N)$ as the FE error indicator, 
compute best possible solution with $\eps_h^{N-1}$ and $\cN_{\fe,\max}$.
The final mesh is $\cT_N$. {\tt{refined}}$=0$.

\STATE Update $V^N_{\rb} $ and $\Sigma^N_{\rb}$.

\STATE Update  $\eps_h^N = \max\{\eps_h^{N-1},\eta(u_h,\sigma_h;\mu_i), i=1:N\}$ 
and $\eps_{\rb}^N = \rferb \eps_h^N$.
\ENDIF

\STATE 
Choose 
$\mu_{N+1} = \mbox{argmax}_{\mu\in \Xi_{\mathtt{train}}} \eta_{\rb}(u_{\rb}^N,\sigma_{\rb}^N;\mu).
$

\STATE If $\eta_{\rb}(u_{\rb}^N,\sigma_{\rb}^N;\mu_{N+1}) > \eps_{\rb}^N$ and $N<N_{\max}$, then set $N=N+1$, {\bf goto} 1. Otherwise, {\bf terminate}. 

\end{algorithmic}\caption{A primal-dual exact energy error certified RB greedy algorithm with bi-adaptivity and a balanced mesh for the RB parameter set $S_N$ (a bi-adaptivity balanced greedy algorithm)}
\label{greedy_combined}
\end{algorithm}

\begin{rem}
Instead of working on a big super-mesh for the whole $S_N$, the other possibility is to construct a common mesh (as in \cite{URL:16}) for any two parameters with the PD-sum error estimator for two.
\end{rem}
\subsection{Numerical experiments for Algorithm 3} In this subsection, we test PD-RBM for the test diffusion  problem with Algorithm 3.

\noindent ({\bf Setting for the numerical experiment for Algorithm 3})
{\em 
Choose the initial $\eps_h^0 = 0.05$ and $\cN_{\fe,\max} = 2\times 10^4$. The  rest is the same as the setting for the numerical experiment for Algorithm 2.
}

We first test the case $\mathtt {r_{rb,fe}}=2$ ($\eps_{\rb}^0= 0.1$). The numerical results can be found in Table \ref{tab-d2_com} and Fig. \ref{d_2d_com}. We need $4$ RB snapshots to balance the RB dimension, $\eps_{\rb}$, and $\cN_{\fe,\max}$.

From the 10th row of table, we see that the $\cN_{\fe,\max}$ was not enough in $3$ occasions. For a smaller RB dimension $N$, $\eps_{\fe}^N$ and $\eps_{\rb}^N$ are smaller, so we need more DOFs for each $\mu_i$ to satisfy the error requirement, but since $N$ is smaller, the situation that different $\mu_i$ needs very different meshes is less severe. For a bigger RB dimension $N$,  the situation is different, more $\mu_i$ in $S_N$ but a bigger FE tolerance $\eps_h$. Thus it is a dynamic balance of the dimension $N$ and the FE and RB tolerances. 
We show the four refined meshes corresponding to the first adaptive mesh and the other $3$ non-enough cases on  Fig. \ref{d_2d_com_mesh}. It is easy to see that the last two meshes are more balanced for all the $\mu_i\in S_N$, which is quite similar to the combination of all three meshes depicted in Fig. \ref{mesh_d_diff_para}.
%

\begin{figure}[!ht]
    \centering
   \begin{minipage}[!hbp]{0.33\linewidth}
        \includegraphics[width=0.99\textwidth,angle=0]{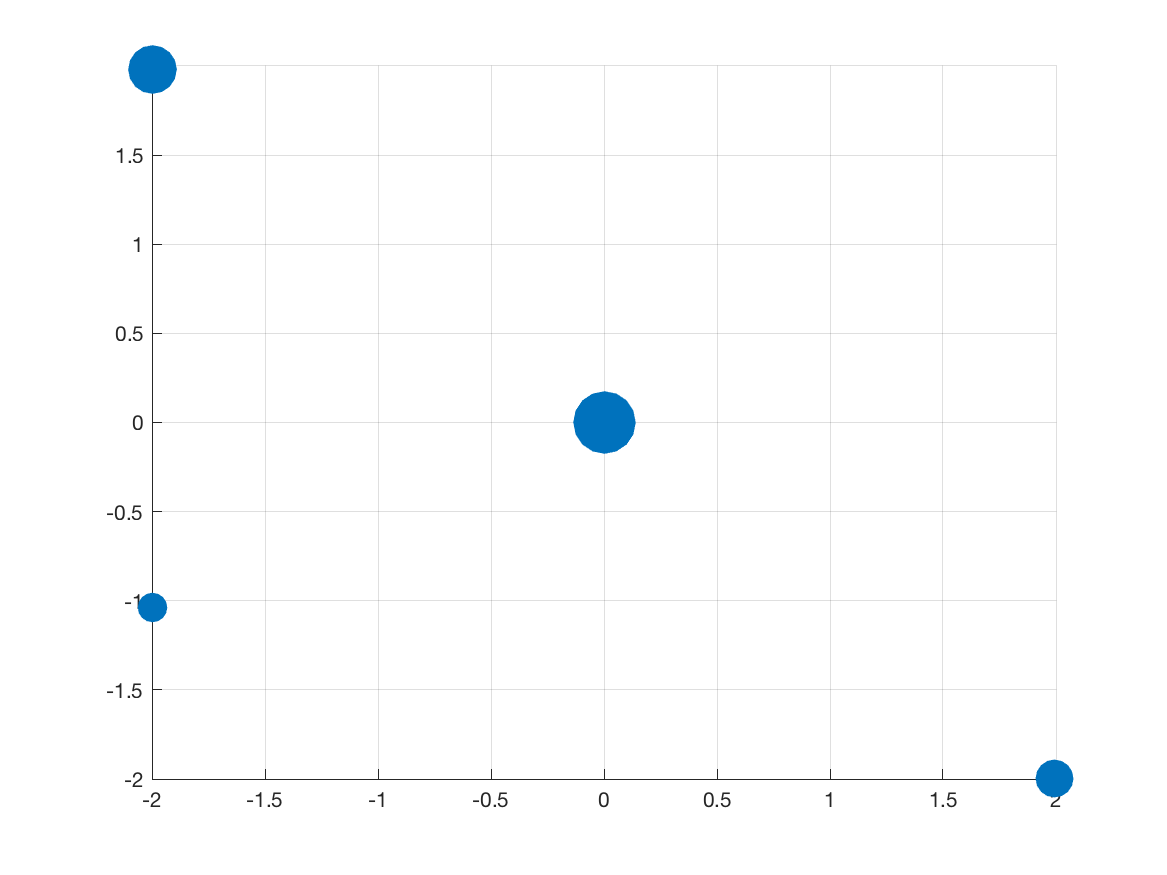}
        \end{minipage}%
        \quad
    \begin{minipage}[!htbp]{0.33\linewidth}
        \includegraphics[width=0.99\textwidth,angle=0]{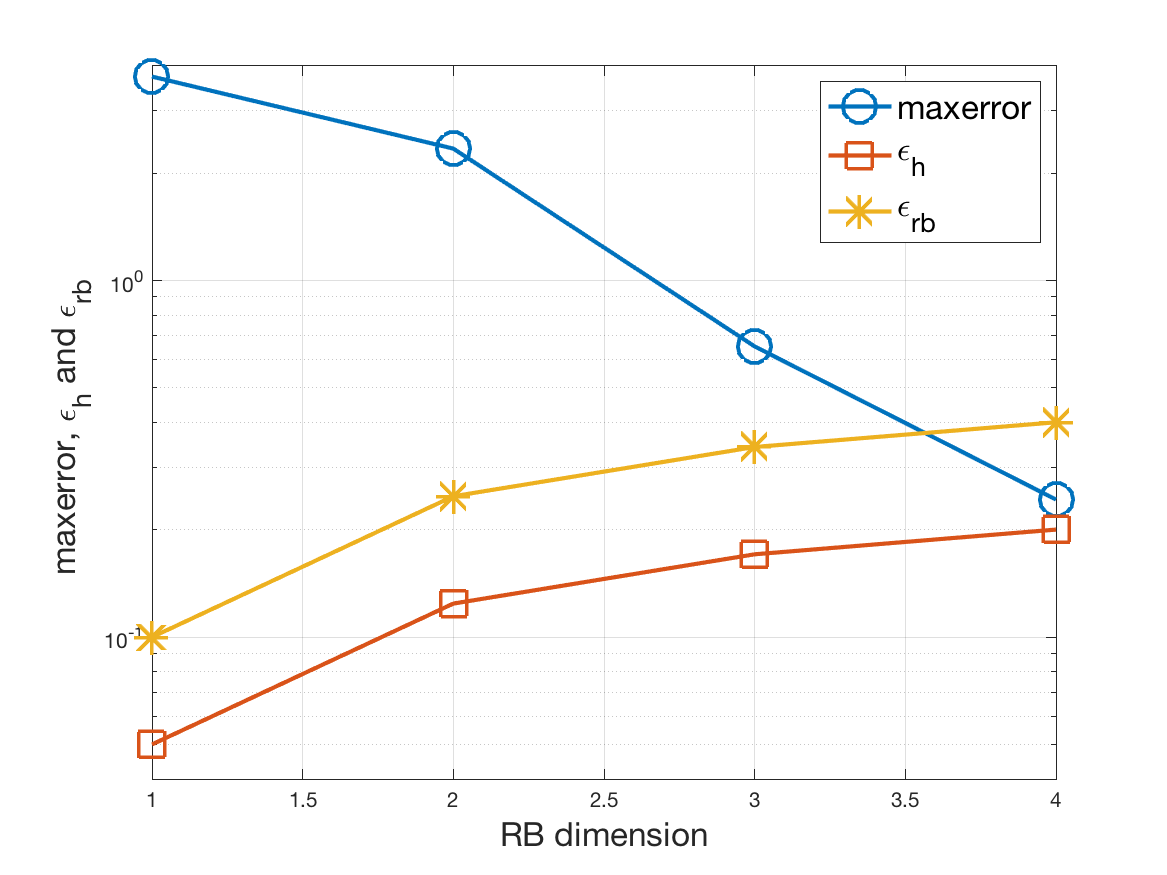}
        \end{minipage}
        \caption{Selection of RB points $S_N$ (left) and convergence history of RB maxerror (right) for test diffusion problem with PD-RBM and Algorithm 3 ($\mathtt {r_{rb,fe}}=2$)}%
        \label{d_2d_com}
\end{figure}

\begin{figure}[!ht]
    \centering
   \begin{minipage}[!hbp]{0.33\linewidth}
        \includegraphics[width=0.99\textwidth,angle=0]{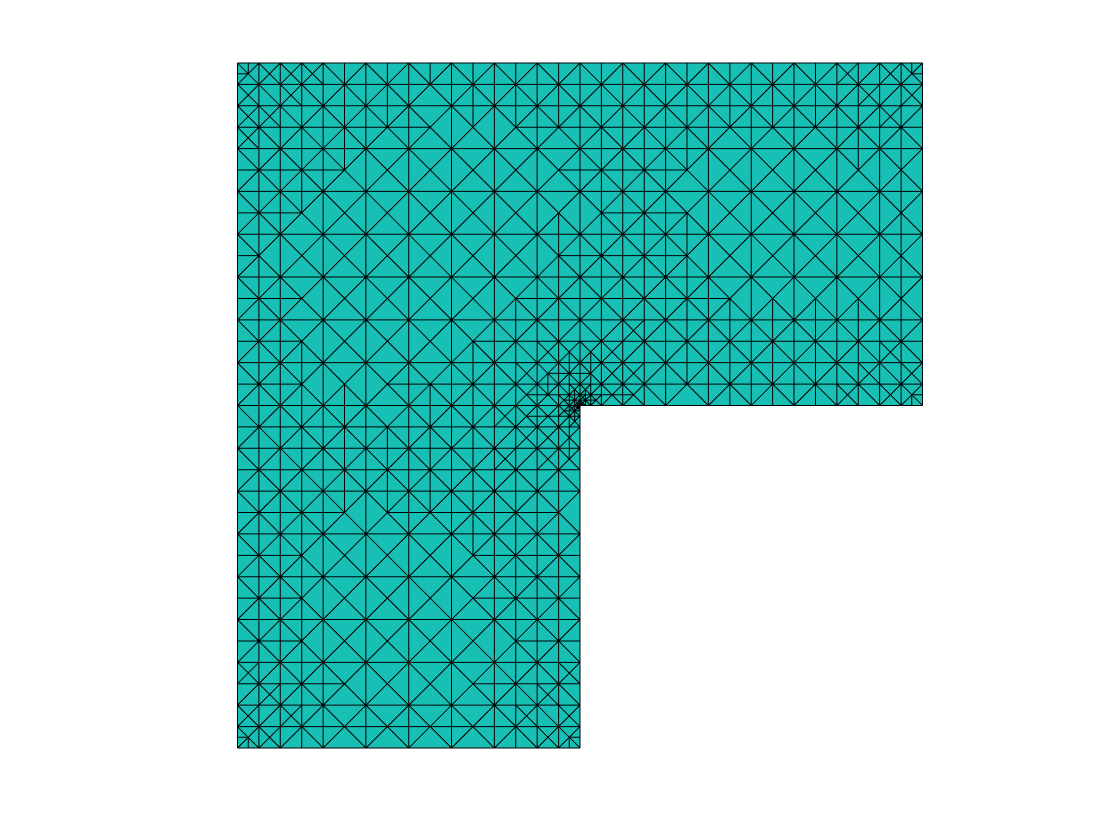}
        \end{minipage}%
        \quad
    \begin{minipage}[!htbp]{0.33\linewidth}
        \includegraphics[width=0.99\textwidth,angle=0]{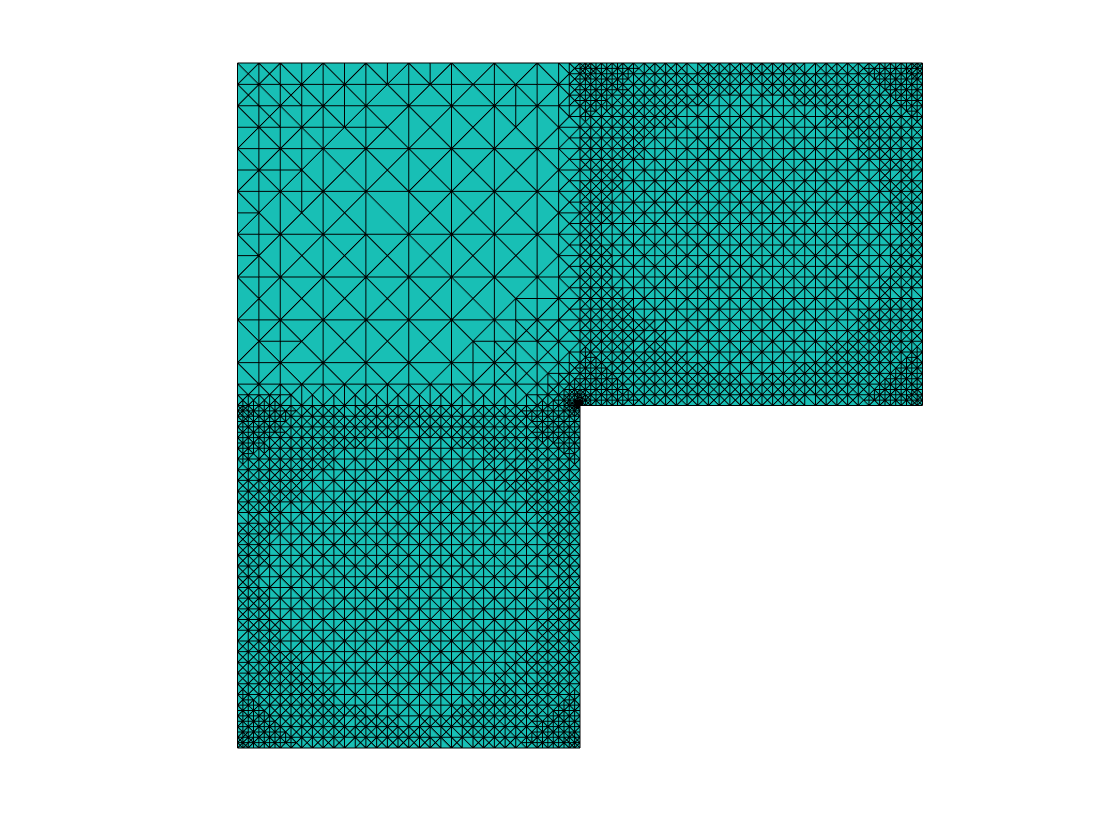}
        \end{minipage}
        \\
        \centering
   \begin{minipage}[!hbp]{0.33\linewidth}
        \includegraphics[width=0.99\textwidth,angle=0]{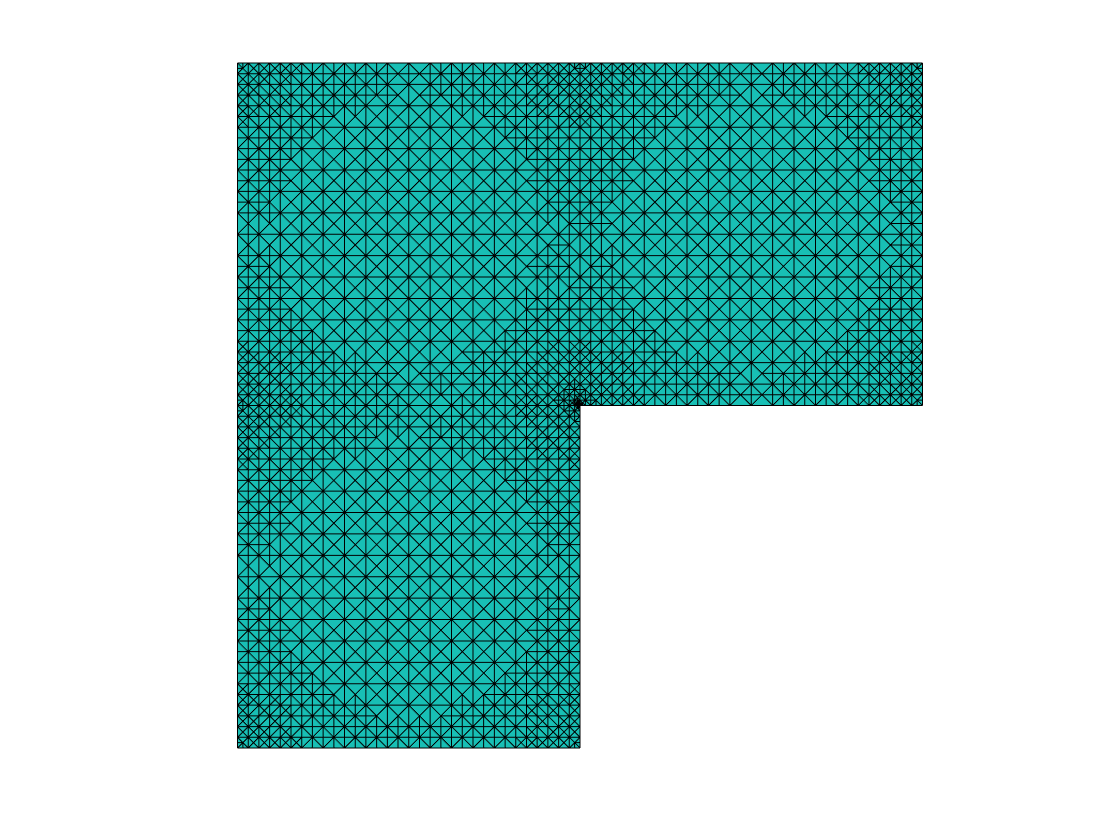}
        \end{minipage}%
        \quad
    \begin{minipage}[!htbp]{0.33\linewidth}
        \includegraphics[width=0.99\textwidth,angle=0]{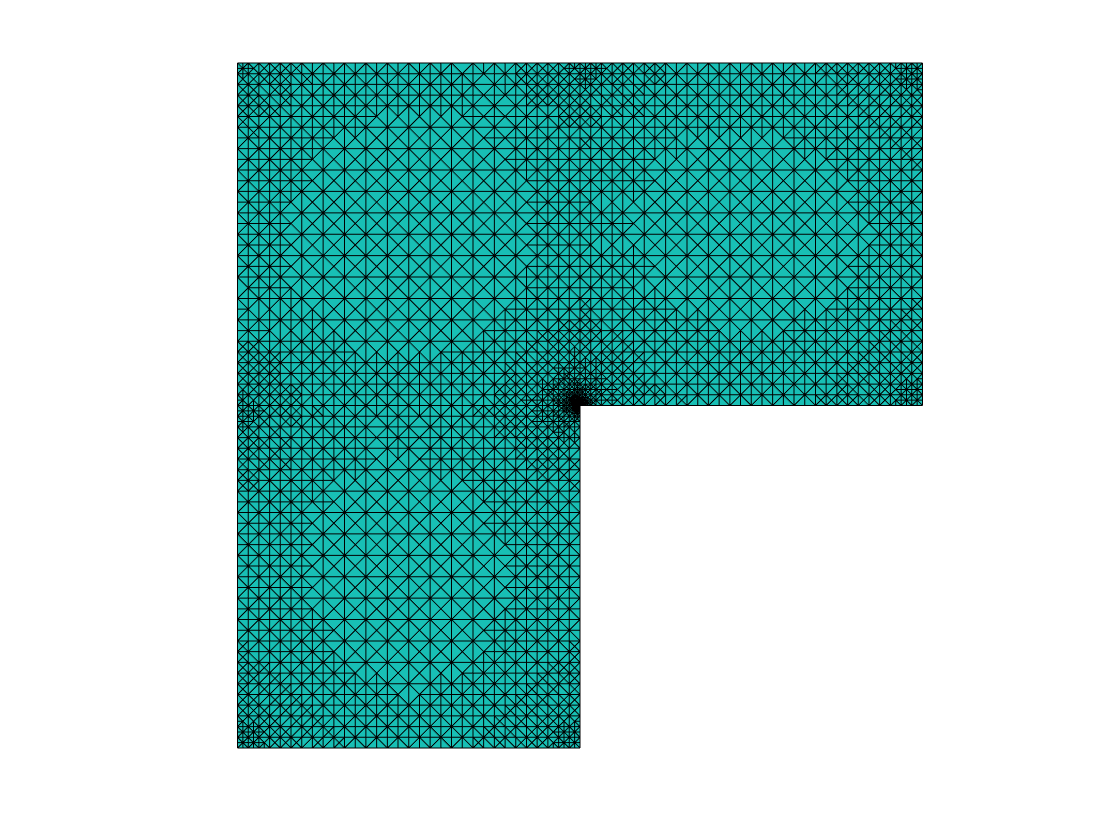}
        \end{minipage}
        
        \caption{Four refined meshes generated by PD-RBM and Algorithm 3 for the test diffusion equation ($\mathtt {r_{rb,fe}}=2$)}%
        \label{d_2d_com_mesh}
\end{figure}

\begin{table}[!ht]
\caption{Convergence results for PD-RBM with Algorithm 3 for test diffusion equation ($\mathtt {r_{rb,fe}}=2$)}
\begin{center}
\begin{tabular}{|c|r|r|r|r|r|r|r|r|r|}
 \hline
       RB Dim  & 1 & 2 & 3& 4\\
 \hline
$\mu_{[1]}$ &      
   0.0000 &
     -1.9996&
    1.9936&
   -1.9973\\
    \hline
    $\mu_{[2]}$ &      
   0.0000 &
    1.9808&
   -1.9999&
   -1.0361\\
    \hline
    $\eps_{h}$ &      
  0.0500&
    0.1242&
    0.1707&
    0.2003\\
 \hline
     $\eps_{\rb}$ &      
     0.1000&
    0.2484&
    0.3415&
    0.4006\\
 \hline
maxerror & 
3.7357&
    2.3461&
    0.6537&
    0.2432\\
 \hline
 $\cN_{\fe,p}$ &  
   814&
        3791&
        2675&
        3694\\
 \hline
  $\cN_{\fe,d}$ &
  3793&
       18398&
       12844&
       17821\\
 \hline
  skipped &  
  		0&  90624&
       56967&
       67516\\
 \hline
   enough &
     1&
     0&
     0&
     0
     \\
     \hline
test error &   
 &&&       0.2424\\
        \hline
\end{tabular}
\end{center}
\label{tab-d2_com}
\end{table}%

The numerical test results can be found in  Fig. \ref{d_2d_com11} for $\mathtt {r_{rb,fe}}=1.1$. We need $5$ RB snapshots to balance the RB dimension, $\eps_{\rb}$, and $\cN_{\fe,\max}$.

\begin{figure}[!ht]
    \centering
   \begin{minipage}[!hbp]{0.33\linewidth}
        \includegraphics[width=0.99\textwidth,angle=0]{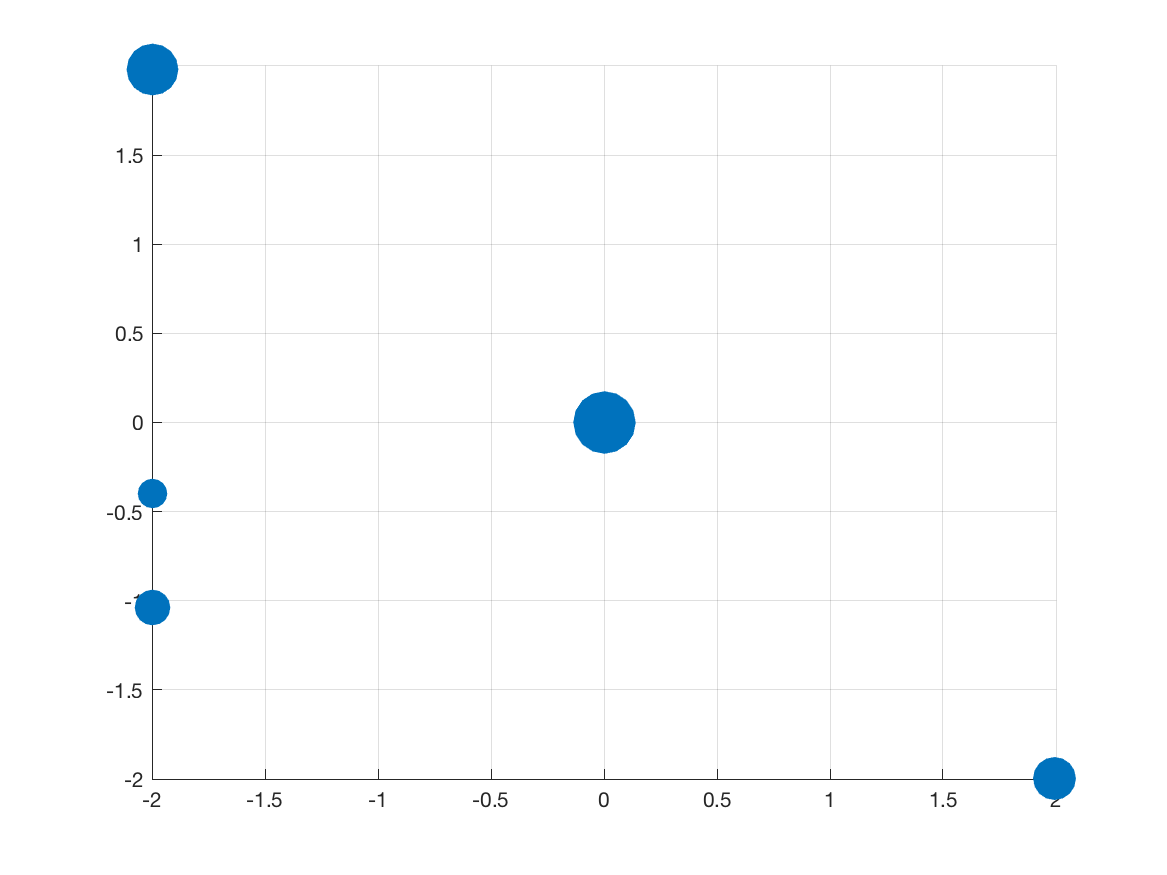}
        \end{minipage}%
        \quad
    \begin{minipage}[!htbp]{0.33\linewidth}
        \includegraphics[width=0.99\textwidth,angle=0]{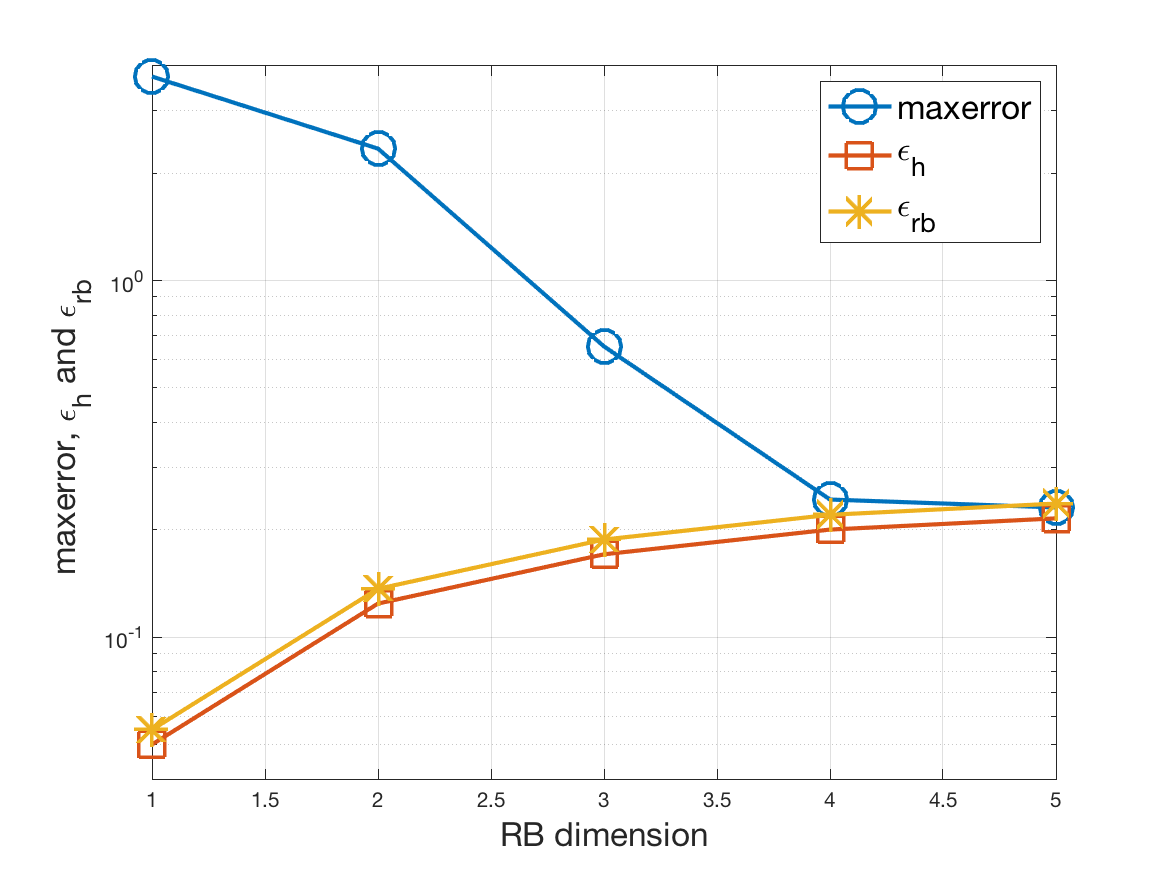}
        \end{minipage}
        \caption{Selection of RB points $S_N$ (left) and convergence history of RB maxerror (right) for test diffusion problem with PD-RBM and Algorithm 3 ($\mathtt {r_{rb,fe}}=1.1$)}%
        \label{d_2d_com11}
\end{figure}

\begin{rem}[Computational Cost]
Algorithm 3 is only used in the extreme case that the computational resource is not enough and Algorithm 2 is not applicable. If the computational resource is enough, the else part of Algorithm 3 will not be evoked and it is simply the Algorithm 2.

Based on the same observation that an adaptive mesh is much cheaper than  methods on a non-optimal mesh and the number of DOFs is under control, even though we need to re-compute the solutions on the new mesh in each step, this is still worthwhile. 
\end{rem}

\subsection{Remarks on numerical tests}
We can also use the relative energy RB error in the algorithms. 

Note that although we use quite different meshes and the algorithm of saturation trick in the greedy algorithms, the selection of the set $S_N$ and their orders are identical for all three algorithms. The algorithms of saturation trick also saves a large percent of offline loop in $\mu \in \Xi_{\mathtt{train}}$ for all three algorithms. 
The 10000 online random tests for all algorithms also show that the RB tolerances are well selected (in the sense the next RB selection will require more resources).
%

The algorithms defined here can be applied to other true-solution/true-residual certificated RBMs, for example, RBM based on least-squares principles \cite{Olson:20} with the least-squares energy functional/norm/estimator as the new energy functional/norm/estimator. For the standard RBM, it is also possible to apply the similar idea with different FE error estimators.


\section{Concluding Comments and Future Plans}
In this paper, with the parametric symmetric coercive elliptic equation as an example of a class of primal-dual problems with the strong duality, we develop primal-dual reduced basis methods with robust true error certification and adaptive balanced greedy algorithms.
	
	

Although we only discuss PD-RBM for the linear problem in the paper, we plan to apply the PD-RBM to other more challenging problems in the primal and dual variational framework, such as variational inequalities and nonlinear problems. Also, we plan to relax the requirements of primal-dual problems, for example, to use discontinuous approximations (see \cite{CHZ:20}), to relax the symmetric requirement (see \cite{Yano:17}), and others. 

As mentioned in Section 6.5, for many other problems not fitted in the framework, the reasonable choice is using the artificial minimization principle: the least-squares principle \cite{BG:09,Yano:16,Olson:20}. It is an ongoing work apply the greedy algorithms and ideas developed in this paper to construct the true residual certified least-squares RBM for a wide range of linear and nonlinear problems where the standard RBMs have difficulties, for example, transport equations \cite{LZ:18,LZ:19}.

In this paper, we focus on the robust energy certification only, it is also very important to develop robust certification for the quantity of interest with balanced mesh adaptivities for both the PD-RBM and the related least-squares RBM. We plan to work on this direction following two approaches, one is the approach based on the adjoint problems \cite{Yano:18} and the other one is  adding the output functional into the minimization functional \cite{Olson:14}.

\subsection*{Acknowledgments} The author sincerely thanks the anonymous referees for the helpful comments and valuable suggestions, which considerably improved the exposition of this work. 

\bibliographystyle{siamplain}
\bibliography{../bib/szhang}

\end{document}